\documentclass[10pt]{amsart}

\usepackage[pdftex]{graphicx} 
\usepackage[usenames,dvipsnames]{xcolor}

\usepackage[english]{babel} \usepackage[utf8]{inputenc}

\usepackage[centering,margin=2.4847cm]{geometry}
\usepackage{upgreek}
\usepackage{amsfonts}
\usepackage[font=small,labelfont=bf]{caption}
\usepackage{tikz}
\usepackage{setspace}
\usepackage{tensor}
\usepackage{lmodern}
\usepackage{verbatim}
\usepackage{amsmath}
\usepackage[all, cmtip]{xy}\usepackage{amssymb}
\usepackage{amsthm}
\usepackage{extpfeil}
\usepackage{bbm}
\usepackage{paralist}
\usepackage[colorlinks,citecolor=magenta,pagebackref=true,urlcolor=magenta,pdftex]{hyperref}
\usepackage{cancel}

\usepackage{soul}

\usepackage[T1]{fontenc}

\usepackage{nicefrac}
\usepackage{microtype}
\usepackage{mathrsfs}

\makeatletter
\newtheorem*{rep@theorem}{\rep@title}
\newcommand{\newreptheorem}[2]{%
\newenvironment{rep#1}[1]{%
 \def\rep@title{#2~\ref{##1}}%
 \begin{rep@theorem}}%
 {\end{rep@theorem}}}
 
\makeatother

\theoremstyle{plain}

\newtheorem*{thm*}{Theorem}

\newtheorem*{UBTP}{Upper Bound Theorem for polytopes}
\newtheorem*{UBTS}{Upper Bound Theorem for spheres}
\newtheorem*{UBTM}{Upper Bound Theorem for Minkowski sums}
\newtheorem*{UBTMF}{Upper Bound Theorem for mixed facets}

\newtheorem{thm}{Theorem}[section]
\newtheorem{cor}[thm]{Corollary}
\newtheorem{lem}[thm]{Lemma}
\newtheorem*{lem*}{Lemma}
\newtheorem{prp}[thm]{Proposition}

\newtheorem{quest}[thm]{Question}
\newreptheorem{thm}{Theorem}
\newreptheorem{cor}{Corollary}

\theoremstyle{definition}
\newtheorem{dfn}[thm]{Definition}
\newtheorem{ex}[thm]{Example}

\newtheorem{rem}[thm]{Remark}
\newtheorem{obs}[thm]{Observation}

\newcommand*\longhookrightarrow{\ensuremath{\lhook\joinrel\relbar\joinrel\rightarrow}}
\newcommand*\longtwoheadrightarrow{\ensuremath{\relbar\joinrel\twoheadrightarrow}}

\definecolor{light-gray}{gray}{0.55}

\newcommand\hlightb[1]{\tikz[overlay, remember picture,baseline=-\the\dimexpr\fontdimen22\textfont2\relax]\node[rectangle,fill=blue!50,rounded corners,fill opacity = 0.2,draw,thick,text opacity =1] {$#1$};} 
\newcommand\hlightr[1]{\tikz[overlay, remember picture,baseline=-\the\dimexpr\fontdimen22\textfont2\relax]\node[rectangle,fill=red!50,rounded corners,fill opacity = 0.2,draw,thick,text opacity =1] {$#1$};}

\newcommand{\xqedhere}[2]{%
  \rlap{\hbox to#1{\hfil\llap{\ensuremath{#2}}}}}

\renewcommand{\k}{\mathbbm{k}}

\newcommand{\m}{\mathfrak{m}}

\renewcommand{\a}{\alpha}
\renewcommand\emptyset{\varnothing}
\newcommand\K{\Delta}
\newcommand\emp{\K_{0}}
\newcommand\Hom{H}
\newcommand\rHom{\widetilde{\Hom}}
\newcommand\loCo{\Hom_\m}

\newcommand\rBetti{\widetilde{\beta}}
\newcommand\RC{\Psi}
\newcommand\FM{\mathrm{M}}

\newcommand\Cay{\mathrm{Cay}}
\newcommand\V{\mathrm{V}}

\newcommand\Pm{P_{[m]}}

\newcommand\Defn[1]{\textbf{#1}}
\newcommand{\cm}[1]{}
\newcommand\mc[1]{\mathcal{#1}}
\newcommand\frk[1]{\mathfrak{#1}}
\newcommand\mr[1]{\mathrm{#1}}
\newcommand\ol[1]{\overline{#1}}

\newcommand\wt[1]{\widetilde{#1}}
\newcommand\scr[1]{\mathscr{#1}}

\newcommand{\bigast}{\mathop{\scalebox{1.7}{\raisebox{-0.2ex}{$\ast$}}}}%
\newcommand{\R}{\mathbb{R}}
\newcommand{\Z}{\mathbb{Z}}

\newcommand\MUBT{\omega}
\newcommand\Hilb{\mathrm{F}}
\newcommand\hsop{h.s.o.p.}
\newcommand\lsop{l.s.o.p.}

\newcommand{\sm}{\hspace{0.08 em}}

\newcommand\cfs[1]{\ol{\scr{#1}}}
 \newcommand\x{\mathbf{x}}
 \renewcommand\t{\mathbf{t}}
 \newcommand\I{\mathrm{I}}
\newcommand\ch{\mr{I}_{\frk{N}}}

\newcommand{\LT}{h^\mathsf{alg}}
\newcommand{\GT}{h^\mathsf{top}}

\DeclareMathOperator{\Lk}{lk}
\DeclareMathOperator{\St}{st}
\DeclareMathOperator{\supp}{supp}
\DeclareMathOperator*{\relint}{relint}
\DeclareMathOperator*{\conv}{conv}

\DeclareMathOperator*{\codim}{codim}
\DeclareMathOperator*{\depth}{depth}

\title{Relative Stanley--Reisner theory and Upper Bound Theorems for Minkowski
sums}

\author{Karim A.~Adiprasito}
\author{Raman Sanyal}
\address{Einstein Institute for Mathematics, Hebrew University of Jerusalem, Jerusalem, Israel}
\email{adiprasito@math.huji.ac.il}
\address{Fachbereich Mathematik und Informatik, %
Freie Universit\"at Berlin, Berlin, %
Germany}
\email{sanyal@math.fu-berlin.de}


\date{\today}
\thanks{K.~Adiprasito was supported by an EPDI/IPDE postdoctoral fellowship, 
a Minerva fellowship of the Max Planck Society, the DFG within the research training group ``Methods for Discrete Structures''
(GRK1408) and by the Romanian NASR, CNCS --- UEFISCDI, project
PN-II-ID-PCE-2011-3-0533.}\thanks{R.~Sanyal was supported by European Research Council under the
European Union's Seventh Framework Programme (FP7/2007-2013) / ERC grant
agreement n$^\mathrm{o}$ 247029 and by the DFG-Collaborative Research Center,
TRR 109 ``Discretization in Geometry and Dynamics''.}

\begin{document}

\begin{abstract}
In this paper we settle two long-standing questions regarding the combinatorial
complexity of Minkowski sums of polytopes: We give a tight upper bound for the
number of faces of a Minkowski sum, including a characterization of the case
of equality. We similarly give a (tight) upper bound theorem for mixed facets
of Minkowski sums. This has a wide range of applications and generalizes the
classical Upper Bound Theorems of McMullen and Stanley.

Our main observation is that within (relative) Stanley--Reisner theory, it is possible to 
encode topological as well as combinatorial/geometric restrictions in an algebraic setup.
We illustrate the technology by providing several simplicial isoperimetric and reverse isoperimetric 
inequalities in addition to our treatment of Minkowski~sums.
\end{abstract}


\maketitle

\newcommand\Neighborly{\mr{NP}}
\newcommand\Cyclic{\mr{Cyc}}

The Upper Bound Theorem (UBT) for polytopes is one of the cornerstones of
discrete geometry. The UBT gives precise bounds on the
`combinatorial complexity' of a convex polytope $P$ as measured by the number of $k$-dimensional faces
$f_k(P)$ in terms of its dimension and the number of
vertices.
\begin{UBTP}
    For a $d$-dimensional polytope $P$ on $n$ vertices and $0 \le k < d$
    \[
        f_k(P) \ \le \ f_k(\Cyclic_d(n))
    \]
    where $\Cyclic_d(n)$ is a $d$-dimensional cyclic polytope on $n$ vertices. 
    Moreover, equality holds for all $k$ whenever it holds for some $k_0,\ k_0+1 \ge 
    \lfloor\frac{d}{2}\rfloor$.
\end{UBTP}

Polytopes attaining the upper bound are called (simplicial) \Defn{neighborly}
polytopes and are characterized by the fact that all non-faces are of
dimension at least $\frac{d}{2}$. Cyclic polytopes are a particularly
interesting class of neighborly polytopes whose combinatorial structure allows
for an elementary and explicit calculation of $f_k(\Cyclic_d(n))$ in terms of
$d$ and $n$; cf.~\cite[Section~0]{Z}.  The UBT was conjectured by
Motzkin~\cite{Motzkin57} and proved by McMullen~\cite{mcmullen1970}. One of
the salient features to note is that for given $d$ and $n$ there is a polytope
that maximizes $f_k$ for \emph{all} $k$ simultaneously --- a priori, this is
not to be expected. 

In this paper we will address more general upper bound problems for polytopes
and polytopal complexes. To state the main applications of the theory to be
developed, recall that the \Defn{Minkowski sum} of polytopes $P,Q \subseteq
\R^d$ is the polytope $P + Q = \{p+q : p \in P, q \in Q\}$.  There is no
understating the importance of Minkowski sums for modern mathematics. It is
named after Hermann Minkowski~\cite{minkowski}, who inaugurated the rich
theory of mixed volumes and geometric inequalities; see~\cite{Schneider93}.
Applications reach into algebraic geometry~\cite{KV,CoxLittleSchenck},
geometry of numbers and packings, computational commutative
algebra~\cite{GS,Sturmfels02}, robot motion planning~\cite{latombe}, and game
theory~\cite{games}. An important and practically relevant question is
regarding the combinatorial complexity of $P+Q$ is in terms of $P$ and $Q$.
More precisely the \emph{Upper Bound Problem for Minkowski sums} (UBPM),
raised (in print) by Gritzmann and Sturmfels~\cite{GS}, asks:

{\it
For given $k < d$ and $n_1,n_2,\dots,n_m$, what is the maximal number of 
$k$-dimensional faces of the Minkowski sum $P_1 + P_2+\cdots+P_m$  for
polytopes $P_1,\dots,P_m \subseteq \R^d$ with vertex numbers $f_0(P_i) =
n_i$ for $i=1,\dots,m$?
}

A solution to the UBPM subsumes the UBT for $m=1$. For $m > 1$, it is
nontrivial even for $k=0$: In~\cite{Sanyal09}, a comparatively involved topological argument is
employed to show that for $m \ge d$ the trivial upper bound of $n_1n_2\cdots
n_m$ vertices can not be attained.  On the constructive side, Fukuda and
Weibel~\cite{FukudaWeibel07,FukudaWeibel10,Weibel12} and
Matschke--Pfeifle--Pilaud~\cite{MPP} gave several constructions for Minkowski
sums that potentially maximize the number of faces.  In particular, the
constructions maximize the number of low-dimensional faces and, in analogy to
the classical situation, they will be called \emph{Minkowski neighborly
families} (see Sections~\ref{sec:UBTMS} and~\ref{sec:nonpure}).  Weibel
\cite{Weibel12} proved that the number of \emph{vertices} of a Minkowski sum
is maximized by Minkowski neighborly families. A recent breakthrough was
achieved by Karavelas and Tzanaki~\cite{Karavelas11} who resolved the UBPM for
two summands and subsequently for three summands in collaboration with
Konaxis~\cite{Karavelas12}. Both papers adapt McMullen's geometric approach
via shellings but with a dramatic increase in the complexity of the arguments.
In this paper we give a complete resolution of the UBPM including a
characterization of the equality case using a simple algebraic setup.

\begin{UBTM}[UBTM]
    For polytopes $P_1,\dots,P_m \subseteq \R^d$ with $n_1,\dots,n_m$ vertices
    and $0 \le k < d = \dim P_1 + \cdots + P_m$ 
    \[f_k(P_1 + \cdots + P_m) \ \le \ f_k(N_1 + \cdots +N_m)\]
    where the family $(N_1,\dots,N_m)$ is Minkowski neighborly with
    $f_0(N_i) = n_i$ for all $i=1,\dots,m$. Equality holds for all $k$ if it
    holds for some $k_0,\ k_0+1\ge \frac{d+2m-2}{2}$.
\end{UBTM}

A face of a Minkowski sum is \emph{mixed} if it is the sum of
positive-dimensional faces of the summands. Mixed faces play an important role
in mixed volume computations and they prominently appear in toric/tropical
intersection theory~\cite{FS,katz,TS}, sparse resultants~\cite{PS, EC} as well
as colorful geometric combinatorics~\cite{ABPS} and game theory. Our methods
also apply to the study of mixed faces and we establish strong upper bounds
and in particular characterize the case of equality in the most important
case.

\begin{UBTMF}
The number of mixed facets of a Minkowski sum is maximized by Minkowski neighborly families.
\end{UBTMF}

\noindent\emph{From discrete geometry to combinatorial topology to
commutative algebra.} An intriguing feature of the UBT is that its 
validity extends beyond the realm of convex polytopes and into combinatorial topology. Let $\Delta$ be a
triangulation of the $(d-1)$-sphere and, as before, let us write $f_k(\Delta)$
for the number of $k$-dimensional faces. For example, boundaries of simplicial
$d$-dimensional polytopes yield simplicial spheres, but these are by far not
all.

\begin{UBTS}
    For a simplicial $(d-1)$-dimensional sphere $\Delta$ on $n$ vertices
    \[
        f_{k}(\Delta) \ \le \ f_k(\Cyclic_d(n))
    \]
    for all $k = 0, 1,\dots,d-1$. Equality holds for some $k \ge
    \lfloor\frac{d}{2}\rfloor$ if and only if $\Delta$ is neighborly.
\end{UBTS}

The UBT for spheres was proved by Stanley~\cite{Stanley75} in answer to a conjecture of Klee \cite{Klee64} and relied on a
ground-breaking connection between combinatorial topology and commutative algebra
that was first described by Hochster and Reisner~\cite{Hochster77,Reisner}.
To a simplicial complex $\Delta$ one associates a finitely generated graded
$\k$-algebra $\k[\Delta]$ --- the \emph{Stanley--Reisner ring} of $\Delta$ ---
that algebraically encodes the simplicial complex.
Hochster and Reisner showed that, in turn, algebraic properties such as
Cohen--Macaulayness of $\k[\Delta]$ are determined by topological properties
of $\Delta$. The key observation of Stanley was that enumerative properties
and especially upper bounds on face numbers can be extracted from $\k[\Delta]$
using algebraic implications of Cohen--Macaulayness. This was the starting
point of Stanley--Reisner theory. Stanley's work spawned extensions of the UBT
to (pseudo-)manifolds with (mild) singularities; see for
example~\cite{Novik03, Novik05, MillerNovikSwartz11, NovikSwartz12}. A pivotal
result was a formula of Schenzel~\cite{Schenzel81} that relates algebraic
properties of $\k[\Delta]$ to the face numbers as well as topological
properties of $\Delta$, provided $\k[\Delta]$ is a \emph{Buchsbaum} ring
(which is in particular true for all manifolds).

The UBTM too will be the consequence of a statement in the topological domain
that we derive using algebra, though we will also briefly comment on a geometric approach to the problem.  The appropriate combinatorial/topological setup
for the UBPM is that of \Defn{relative simplicial complexes}: A relative
simplicial complex is a pair of simplicial complexes $\RC = (\Delta,\Gamma)$
where $\Gamma \subseteq \Delta$ is a subcomplex. The faces of $\RC$ are
precisely the faces of $\Delta$ not contained in $\Gamma$. The number of
$k$-dimensional faces of $\RC$ is therefore $f_k(\RC) = f_k(\Delta,\Gamma) :=
f_k(\Delta) - f_k(\Gamma)$.  The algebraic object naturally associated to a
relative complex $\RC = (\Delta,\Gamma)$ is the \Defn{Stanley--Reisner module}
or \Defn{face module} $\FM[\RC]$.  Upper Bounds Problems for relative
complexes have been considered in different guises for instance in the study
of comparison theorems for $f$-vectors~\cite{Bj07}, Upper Bound Theorems of
manifolds \cite{NovikSwartz09} and polyhedra~\cite{bl81,BKL86}, triangulations
of polytopes~\cite{mcmullen04}, and the study of sequentially Cohen--Macaulay
complexes and rings \cite{Duval, ABG}.
For the type of relative upper bound problems we will consider, however, it is
crucial to study complexes not only under topological restrictions (such as
the Buchsbaum or Cohen--Macaulay property) but to also take the combinatorics
and geometry of $\Gamma$ in $\Delta$ into account.  We show that relative
Stanley--Reisner theory has the capacity to encode such restrictions, and
exploit this fact heavily in the present paper.

\enlargethispage{5mm} 
\subsubsection*{Outline of the paper} 
We provide a gentle introduction to (relative) Stanley--Reisner theory that
starts (in \textbf{Section~\ref{sec:RelativeSR}}) with a review of the
classical setup, collecting also results pertaining to relative simplicial
complexes that are implicit in works of Stanley, Schenzel and others. The same
applies to \textbf{Section~\ref{sec:SchenzelFormula}}, where we extend the
Schenzel formula to the relative setting.  In \textbf{Section~\ref{sec:int}}
we recall Stanley's proof of the UBT for spheres which sets the stage for
general relative upper bound theorems. In particular, we discuss combinatorial
isoperimetric problems and the combinatorial restrictions we can impose on
relative complexes.

We illustrate our methods on a variety of simplicial isoperimetric and reverse
isoperimetric inequalities in parallel to the developments of the main methods.
A combinatorial isoperimetric inequality bounds (from above) the size of the
interior of a combinatorial object in terms of its boundary; a reverse
isoperimetric problem bounds the boundary in terms of its interior.

The Schenzel formula states that the entries of the $h$-vector of $\RC$ are
given by an \emph{algebraic} component $\LT(\RC)$ and a \emph{topological}
component $\GT(\RC)$ that we study individually. The latter is typically
an invariant of the problem we wish to consider, and hence
the former will be of main interest to us. In \textbf{Section~\ref{sec:loct}}
we develop several powerful tools for studying the algebraic component.

\begin{compactenum}[\rm (1)]
    \item \textbf{Section~\ref{ssec:change_presentation}} provides bounds by
        comparing a given relative complex to a simpler one. This technique
        recovers Stanley's approach to the UBT as a special case. The most
        challenging part is to characterize the case of equality which has an
        interesting connection to the Nerve Lemma. This approach is
        demonstrated in \textbf{Section~\ref{ssec:arrCM}} for arrangements of
        Cohen--Macaulay subcomplexes.  
    \item In \textbf{Section~\ref{ssec:primes}}, we integrate local
        information on the $h$-vector to obtain global bounds.  Combined with
        the fact that the algebraic component $\LT(\RC)$ is monotone under
        passing to subcomplexes, this can be used to derive effective upper
        bounds in many settings. This is an algebraic generalization of a
        geometric idea due to McMullen.
    \item The latter technique is refined in
        \textbf{Section~\ref{ssec:rel_diff}}, to give even stronger upper
        bounds on the algebraic $h$-numbers; in particular, we obtain a
        reverse isoperimetric inequality of a kind that seems new to the
        subject. 
    \item We close in \textbf{Section~\ref{ssec:rel_shell}} with a brief discussion
        of relative shellability. This technique can be used to give a
        combinatorial-geometric proof of the UBTM, although a proper proof
        is more intricate.
\end{compactenum}

In \textbf{Section~\ref{sec:UBTMS}} we cast the Upper Bound Problem for
Minkowski sums into a relative upper bound problem. The connection to
relative complexes is via \emph{Cayley polytopes}. \textbf{Section~\ref{ssec:two}}
illustrates the general approach for two summands and gives a simple proof
for the results of Karavelas--Tzanaki~\cite{Karavelas11}. The remainder of
the section gives a complete proof of the UBTM for \emph{pure} collections,
that is, polytopes $P_1,\dots,P_m \subseteq \R^d$ with $f_0(P_i) \ge d+1$ for
all~$i$.  \textbf{Section~\ref{sec:nonpure}} treats the general case without
restrictions on the number of vertices.
In \textbf{Section~\ref{sec:t_mxf}} we combine our results with the
combinatorics of Cayley polytopes to give an upper bound on the number of
mixed faces of a general sum $P_1+P_2+\cdots+P_m$.  For mixed facets, this
bound is tight, and maximized by Minkowski neighborly families. 

\textbf{Acknowledgements.} We wish to thank the anonymous referee for a very thorough report. The first author also wants to express his gratitude to Eran Nevo for inspiring conversations and helpful comments. We are also grateful for the hospitality of the IH\'ES in Bures-sur-Yvette, where most of the research leading to this paper was conducted.

\section{Relative Stanley--Reisner theory}
\label{sec:RelativeSR}

In this section we lay out the foundations for \emph{relative Stanley--Reisner}
theory, an algebraic-combinatorial theory for relative simplicial
complexes. For further background on Stanley--Reisner rings and combinatorial
commutative algebra, we refer to~\cite{Stanley96}
and~\cite{MillerSturmfels05}.

A \Defn{simplicial complex} $\Delta$ is a collection of subsets $\Delta
\subseteq 2^{[n]}$ for some $[n]:=\{1,2,\dots,n\}$ that is closed under
taking subsets. We explicitly allow $\Delta$ to be empty and we
call $\emptyset$ the \Defn{void complex}. Thus any simplicial complex
$\Delta \neq \emptyset$ contains the empty face $\emptyset$. For $S
\subseteq [n]$ the simplex with vertex set $S$ is denoted by $\K_S := 2^S$.
We also write $\K_n = \K_{[n]}$, and set $\emp := \{ \emptyset \}\neq
\emptyset$.  A \Defn{relative simplicial complex} is a pair $\RC =
(\Delta,\Gamma)$ of simplicial complexes for which $\Gamma \subseteq \Delta$
is a proper subcomplex.  The \Defn{faces} of $\RC=(\Delta,\Gamma)$ are
the elements
\[
    \Delta {\setminus} \Gamma \ = \ \{ \sigma \in \Delta : \sigma \not\in \Gamma\}.
\] 
An ordinary simplicial complex is thus a relative simplicial complex with
$\Gamma = \emptyset$.  The \Defn{dimension} of a relative simplicial complex
is 
\[
    \dim \RC := \max \{ \dim \sigma : \sigma \in (\Delta,\Gamma)\} 
\] 
where $\dim \sigma = |\sigma|-1$.
We say $\RC$ is \Defn{pure} if all inclusion maximal faces in
$\Delta {\setminus}\Gamma$ are of the same dimension.  The \Defn{vertices} of a
relative complex are denoted by $\V(\RC) := \{ i \in [n] : \{i\}\in
\RC\}$.  We write $\RC^{(i)}$ to denote the \Defn{$(i-1)$-skeleton} of
$\RC$, i.e.\ the subcomplex of all faces of $\RC$ of dimension $< i$.

We denote by $\rHom_\bullet(\Delta,\k)$ \Defn{reduced homology} with
coefficients in $\k$; unless reference to the coefficient field is necessary, 
we omit it.  If $\RC =
(\Delta,\Gamma)$ is a relative simplicial complex, then $\rHom_\bullet(\RC) =
\rHom_\bullet(\Delta,\Gamma)$ is the usual relative homology. Observe that 
$\rHom_\bullet(\Delta,\K_0)$ is the \emph{unreduced} homology of $\Delta$.
The \Defn{reduced Betti numbers} are denoted by $\rBetti_i(\RC;\k) = \dim_\k \rHom_i(\RC;\k)$.

Let $\k[\x] = \k[x_1,\dots,x_n]$ be the polynomial ring over $\k$ in $n$
variables.  For a monomial $\x^\alpha$ the support is defined as
$\supp(\x^\alpha) = \supp(\alpha) = \{ i
\in [n] : \alpha_i > 0 \}$.  For a simplicial complex $\Delta$ on 
$[n]$, the \Defn{Stanley--Reisner ideal}, or \Defn{face ideal}, is the ideal 
\[
    \I_\Delta 
    \ := \ \langle \x^\tau : \tau \subseteq [n], \tau \not\in \Delta
    \rangle 
    \ = \ \k\text{-span}\{ \x^\alpha : \supp(\x^\alpha) \not\in
    \Delta\} 
    \ \subseteq \ \k[\x].
\]
The \Defn{Stanley--Reisner ring} or \Defn{face ring} is $\k[\Delta] := \k[\x] 
/ \I_\Delta$.  
The appropriate algebraic object associated to a relative complex $\RC =
(\Delta,\Gamma)$ is the \Defn{Stanley--Reisner module} or
\Defn{face module}
\[
    \FM[\RC] \ = \ \FM[\Delta,\Gamma] \ := \ \ker( \k[\Delta]
    \longtwoheadrightarrow \k[\Gamma] ) \ \cong \ \I_\Gamma / \I_\Delta \
    \subseteq \ \k[\Delta].
\]
We regard $\FM[\RC]$ as a module over $\k[\x]$.
If $\Gamma = \emp$ is the empty complex, then $\FM[\RC] \cong \k[\Delta]$.

\subsection{Face numbers and Hilbert functions}
For a (relative) simplicial complex $\RC$ of dimension $\dim \RC = d-1$, the
\Defn{$f$-vector} of $\RC$ is defined as
$
    f(\RC)  \ := \ (f_{-1},f_0,\dots,f_{d-1})
$
where $f_i=f_i(\RC)$ is the number of $i$-dimensional faces of $\RC$.  If $\RC = (\Delta,
\Gamma)$, then $f(\RC) = f(\Delta) - f(\Gamma)$. 

A $\k[\x]$-module $M$ is \Defn{$\Z^n$-graded} or \Defn{finely graded} if $M =
\bigoplus_{\alpha \in \Z^n} M_\alpha$ and $\x^{\beta}M_\alpha \subseteq
M_{\alpha+\beta}$ for all $\beta \in \Z^n_{\ge 0}$. 
For $\RC = (\Delta,\Gamma)$,
the \Defn{fine Hilbert series} is given by 
\[
    \Hilb(\FM[\RC],\t) \ = \ \Hilb(\FM[\RC],t_1,\dots,t_n) \ := \
    \sum_{\alpha \in \Z^n} \dim_\k \FM[\RC]_\alpha \, \t^\alpha \ = \
    \sum_{\sigma \in \Delta{\setminus}\Gamma} \prod_{i \in \sigma}
    \frac{t_i}{1-t_i}.
\]
The fine grading specializes to a \Defn{$\Z$-grading} or \Defn{coarse grading} and we obtain
\[
    \Hilb(\FM[\RC],t) 
    \ = \ \sum_{\sigma \in \Delta{\setminus}\Gamma}
        \frac{t^{|\sigma|}}{(1-t)^{|\sigma|}} 
    \ = \ \frac{ \sum_{k=0}^d f_{k-1}\,t^k(1-t)^{d-k}}{(1-t)^{d}} 
    \ = \ \frac{h_0 + h_1t + \dots + h_d t^d}{(1-t)^{d}}.
\]
We use the last equality as the definition of the \Defn{$h$-vector} $h(\RC) =
(h_0,\dots,h_d)$ of the (relative) simplicial complex $\RC$. If $\RC$ is a
simplicial complex, that is, if $\Gamma = \emptyset$, then $h_0 = 1$ and $h_1
= f_0(\Delta) - d$. If $\dim \Gamma = \dim \Delta$, then
$h(\RC) = h(\Delta) - h(\Gamma)$ and hence $h_0(\RC) = 0$ and $h_1(\RC) =
f_0(\Delta) - f_0(\Gamma)$. 
The conversion between $f$-vector and $h$-vector can be made explicit as
\begin{equation}\label{eqn:f-h-polynomial}
    \sum_{i=0}^d f_{i-1}(\RC) t^{d-i} \ = \ \sum_{k=0}^d h_k(\RC) (t+1)^{d-k}.
\end{equation}
Individual entries are thus given by 
\[
    h_k(\RC) \ = \ \sum^k_{i=0} (-1)^{k-i} \binom{d-i}{k-i}f_{i-1}(\RC)
    \quad \text{and} \quad 
    f_{i-1}(\RC) \ = \ \sum_{k=0}^{i}\binom{d-k}{i-k} h_k(\RC).
\]
The second formula is crucial for upper bounds on face numbers:
\begin{obs}\label{obs:f_via_h}
    The number of $(i-1)$-faces $f_{i-1}(\RC)$ is a positive linear
    combination of the $h$-numbers $h_0(\RC),\dots,h_i(\RC)$. In particular, upper bounds on
    entries of the $h$-vector imply upper bounds on the $f$-vector.
\end{obs}
Finally, the \Defn{$g$-vector} of the
(relative) simplicial complex $\RC$ is
\[
    g(\RC) \ = \ (g_1,\dots,g_d) \ := \ (h_1-h_0,h_2 - h_1,\dots,h_d-h_{d-1}).
\]
The \Defn{link} of $\sigma \subseteq [n]$ in a simplicial complex $\Delta$
is $\Lk(\sigma,\Delta) := \{ \tau \in \Delta : \sigma \cap \tau = \emptyset,
\sigma \cup \tau \in \Delta\}$.  In particular, $\Lk(\emptyset,\Delta) =
\Delta$  and $\Lk(\sigma,\Delta) = \emptyset$ is the void complex whenever
$\sigma \not\in \Delta$.  The (closed) \Defn{star} of a face $\sigma$ in a
simplicial complex $\Delta$ is defined as $\St(\sigma,\Delta) := \{ \tau \in
\Delta : \sigma \cup \tau \in \Delta\}$ and we define the \Defn{deletion} of a
face $\sigma$ of $\Delta$ as $\Delta-\sigma:=\{\tau\in \Delta: \sigma
\not\subseteq \tau\}$. In particular, link and star of a face $\sigma$ are
related by $\Lk(\sigma,\Delta)=\St(\sigma,\Delta) -\sigma$.  For a relative
simplicial complex $\RC=(\Delta,\Gamma)$, the notions of link and star are
defined to respect the relative structure: For $\sigma \subseteq [n]$, we set
$\Lk(\sigma,\RC) = (\Lk(\sigma,\Delta),\Lk(\sigma,\Gamma))
$ and $\St(\sigma,\RC) = (\St(\sigma,\Delta),\St(\sigma,\Gamma)).$
\begin{lem}\label{lem:lkst}
Let $\RC=(\Delta,\Gamma)$ be a relative simplicial complex. Let $v \in \V(\Delta)$ be any
vertex and let $e-1 = \dim \St(v,\RC) = \dim \Lk(v,\RC)+1$. Then
\[
   h_k(\Lk(v,\RC)) \ = \ h_k(\St(v,\RC))
\]
for all $0 \le k < e$ and $h_e(\St(v,\RC)) = 0$.
\end{lem}
\begin{proof}
    Observe that $\FM[\St(v,\RC)] / x_v \FM[\St(v,\RC)] \cong
    \FM[\Lk(v,\RC)]$ and that $\FM[\St(v,\RC)] \xrightarrow{\, \times x_v\, }\FM[\St(v,\RC)]$
    is injective. Passing to the coarse Hilbert series proves the claim.
\end{proof}

The \Defn{Euler characteristic} of a (relative) simplicial $(d-1)$-complex $\RC$ is
\[
    \chi(\RC) \ := \  \sum_{i=0}^{d-1} (-1)^i f_{i}(\RC) \ = \ (-1)^{d-1} h_d(\RC).
\]
Notice that $\chi(\Delta) = \chi(\Delta,\emptyset)$ is the \emph{reduced}
Euler characteristic and $\chi(\RC) = \chi(\Delta) - \chi(\Gamma)$.

It turns out that for various classes of simplicial complexes, the entries
of the $h$-vector are not independent from each other. If $\Delta$ is a
simplicial sphere, the classical Dehn--Sommerville equations state $h_k =
h_{d-k}$ for all $k=0,\dots, d$, a relation closely related to
Poincar\'e duality.

The following two results are
generalizations of the classical Dehn--Sommerville relations to the relative 
setting, to manifolds and to balls.
Recall that a (relative) simplicial complex $\RC=(\Delta,\Gamma)$ is
\Defn{Eulerian} if $\RC$ is pure and $\chi(\Lk(\sigma,\RC)) =
(-1)^{\dim\Lk(\sigma,\RC)}$ for all $\sigma \in \RC$.  For example, all
(homology) spheres are Eulerian. For general (homology) manifolds, a weaker
notion is in order.  The relative complex $\RC$ is \Defn{weakly Eulerian} if
$\chi(\Lk(\sigma,\RC)) = (-1)^{\dim\Lk(\sigma,\RC)}$ for all nonempty
faces~$\sigma$. The following lemma is a version of the 
Dehn--Sommerville relations (cf.\ \cite{Klee64, HGG, NovikSwartz09}) for relative complexes.

\begin{lem}[Dehn--Sommerville relations]\label{lem:DS}
    Let $\RC = (\Delta,\Gamma)$ be a relative simplicial complex of dimension
    $d-1$. If $\RC$ is weakly Eulerian, then
    \[
        h_{d-i}(\RC) \ = \
        h_{i}(\Delta)+(-1)^i\binom{d}{i}\Bigl((-1)^{d-1}\chi(\Delta,\Gamma) -
        1 \Bigr).
    \]
\end{lem}

Lemma~\ref{lem:DS} is a direct consequence of the following
reciprocity law of $\FM[\RC]$. 

\begin{prp}\label{prp:Hilbert_reciproc}
    Let $\RC = (\Delta,\Gamma)$ be a relative complex of dimension $d-1$.
    Then the fine Hilbert series of $\FM[\RC]$ satisfies
    \[
        (-1)^d\,\Hilb(\FM[\RC],\tfrac{1}{t_1},\dots,\tfrac{1}{t_n}) \ = \
        \sum_{a \in \Z_{\ge 0}^n} (-1)^{\codim \sigma_a}
        \chi(\Lk(\sigma_a,\RC)) \, \t^a
    \]
    where $\sigma_a = \supp(a)$ and $\codim \sigma_a := \dim \RC - \dim
    \sigma_a$.
\end{prp}
This is a statement analogous to Lemma~5.4.3 in \cite{Bruns-Herzog}.
If $\RC = (\Delta,\Gamma)$
is weakly Eulerian, then Proposition~\ref{prp:Hilbert_reciproc}
yields
\[
    (-1)^d \Hilb(\FM[\RC],\tfrac{1}{\t}) \ = \
    (-1)^{d-1}\chi(\Delta,\Gamma)-1+\Hilb(\FM[\Delta],\t).
\] 
Passing to the coarse Hilbert series proves Lemma~\ref{lem:DS}.

\subsection{Cohen--Macaulay and Buchsbaum modules and complexes}
Let $M$ be a finitely generated graded module over $\k[\x]$.  For a sequence
$\Theta = (\theta_1,\dots,\theta_\ell)$ of elements of $\k[\x]$, let us write
$\Theta_s = (\theta_1,\dots,\theta_s)$ for the subsequence of the first $s$
elements and $\Theta_s M = \sum_{i=1}^s \theta_i M$.  A \Defn{partial
homogeneous system of parameters} (partial \hsop) is a sequence $\Theta =
(\theta_1,\dots,\theta_\ell)$ of homogeneous elements of $\k[\x]$ such that
\[
    \dim M / \Theta M \ = \ \dim M-\ell.
\]
If $\ell = \dim M$, then $\Theta$ is a homogeneous system of parameters.  If
all $\theta_i$ are of degree~one, then $\Theta$ is a \Defn{(partial) linear system of
parameters} (\lsop). Throughout the paper we assume that the field $\k$ is
infinite, which guarantees the existence of a linear system of parameters 
(cf.~\cite[Theorem.~13.3]{eis}). 

A sequence $\Theta = (\theta_1,\dots,\theta_\ell)$ of homogeneous elements
is called \Defn{$M$-regular}, if $\Theta M \neq M$ and  
\[
    \Theta_{i-1}M : \theta_i \ := \ \{ m \in M : m\theta_i \in \Theta_{i-1}M\}
    \ = \ \Theta_{i-1}M
\]
for all $i = 1,\dots,\ell$. Every regular sequence is a partial \hsop\ but the
converse is false. An immediate consequence of the definition is the
following.
\begin{prp}\label{prp:regular}
    Let $M$ be a finitely generated graded module over $\k[\x]$ and $\Theta =
    (\theta_1,\dots,\theta_r)$ be an $M$-regular sequence of linear forms.
    Then
    $(1-t)^r \Hilb(M,t)=\Hilb(M/\Theta M,t)$.
\end{prp}

The length of the longest regular sequence of a module $M$ is called the
\Defn{depth} $\depth(M)$ of $M$. Clearly, $\depth(M) \le \dim(M)$, and
Cohen--Macaulayness characterizes the case of equality.

\begin{dfn}
    A $\k[\x]$-module $M$ is called \Defn{Cohen--Macaulay} if $\depth(M) =
    \dim(M)$.  A relative simplicial complex $\RC = (\Delta,\Gamma)$ is
    \Defn{Cohen--Macaulay} (CM, for short) if $\FM[\RC]$ is a Cohen--Macaulay
    module.
\end{dfn}

To treat Upper Bound Problems on manifolds, we will need to consider 
a more general class of complexes.
Let us denote by $\m:=\langle x_1, \dots, x_n
\rangle \subseteq \k[\x]$ the irrelevant ideal.  An \hsop\ $\Theta =
(\theta_1,\dots,\theta_\ell)$ is a \Defn{weak $M$-sequence} if $\Theta M \neq M$ and
\begin{equation}\label{eqn:weakMseq}
    \Theta_{i-1} M : \theta_i \ = \ 
    \Theta_{i-1} M : \m
\end{equation}
for all $i = 1,\dots,\ell$.

\begin{dfn}
    A finitely generated graded module $M$ over $\k[\x]$ is \Defn{Buchsbaum}
    if every \hsop\  is a weak $M$-sequence. A relative simplicial complex
    $\RC$ is \Defn{Buchsbaum} if the face module $\FM[\RC]$ is Buchsbaum.
\end{dfn}

Every Cohen--Macaulay module is also a Buchsbaum module.  The converse is
clearly false, as will become clear once we provide the topological criterion
for Cohen--Macaulay and Buchsbaum complexes.

The depth of a module is detected by local cohomology.  We denote by
$\loCo^i(M)$ the $i$-th \Defn{local cohomology} of $M$ with support in $\m$,
cf.~\cite{24loCo}.
In the case of a face module, the $\Z^n$-graded
Hilbert-Poincar\'e series can be computed in terms of local topological
information of the relative simplicial complex. The following is a relative
version of a formula due to Hochster; see~\cite[Theorem~II.4.1]{Stanley96}
or~\cite[Theorem~13.13]{MillerSturmfels05}.

\begin{thm}[Hochster's formula for relative complexes]\label{thm:HF}
Let $\RC=(\Delta,\Gamma)$ be a relative simplicial complex and let
$\FM=\FM[\RC]$ denote its face module. The Hilbert series of the local
cohomology modules in the fine grading is
\[
    \Hilb(\loCo^i(\FM),\t) \ = \ \sum_{\sigma\in \Delta} \dim_\k \rHom_{i-\dim
    \sigma-2} (\Lk(\sigma,\RC))\prod_{i\in \sigma}
    \frac{t_i^{-1}}{1-t_i^{-1}}.
\]
In other words, for $\alpha\in \Z^n$ and $\sigma_\alpha:=\supp(\alpha)$, we
have 
\[
    \loCo^i(\FM)_\alpha \ \cong \ 
    \begin{cases}
        \rHom_{i-\dim \sigma-2}(\Lk(\sigma_\alpha,\RC)) & \text{if } \alpha
        \le 0 \text{ and } \sigma_\alpha \in \Delta, \text{ and}\\
        0 & \text{otherwise}.
    \end{cases}
\]
\end{thm}
\begin{proof}
    Following the proof of \cite[Theorem~13.13]{MillerSturmfels05}, let
    $\check{C} = \bigoplus_{\sigma \subseteq [n]} \k[\x]_{\x^\sigma}$ be the
    $\Z^n$-graded \v{C}ech complex with respect to $x_1,\dots,x_n$ where
    $\k[\x]_{\x^\sigma}$ is the localization at $\x^\sigma$. By
    definition of $\FM = \FM[\RC]$, we have the short exact sequence
    of $\Z^n$-graded modules
    \[
        0 \longrightarrow \FM \longrightarrow \k[\Delta] \longrightarrow \k[\Gamma]
        \longrightarrow 0.
    \]
    Let $\alpha \in \Z^n$ be arbitrary but fixed and
    let $\alpha^+,\alpha^- \in \Z^n_{\ge0}$ such that $\alpha = \alpha^+ -
    \alpha^-$ and $\supp(\alpha^+)\cap\supp(\alpha^-) = \emptyset$. Moreover
    let $\sigma = \supp(\alpha)$.
    From the proof of Theorem 13.13 in \cite{MillerSturmfels05}, we know that 
    the complex of $\k$-vector spaces $(\k[\Delta] \otimes \check{C})_\alpha$ is
    isomorphic to the chain complex of $\Lk(\sigma,\Delta)$ shifted by
    $|\supp(\alpha^-)|+1$ if $\alpha^+ = 0$ and acyclic otherwise. The
    argument applies to $(\k[\Gamma] \otimes \check{C})_\alpha$ as well and
    thus $(\FM \otimes \check{C})_\alpha$ is isomorphic to the chain complex
    of the pair
    $\Lk(\sigma, \RC) = (\Lk(\sigma,\Delta,\Lk(\sigma,\Gamma))$, again shifted
    by $|\supp(\alpha^-)|+1$ if $\alpha^+ = 0$. It
    follows that as $\k$-vector spaces $\loCo^i(\FM)_\alpha \cong 
    \rHom_{i-\dim \sigma-2}(\Lk(\sigma_\alpha,\RC))$ whenever $\alpha \le
    0$ and identically zero otherwise.
\end{proof} 

From the relative Hochster formula and the fact that $H^i_\m(M) = 0$ for $i
< e = \depth(M)$ and $H^e_\m(M) \neq 0$ we deduce a relative version of a
criterion of Reisner~\cite{Reisner} for  simplicial complexes to be
Cohen--Macaulay.

\begin{thm}[{\cite[Theorem III.7.2]{Stanley96}}]\label{thm:SR}
    A relative simplicial complex $\RC=(\Delta,\Gamma)$ is Cohen--Macaulay if
    and only if
    $\rHom_i(\Lk(\sigma,\RC))  =  0$ for all faces $\sigma \in \Delta$  and all $i<\dim \Lk(\sigma,\RC)$.
\end{thm}

\begin{cor}[{\cite[Theorem III.7.3]{Stanley96}}]\label{cor:cmincm}
    Let $\Delta$ be a simplicial complex and $\Gamma \subseteq \Delta$ a
    subcomplex. If $\Delta$ and $\Gamma$ are Cohen--Macaulay and $\dim \Delta
    - \dim \Gamma \le 1$, then $(\Delta,\Gamma)$ is Cohen--Macaulay.
    Conversely, if $\Delta$ and $(\Delta,\Gamma)$  are Cohen--Macaulay then
    $\dim \Delta - \dim \Gamma \le 1$ and $\rHom_i(\Lk(\sigma,\Gamma)) = 0$
    for all $\sigma \in \Gamma$ and $i < \dim \Delta -\dim \sigma - 2$.
\end{cor}
\begin{proof}
    Let us write $\Lk(\sigma,\RC) = (\Delta_\sigma, \Gamma_\sigma)$ for $\sigma \in \Delta$.  For the pair
    $(\Delta_\sigma,\Gamma_\sigma)$ consider the long exact sequence 
    \[ 
        \cdots\; \longrightarrow 
        \rHom_{i}(\Delta_\sigma) \, \longrightarrow \, 
        \rHom_{i}(\Delta_\sigma,\Gamma_\sigma) \, \longrightarrow 
        \rHom_{i-1}(\Gamma_\sigma) \, \longrightarrow \, 
        \rHom_{i-1}(\Delta_\sigma) \, \longrightarrow \, 
        \rHom_{i-1}(\Delta_\sigma,\Gamma_\sigma) \, \longrightarrow \, 
        \cdots
    \]
    The vanishing of homologies splits the sequence and the first claim
    follows from 
    \[
        0 = \rHom_i(\Delta_\sigma) \rightarrow
        \rHom_i(\Delta_\sigma,\Gamma_\sigma) \rightarrow
        \rHom_{i-1}(\Gamma_\sigma) = 0\ \text{ for }\ i < \dim
        (\Delta_\sigma,\Gamma_\sigma) \le \dim \Delta_\sigma.
    \] 
    The second claim follows analogously.
\end{proof}

A similar criterion can be derived to characterize (relative) Buchsbaum
complexes; cf.\ \cite{miyazaki, Schenzel81}.

\begin{thm}\label{thm:topb}
    For a pure relative simplicial complex $\RC = (\Delta,\Gamma)$ of
    dimension $d-1$ the following are equivalent:
    \begin{compactenum}[\rm (i)]
        \item $\RC$ is a Buchsbaum complex.
        \item $\FM[\RC]$ is a Buchsbaum module.

        \item The link of every vertex is Cohen--Macaulay. {\rm ($\RC$ is
            \Defn{locally Cohen--Macaulay}.)}

        \item For every nonempty face $\sigma$ of $\Delta$ and all $i <
            d-\dim \sigma-1$, we have
            $\rHom_i(\Lk(\sigma,\RC))=0$.
\end{compactenum}
\end{thm}

\begin{proof}
    (i) $\Leftrightarrow$ (ii) is true by definition. The equivalence (iii)
    $\Leftrightarrow$ (iv) is Theorem~\ref{thm:SR}.
    Assuming (ii), we obtain from equation~\eqref{eqn:weakMseq} that every localization
    $\FM[\RC]_\mathfrak{p}$ at primes $\mathfrak{p} \neq \m$ yields a
    Cohen--Macaulay module. The implication (ii) $\Rightarrow$ (iii) now
    follows from the same argument as in~\cite[Lemma~5]{Reisner} applied to
    face modules.
    Finally, assuming~(iv), it follows from Theorem~\ref{thm:HF}
    that $\loCo^i(\FM[\RC])_j = 0$ whenever $j \neq 0$ and  $0 \le i \le \dim
    \RC$ and~\cite[Satz~4.3.1]{Schenzel82} assures us that $\FM[\RC]$ is
    Buchsbaum.
\end{proof}

An immediate corollary from the topological characterizations is that the Cohen--Macaulay and Buchsbaum properties are
inherited to skeleta.
\begin{cor}\label{cor:bher}
    The $k$-skeleton of a relative Cohen--Macaulay or Buchsbaum complex is
    Cohen--Macaulay or  Buchsbaum, respectively.
\end{cor}

\section{Local cohomology of Buchsbaum face modules}
\label{sec:SchenzelFormula}

A key observation in Stanley's proof of the UBT for spheres $\Delta$ is that 
\[
    \Hilb(\k[\Delta]/ \Theta \k[\Delta],t) \ = \ 
    h_0(\Delta) + h_1(\Delta)t + \dots + h_d(\Delta) t^d
\]
if $\Theta$ is a \lsop\ of length $d = \dim \Delta +1$. However,
this line of reasoning fails for general manifolds or Buchsbaum complexes. For these cases, an important tool was
developed by Schenzel~\cite{Schenzel81} that takes into account the
topological/homological properties of $\Delta$. In this section, we prove a
generalization of Schenzel's formula to relative complexes. Our
treatment is tailor-made for Stanley--Reisner modules and
slightly simpler than Schenzel's original approach.

A criterion of Schenzel~\cite[Satz~4.3.1]{Schenzel82} states that a graded
$\k[\x]$-module $M$ is Buchsbaum whenever the $\Z$-graded local cohomology
$\loCo^i(M)$ is concentrated in a fixed degree for all $0 \le i < \dim M$.
Schenzel~\cite{Schenzel81} showed that the converse holds for special classes
of Buchsbaum modules. For this, he used the purity of the Frobenius based on earlier work of Hochster and Roberts
\cite{H-R-Purity}. For Stanley--Reisner
modules, the concentration of local cohomology is a consequence of
Theorem~\ref{thm:HF}:

\begin{cor}\label{cor:deg0}
    Let $\RC=(\Delta,\Gamma)$ be a pure relative simplicial complex of
    dimension $d-1$. Then the following are equivalent:
    \begin{enumerate}[\rm (i)]
        \item $\rHom_i(\Lk(\sigma,\RC))=0$ for all non-empty faces $\sigma
            \in \Delta$ and all $i< \dim \Lk(\sigma,\Delta)=d-\dim \sigma-1$,
        \item The coarse-graded local cohomology of $\RC$ is concentrated in
            degree zero, that is, $\loCo^i(\FM[\RC])_j = 0$ for all $i, j$, $j \neq 0$. 
        \item $\FM[\RC]$ is a Buchsbaum module.
\end{enumerate}
\end{cor}

For the following discussion, let us write $R = \k[\x]$ and define $R(t)$ to
be the free rank-one $R$-module generated in degree $-t$. Recall that for an
homogeneous element $\theta_1 \in R_t$ the \Defn{Koszul complex} $K(\theta_1)$
is the complex
\[
    0\longrightarrow R \xrightarrow{\times \theta_1} R(t) \longrightarrow 0.
\]
For family of homogeneous elements $\Theta = (\theta_1, \theta_2,
\dots,\theta_n)$ the Koszul complex is defined as
\[
    K(\Theta) \ = \ K(\theta_1)\otimes_R
    K(\theta_2)\otimes_R \cdots \otimes_R K(\theta_n).
\]
For a graded $R$-module $M$ we denote by $\Hom^\bullet(\Theta;M) :=
\Hom^\bullet(K(\Theta)\otimes_R M)$ the \Defn{Koszul cohomology} of $M$ with
respect to $\Theta$.  The Koszul complex is a basic tool in the study of local
cohomology and we refer the reader to~\cite{Bruns-Herzog,24loCo} for further
information.

\begin{lem}\label{lem:koszulRanks}
    Let $M$ be a Buchsbaum module of dimension $d$ and $\Theta =
    (\theta_1,\dots,\theta_d)$ a \hsop\ for $M$. Then for all $0 \le i < d$,
    $\m\Hom^i(\Theta;M) = 0$. In particular, the Koszul cohomology modules are
    $\k$-vector spaces of dimension
    \[
        \dim_\k \Hom^i(\Theta_s;M) \ = \ \sum_{j=0}^i \binom{s}{i-j}\, 
            \dim_\k \loCo^j(M).
    \]
\end{lem}

The lemma is a direct consequence of Lemma~4.2.1 of~\cite{Schenzel82}
except that we need to verify that $\Hom^i(\Theta_s;M)$ is of finite length
for all $i$.

\begin{prp}
    If $M$ is a Buchsbaum module and $\Theta$ is a~\hsop, then
    $\Hom^i(\Theta_s;M)$ is of finite length for all $0 \le i \le s$.
\end{prp}

\begin{proof}
    For $s=d := \dim M$, $\Theta$ is a homogeneous system of parameters. Thus
    $\Hom^d(\Theta;M) = M/\Theta M$ is of dimension zero, and therefore of finite length and the result
    follows. 
    For $s < d$, observe that for all $j \ge 1$, $\Theta^\prime =
    (\theta_1,\dots,\theta_s,\theta_{s+1}^j)$ is a partial \hsop\ for $\FM$.
    Now, $K(\Theta^\prime)$ is the mapping cone of
    $\times \theta^j_{s+1}: K(\Theta_s)\longrightarrow K(\Theta_s)$ and the long
    exact sequence in cohomology yields
    \[
        0\longrightarrow \Hom^{i-1}(\Theta_{s};\FM)/\theta^{j+1}_s
        \Hom^{i-1}(\Theta_{s};\FM) \longrightarrow \Hom^{i}(\Theta^\prime;\FM).
    \]
    By induction on $s$ and Lemma~\ref{lem:koszulRanks}, 
    $\m\Hom^{i-1}(\Theta_s;\FM)=0$ since $\Hom^{i}(\Theta^\prime;\FM)$ is 
    annihilated by the irrelevant ideal for all $j\ge 1$. The finite length of 
    $\Hom^{i-1}(\Theta_s;\FM)$ now follows from Nakayama's Lemma.
\end{proof}

As a consequence we note the following.
\begin{cor}\label{cor:qm}
    Let $M$ be Buchsbaum $R$-module of dimension $d$ and let $\Theta$ be a
    \lsop\ for $M$. Then for all $0< s\le d$, we have
    \[
        \dim_\k (\Theta_{s-1} M :\theta_s)/(\Theta_{s-1} M) \  = \
        \sum_{i=0}^{s-1} \binom{s-1}{i}\, \dim_{\k} \loCo^i(M).
    \]
\end{cor}

\begin{proof}
As in \cite[Lemma 6.3.4]{Schenzel82}, we have the short exact sequence 
\[
    0\longrightarrow H^{s-2}(\Theta_{s-1};M)\longrightarrow
    H^{s-1}(\Theta_{s};M)\longrightarrow (\Theta_{s-1}
    M:\theta_s)/(\Theta_{s-1} M) \longrightarrow 0
\]
as $H^{s-1}(\Theta_{s};M)$ is annihilated by the irrelevant ideal.  By
Lemma~\ref{lem:koszulRanks} and the Buchsbaum property~\eqref{eqn:weakMseq},
this is a sequence of $\k$-vector spaces and hence splits. The result now
follows from Lemma~\ref{lem:koszulRanks}.
\end{proof}

With this, we can reprove and generalize Schenzel's Formula to relative 
Buchsbaum complexes, and therefore provide a central tool for relative 
Stanley--Reisner Theory.

\begin{thm}[Relative Schenzel Formula]\label{thm:RS}
    Let  $\RC=(\Delta,\Gamma)$ be a relative Buchsbaum complex of
    dimension $d-1$ and $\FM = \FM[\RC]$ the associated face module.
    If $\Theta$ is a \lsop\ for $\FM$, then 
    \[
        (1-t)^{d}\Hilb(\FM,t) \ = \ \Hilb(\FM/\Theta \FM,t) \ + \
        \sum_{j=1}^{d}\binom{d}{j}\left(\sum_{i=0}^{j-1}(-1)^{j-i}
        \rBetti_{i-1} (\Delta,\Gamma)\right)t^j.
    \]
\end{thm}

The formula states that the $h$-vector of a relative
complex $\RC$ is the sum of an algebraic and a topological component. In this
spirit, we denote the \Defn{algebraic component} by
\begin{align*}
    \LT_k(\RC) \ &:= \ \dim_\k (\FM[\Delta,\Gamma]/\Theta
    \FM[\Delta,\Gamma])_k,\\ 
    \intertext{where $\Theta$ is a \lsop\ and the \Defn{topological component} by}
    \GT_k(\RC) \ &:= \ 
        \binom{d}{k}\sum_{i=0}^{k-1}(-1)^{k-i}
        \rBetti_{i-1} (\Delta,\Gamma).
\end{align*}

We can rewrite the relative Schenzel formula as follows.

\begin{repthm}{thm:RS}
    Let  $\RC=(\Delta,\Gamma)$ be a relative Buchsbaum complex of dimension
    $d-1$. If $\Theta$ is a \lsop\ for $\FM = \FM[\RC]$, then 
    \[
        h_k(\RC) \ = \ \LT_k(\RC)+\GT_k(\RC)
    \]
    for any $k$.
\end{repthm}

\begin{proof}
    We closely follow Schenzel's proof of the case of simplicial
    complexes~\cite[Theorem~4.3]{Schenzel81}.  For a linear form $\theta  \in
    \k[\x]$ and a graded $\k[\x]$-module $N$ we have the exact sequence
    \[
        0\longrightarrow (0_N: \theta) \longrightarrow N 
        \xrightarrow{\  \theta\  } N(1) \longrightarrow
        (N/\theta N)(1)\longrightarrow 0
    \]
and hence 
    \[
        (1-t)\Hilb(N,t) \ = \ \Hilb(N/\theta N, t) -
        t\Hilb((0_N:\theta),t).
    \]
Iterating this argument with $N = \FM[\RC]/\Theta_s \FM[\RC]$ for $1\le s\le d$, this yields
\[
    (1-t)^{d}\Hilb(\FM,t) \ = \ \Hilb(\FM/\Theta\FM,t) - \sum_{s=1}^{d} t
    (1-t)^{s-1} \Hilb((\Theta_{s-1} \FM:\theta_s)/\Theta_{s-1} \FM,t).
\]
By Corollary~\ref{cor:qm} and Hochster's formula, we finally obtain
\[
    \Hilb((\Theta_{s-1} \FM:\theta_s)/\Theta_{s-1} \FM,t)
    \ = \
    \sum_{i=0}^{d-s} \binom{d-s}{i}\dim_{\k} \rBetti_{i-1}(\RC)\,
    t^i. \qedhere
\]
\end{proof}

\section{Relative Upper Bound Problems} \label{sec:int}

In this section we lay out model problems for relative upper bounds that will be
addressed with relative Stanley--Reisner theory and in particular relative Buchsbaum
complexes. We start by discussing the classical Upper Bound Theorem for polytopes and
spheres. We address two combinatorial isoperimetric problems that allow us to introduce
the notion of \emph{full} subcomplexes. The proofs of the respective upper bounds are
postponed to Section~\ref{ssec:arrCM} where general techniques will be available.

\subsection{The Upper Bound Theorem for spheres}\label{ssec:UBTspheres}

In the proof of the UBT for polytopes, the first step is to reduce the
problem of finding a $d$-polytope on $n$ vertices that maximizes the number
of $k$-faces to a problem about \emph{simplicial} $(d-1)$-spheres by observing that 
by perturbing the vertices of a polytope, the number of faces can only increase, cf.\ \cite{Klee64}.
In light of Observation~\ref{obs:f_via_h}, it is now sufficient to bound the $h$-vector 
of $\Delta$. The crucial lemma due to Stanley is the following. 

\begin{lem}\label{lem:UBT-CM}
    Let $\Delta$ be a $(d-1)$-dimensional Cohen--Macaulay complex on $n$
    vertices. Then
    \[
        h_k(\Delta) \ \le \ \binom{n - d - 1 + k}{k}
    \]
    for all $0 \le k \le d$. Equality holds for some $k_0$ if and only
    if $\Delta$ has no non-face of dimension $< k_0$.
\end{lem}

\begin{proof}
    For a linear system of parameters $\Theta$, Proposition~\ref{prp:regular}
    yields
    \[
        H(\k[\Delta]/\Theta \k[\Delta],t) \ = \ h_0 + h_1t + \dots + h_d t^d
    \]
    where $(h_0,\dots,h_d)$ is the $h$-vector of $\Delta$. We have the
    canonical graded surjection $N := \k[x_1,\dots,x_n] \twoheadrightarrow \k[\Delta]$
    as $\k[\x]$-modules. Now, $\Theta$ is a regular sequence for $N$ and 
    $N/\Theta N \cong \k[y_1,\dots,y_{h_1}]$ with $h_1 = n-d$. We obtain
    \[
        h_k(\Delta) \ = \ \dim_\k (\k[\Delta]/\Theta \k[\Delta])_k \ \le \ 
         \dim_\k \k[y_1,\dots,y_{h_1}]_k \ = \
        \binom{n-d+k-1}{k}
    \]
    for all $0 \le k \le d$ which completes the proof of the inequality. For the equality case, we may assume that $\Delta$ is not the $(n-1)$-simplex and
   thus $I_\Delta \neq 0$. By \cite[Proposition 1.1.4]{Bruns-Herzog}
   we have the short exact sequence 
    \[
    0 \longrightarrow \I_\Delta/\Theta \I_\Delta
        \longrightarrow \k[\x]/\Theta \k[\x] \longrightarrow \k[\Delta]/\Theta
        \k[\Delta]\longrightarrow
        0.
    \]
    Equality holds for $k_0$ if and only if $\I_\Delta/\Theta \I_\Delta$ and hence
    $I_\Delta$ 
     has no generators in degrees $\le k_0$.
\end{proof}

In fact stronger relations hold for the $h$-vector of a Cohen--Macaulay
complex. The following type of inequalities will be the subject of
Section~\ref{ssec:primes}; see Example~\ref{ex:rst} for the proof.

\begin{prp}\label{prop:UBT-g}
    Let $\Delta$ be a $(d-1)$-dimensional Cohen--Macaulay complex. Then
    \[
        k h_k(\Delta) \ \le \ (n-d+k-1) h_{k-1}(\Delta)
    \]
    for all $k=1,\dots,d$. In particular, $h_k(\Delta)\ \le\ \binom{n-d +k-1}{k}$ and
    \[
        g_k(\Delta) \ = \ h_k(\Delta)-h_{k-1}(\Delta)\ =\
        h_{k-1}(\Delta)\left(\frac{h_k(\Delta)}{h_{k-1}(\Delta)}-1\right) \
        \le \ \binom{n-d +k -2}{k}.
    \]
\end{prp}

Lemma~\ref{lem:UBT-CM} together with Reisner's criterion
(Theorem~\ref{thm:SR}) now implies upper bounds on the \emph{first half} of
the $h$-vector. The bounds are tight for boundary complexes of neighborly
polytopes.  The \emph{second half} of $h(P)$ is taken care of by the
Dehn--Sommerville equations which apply as $\partial P$ is Eulerian. This
enabled Stanley~\cite{Stanley75} to generalize McMullen's Upper Bound
Theorem~\cite{mcmullen1970} from polytopes to simplicial spheres.

\begin{thm}[Upper Bound Theorem for spheres]
    If $\Delta$ is a simplicial $(d-1)$-sphere on $n$ vertices, then 
    \[
    h_{k}(\Delta) \ = \ h_{d-k}(\Delta) \ \le \ \binom{n-d + k-1}{k}
    \]
    for all $0 \le k \le \lfloor\frac{d}{2}\rfloor$. Moreover, the $h$-vector
    is maximized precisely on neighborly $(d-1)$-spheres.
\end{thm}

For more on neighborly polytopes and McMullen's geometric perspective on
the upper bound theorem, we refer the reader to Section 8 of Ziegler's book
\cite{Z}.

\subsection{Combinatorial isoperimetric problems}\label{ssec:comb_iso}
The classical isoperimetric problem asks for the maximum volume of a
$d$-dimensional convex body $K$ with an upper bound on the surface area.
The following is a suitable discrete analog.

\begin{quest}[Combinatorial Isoperimetric Problem]\label{quest:comb_iso}
    Let $\Delta$ be a triangulation of a $d$-ball on $m+n$ vertices and $n$
    vertices in the boundary. What is the maximal number of $k$-faces in
    the interior of $\Delta$?
\end{quest}

This is a model problem for relative complexes. We seek to maximize
$f_k(\Delta,\partial \Delta) = f_k(\Delta) - f_k(\partial \Delta)$.  As it
turns out a resolution to the combinatorial isoperimetric problem can be
given using the ``classical'' tools of Section~\ref{sec:RelativeSR}, provided that we
make the additional assumption that the \emph{Generalized Lower Bound
Conjecture} of McMullen and Walkup~\cite{McMullenWalkup} holds for $\Gamma
= \partial\Delta$, that is, $g_k(\Gamma) = h_k(\Gamma) - h_{k-1}(\Gamma) \ge
0$ for all $k$. For a relative complex $\RC = (\Delta,\Gamma)$
with $\Gamma \neq \emptyset$, we have $h_0(\RC) = 0$ and $h_1(\RC) =
f_0(\Delta) - f_0(\Gamma)$ and we need only to worry about $h_k(\RC)$ for
$k \ge 2$.

\begin{thm}[Combinatorial isoperimetry of balls I]\label{thm:comb_iso}
    Let $\Delta$ be a simplicial $(d-1)$-ball on $m+n$ vertices with
    $n$ vertices in the boundary and assume that the Generalized Lower Bound
    Conjecture holds for $\partial\Delta$. Then the following inequalities
    hold:
    \begin{compactenum}[\rm (a)]
    \item For $2 \le k \le \frac{d}{2}$
    \[
        h_k(\Delta,\partial \Delta) \ \le \
        \binom{m+n-d+k-1}{k}.
    \]
Equality holds for some $k_0\le \frac{d}{2}$ if and only if every non-face
    $\sigma$ of $\Delta$ of dimension $< k_0$ is supported in $\partial
    \Delta$. 
    \item For $\frac{d}{2} < k \le d$
    \[
        h_{k}(\Delta, \partial \Delta) \ \le \  \binom{m+n-1-k}{d-k}.
    \]
    Equality holds for some $k_0> \frac{d}{2}$ if and only if $\Delta$ has
    no  non-face $\sigma$ of dimension $< d-k_0$.
\end{compactenum}
Moreover, the bounds are tight: For every $n \ge d \ge 0$ and $m \ge 0$, there is a 
$(d-1)$-ball that attains the upper bounds for every $k$ simultaneously. 
\end{thm}

\begin{proof}
For (a), notice that 
\[
    h_k(\Delta,\partial \Delta) \ = \ h_k(\Delta)-g_k(\partial \Delta) \ \le \ 
    \binom{m+n-d+k-1}{k}-g_k(\partial \Delta)
\] 
by Lemma~\ref{lem:UBT-CM}. The Generalized Lower Bound Conjecture for
$\partial\Delta$ yields the claim.

For part~(b), notice that by the Dehn--Sommerville relations
(Lemma~\ref{lem:DS}), we have $h_{k}(\Delta,\partial \Delta)  =  h_{d-k}
(\Delta)$ for all $0\le k\le d$ and we can again appeal to
Lemma~\ref{lem:UBT-CM}.

For tightness, let us consider any cyclic $(d-1)$-sphere $N$ on $m+n$ vertices.
Now, using the method of Billera--Lee~\cite{BL}, we may find a stacked
$(d-1)$-ball $B\subseteq N$ on $n$ vertices. With this, we can set $\Delta$ as
the subcomplex of $N$ induced by the facets not in $B$. Then, $\Delta$
is neighborly, and $\partial \Delta$ is stacked. Therefore, all inequalities
above are attained with equality.
\end{proof}

\subsection{Combinatorial isoperimetry II: Full complexes} \label{ssec:full}

One of the key features of relative Stanley--Reisner Theory is that we can
impose restrictions on the `position' of $\Gamma$ in $\Delta$.  A
profitable way to encode such a positional restriction of $\Gamma$ is the
notion of full subcomplexes:

\begin{dfn}\label{dfn:full_subcomplex}
    A subcomplex $\Gamma \subseteq \Delta$ is \Defn{full} if every face of
    $\Delta$ whose vertices are contained in $\Gamma$ is also a face of
    $\Gamma$. We call the relative complex $\RC = (\Delta,\Gamma)$ full if
    $\Gamma$ is full in $\Delta$.
\end{dfn}

The notion of full subcomplex generalizes the idea of vertex-induced
subgraphs. This is a very natural notion that makes prominent appearances
in PL topology \cite{ZeemanBK,RourkeSanders}, algebraic topology
\cite{GR,JMR}, graph theory, commutative algebra~\cite{Hochster77} and
geometric group theory \cite{CD,Davis}. 

While it may seem quite restrictive
to consider only full subcomplexes, we 
shall later see that the notion can be refined effectively using the
more flexible notion of ``full arrangements'', compare
Section~\ref{ssec:change_presentation}.

\begin{prp}
Let $\Delta$ denote any simplicial complex, and let $\Gamma$ denote any subcomplex. The following are equivalent:
\begin{enumerate}[\rm (i)]
\item $\Gamma$ is full in $\Delta$;
\item For every face $F\in \Delta$ with $\partial F\in \Gamma$, we have $F \in \Gamma$;
\item $\Gamma=\Delta \cap \K_{\V(\Gamma)}$.
\end{enumerate} 
\end{prp}

Notice that fullness is not a topological 
invariant; it is preserved under subdivisions, but not under PL homeomorphisms. 
The notion of fullness is, for instance, useful when identifying two simplicial 
complexes along a common subcomplex.  Fullness then guarantees that the result 
is again a simplicial complex. Hence, the notion of fullness can in particular be 
used to bound the complexity of PL handlebodies.

\begin{thm}[Combinatorial isoperimetry for manifolds]\label{thm:iso_mfds}
    Let $M$ denote a simplicial $(d-1)$-manifold on $m+n$ vertices, and let $B$
    denote a $(d-2)$-dimensional submanifold on $n$ vertices of $\partial M$ such that $B$ is
    full in $M$. Then 
    \[
        h_k(M,B)\ \le\ \binom{m+n-d+k-1}{k}\ -\ \binom{n-d+k-1}{k}\ +\
        \binom{d}{k}\sum_{i=0}^{k-1}(-1)^{k-i}\rBetti_{i-1}(M,B)
    \]
    for all $0\le k\le d.$
\end{thm}

By Theorem~\ref{thm:topb}, $\RC = (M,B)$ is a relative Buchsbaum complex and
we can use Theorem~\ref{thm:RS} to upper bound $h_k(\RC)$. We postpone the
estimation on the algebraic component to Section~\ref{sec:loct} where the
necessary tools are developed. In the case when $M$ is a $d$-ball and $B$ is
the bounding $(d-1)$-sphere, we can add the equality case to
Theorem~\ref{thm:iso_mfds}.

\begin{thm}[Combinatorial isoperimetry of balls II]\label{thm:comb_iso_full}
    Let $\Delta$ be a simplicial $(d-1)$-ball on $m+n$ vertices with
    $n$ vertices in the boundary and assume that $\partial\Delta \subseteq
    \Delta$ is a full subcomplex.
     Then the following inequalities
    hold:
\begin{compactenum}[\rm (a)]
    \item For every $0\le k\le \frac{d}{2}$
    \[
        h_k(\Delta,\partial \Delta)\ \le\
        \binom{m+n-d+k-1}{k}-\binom{n-d+k-1}{k}
    \]
    Equality holds for some $k_0\le \frac{d}{2}$ if and only if every non-face
    $\sigma$ of $\Delta$ of dimension $< k_0$ is supported in $\partial
    \Delta$. 
    \item For every $\frac{d}{2}< k \le d$
    \[
        h_{k}(\Delta, \partial \Delta) \ \le \ \binom{m+n-1-k}{d-k}.
    \]
    Equality holds for some $k_0> \frac{d}{2}$ if and only if $\Delta$ has
    no non-face $\sigma$ of dimension $< d-k_0$.
    \end{compactenum}
    The bounds are tight: For every $n \ge d \ge 0$ and $m \ge 0$, there is a
    $d$-ball that attains the upper bounds for all $k$ simultaneously. 
\end{thm}

\begin{proof}[Proof of Tightness]
We borrow a construction that we will see again in
Section~\ref{sec:UBTMS} and apply it to two well-chosen cyclic $(d-1)$-polytopes
$C_1,C_2$ used by Matschke, Pfeifle and Pilaud~\cite[Theorem~2.6]{MPP} (compare
Theorem~\ref{thm:M-neighborly}) with $f_0(C_1) = n$ and $f_0(C_2) = m$. Let
\[
    C \ = \ \conv\bigl( C_1 \times \{0\} \cup C_2 \times \{1\} \bigr) \subset
    \R^d \times \R
\]
be the \emph{Cayley polytope} of $C_1$ and $C_2$. Using Theorem~2.6 of \cite{MPP} the Cayley polytope over $C_1$ and $C_2$ may be chosen in such a way that it has no non-face of dimension $<\frac{d}{2}-1$, and such that every non-face of dimension $\frac{d}{2}-1$ is supported in either $C_1$ or $C_2$.
By construction, $C_1$ and $C_2$ are the only non-simplex
faces of $C$. Let us triangulate $C_2$ without new vertices, and such that
there are no non-faces of dimension $\le\frac{d}{2}-1$, and let $\Delta$
be the simplicial complex obtained from $\partial C$ by deleting $
C_1$. Then $\Delta$ is a triangulated $d$-ball with full boundary $\partial
\Delta = \partial C_1$,  and the conditions in
(a) and (b) are met and hence yields an example that attains the upper
bounds.   \end{proof}

\section{Estimating the algebraic contribution}\label{sec:loct}

We discuss three techniques for bounding the $h$-vector entries
$h_k(\Delta,\Gamma)$ based on bounds on the algebraic contribution
$\LT(\Delta,\Gamma)$. 

The first method is based on the idea of a \emph{change of presentation}: We
consider presentations of $\FM[\Delta,\Gamma]$ as quotients of monomial ideals
$I/J$ where $I$ is simpler in structure than $\I_{\Gamma}$.  A particularly
important candidate is the \emph{nerve ideal} $\ch$ that arises from coverings
of $\Gamma$ by full subcomplexes. The nerve ideal can be analyzed in terms of
nerve complex of that covering. This in particular allows us to
interpolate between full and general subcomplexes $\Gamma$. As a special case,
we recover Lemma~\ref{lem:UBT-CM}.

The second method is based on a more delicate trick. It uses a formula that
integrates over the $h$-numbers of subcomplexes to the $h$-vector of the total
complex. We then employ a lemma of Kalai--Stanley for an upper bound on the
local contributions to obtain the desired bounds. The second method has an
interesting refinement that we describe in Section~\ref{ssec:rel_diff}.  In
particular, we find an interesting reverse isoperimetric inequality that
considerably improves on, and is substantially different from, all known
bounds in the area.

Finally, we discuss the role of relative shellability, a combinatorial/geometric method
that can be used to give bounds on $h$-numbers in our setting.

\subsection{Estimates via change of presentation}
\label{ssec:change_presentation}
The idea of this section is that if $\FM[\Delta,\Gamma]$ has a ``nice''
presentation as a quotient, then this presentation can be used to estimate
the algebraic contribution of $\RC = (\Delta,\Gamma)$.  We will see an interesting connection
to poset topology when attempting to characterize the case of equality
and an application of Borsuk's Nerve Lemma (in its filtered version
due to Bj\"orner).

\enlargethispage{-3mm}

Let $M$ be a module over $\k[\x]$. We write $M\varpropto I$ for a
monomial ideal $I \subseteq \k[\x]$ if there is a monomial ideal $J \subseteq
I$ such that $M \cong I / J$ as finely graded modules. 

\begin{lem}\label{lem:inc}
    Let $M \cong I/J$ be a module over $\k[\x]$ for some monomial ideals
    $J\subseteq I$ in $\k[\x]$.  For a sequence $\Theta =
    (\theta_1,\dots,\theta_\ell)$ of linear forms we then have
    \[
        \dim_\k(M/\Theta M)_k \ \le \ \dim_\k (I/\Theta I)_k
    \] 
    for all $k\ge 0$. Moreover, if tensoring with $\k[\x]/\Theta$ preserves
    exactness of $0 \rightarrow J \hookrightarrow I \twoheadrightarrow M
    \rightarrow 0$, then equality holds for some $k_0$ if and only if $M_k
    \cong \I_k$ for all $k \le k_0$.
\end{lem}

\begin{proof}
    By assumption, we have a short exact sequence
\[
        0 \longrightarrow J \longrightarrow
        I \longrightarrow M \longrightarrow 0.
\]
    The first claim follows from the fact that tensoring with 
    $\k[\x]/\Theta\k[\x]$ is a right-exact functor and thus
    $I/\Theta I \twoheadrightarrow M/\Theta M$ is a (graded) surjection.
    Assume now that 
    \[
        0 \longrightarrow J/\Theta J
        \longrightarrow I/\Theta I \longrightarrow M/\Theta M\longrightarrow
        0
    \]
    is exact. In all nontrivial cases $\Theta$ is at best a partial \lsop~for $J$.
    Hence, if $(J/\Theta J)_{k_0} = 0$, then $J$ has no generators in degrees $\le k_0$.
\end{proof}

This result subsumes Lemma~\ref{lem:UBT-CM}: If $M = \k[\Delta]$ is the
Stanley--Reisner ring of a $(d-1)$-dimensional Cohen--Macaulay simplicial
complex on $n$ vertices, then $M \varpropto I$ for $I = \k[x_1,\dots,x_n]$. For a
regular \lsop\ $\Theta$, we infer from Lemma~\ref{lem:inc} that
\[
        h_k(\Delta) \ = \ \dim_\k(M/\Theta M)_k \ \le \ \dim_\k(I/\Theta
        I)_k \ = \ \dim_\k\k[y_1,\dots,y_{n-d}]_k \ = \ \binom{n-d+k-1}{k}.
\]

To conclude tightness in Lemma~\ref{lem:inc}, we need to decide whether a
sequence is regular for all modules in a given exact sequence. To this end, we can use the following  well-known observation. Recall that
$\Theta_m$ is the restriction of the sequence $\Theta$ to the first $m$
elements.  
 
\begin{prp}\label{prp:exact}
Let $R$ be any ring. Let \[D \rightarrow C \rightarrow B \xrightarrow{\varphi}
A \rightarrow 0\] denote a exact sequence of $R$-modules, and let
$\Theta=(\theta_1,\dots,\theta_\ell)$ denote a family of elements of $R$.
Assume that for every $1 \le m \le  \ell$, $\varphi$ induces a surjection
\[
(\Theta_{m-1} B:\theta_m)/\Theta_{m-1} B \longtwoheadrightarrow (\Theta_{m-1}
A:\theta_m)/\Theta_{m-1} A.
\]
Then we have an exact sequence
\[D/\Theta D \rightarrow C/\Theta C \rightarrow B /\Theta B \xrightarrow{\overline{\varphi}} A /\Theta A \rightarrow 0.\]
\end{prp}


In the situation of Proposition~\ref{prp:exact}, we also say that $\varphi$ is a 
\Defn{$\Theta$-surjection}; if the maps between the annihilator modules are even isomorphisms, then we call the map a \Defn{$\Theta$-isomorphism}.

Lemma~\ref{lem:UBT-CM} compares enumerative properties of $\Delta$ to those
of the much simpler complex $\K_n$. This is possible because both are
Cohen--Macaulay. For simplicial complexes, this approach suffices. 
In order to use a reasoning similar to Lemma~\ref{lem:UBT-CM} for relative
complexes, we will use a cover of the subcomplex $\Gamma$ by full
subcomplexes.

\begin{dfn}\label{dfn:arrment}
    Let $\Delta$ be a simplicial complex. An \Defn{arrangement of full
    complexes}, \Defn{full arrangement} for short, is a finite collection
    $\scr{G}$ of full subcomplexes of $\Delta$. 
\end{dfn} 

For an arrangement $\scr{G}$ of complexes,
the collection
\[
    \mc{P}(\scr{G}) \ := \ \Bigl\{ \bigcap S : S \subseteq \scr{G} \Bigr\}
    \cup \{ \Delta \}
\]
together with the partial order given by \emph{reverse} inclusion is the
\Defn{intersection poset} of $\scr{G}$. This is a poset with minimal element
$\hat{0} := \Delta$ and maximal element  $\hat{1} := \bigcap\scr{G}$. Note
that any $\Gamma \in \mc{P}(\scr{G})$ is a full subcomplex of~$\Delta$. The
\Defn{support} of $\scr{G}$ is the subcomplex 
\[
        \cfs{G} \ := \ \bigcup_{\Gamma\in \scr{G}} \Gamma \ \subseteq \ \Delta.
\]
This covering of $\cfs{G}$ by full subcomplexes can be used to obtain a simple
presentation of $\FM[\Delta,\cfs{G}]$.
\begin{dfn}
    For an arrangement $\scr{G}$ of complexes of $\Delta$ we define the
    \Defn{nerve ideal} as the monomial ideal 
    \[
        \ch[\Delta,\scr{G}] \ := \ 
        \langle \x^\tau : \tau \not\subseteq \V(\Gamma) \text{
        for all } \Gamma \in \scr{G} \rangle \ \subseteq \ \k[\x].
    \]
\end{dfn}

For $\Gamma \in \scr{G}$, the smallest simplex containing $\Gamma$ is given by
$\K_{\V(\Gamma)}$. The \Defn{coarse nerve} of $\scr{G}$ is the simplicial 
complex 
\[
    \mathfrak{N}[\Delta,\scr{G}] \ := \ \bigcup\{ \K_{\V(\Gamma)} : \Gamma \in
    \scr{G} \}. 
\]
The nerve ideal $\ch[\Delta,\scr{G}]$ then is the Stanley--Reisner ideal of
the coarse nerve.  The connection to $(\Delta,\cfs{G})$  is the following.

\begin{prp}\label{prp:upb}
    Let $\scr{G}$ be a full arrangement of $\Delta$. Then
    \[
        \FM[\Delta,\cfs{G}] \ \cong \ (\ch[\Delta,\scr{G}] + \I_\Delta)/\I_\Delta \ \cong \ 
        \ch[\Delta,\scr{G}]/(\I_\Delta \cap \ch[\Delta,\scr{G}]).
    \]
\end{prp}
\begin{proof}
    Let $\x^\a$ be a monomial and $\sigma = \supp(\x^\a)$. Then $\x^\a = 0$ in
    both $\FM[\Delta,\cfs{G}]$ and $(\ch[\Delta,\scr{G}]+ \I_\Delta)/\I_\Delta$ if $\sigma
    \not\in \Delta$. Thus, let us assume that $\sigma \in \Delta$.
    Now $\FM[\Delta,\cfs{G}]_\a \neq 0$ iff $\sigma \not\in \cfs{G}$, which is
    the case if and only if $\sigma \not\in \Gamma$ for all $\Gamma \in
    \scr{G}$. Since all subcomplexes in $\scr{G}$ are full, this is
    equivalent to $\sigma \not\subseteq \V(\Gamma)$ for all $\Gamma \in
    \scr{G}$. This, in turn, is equivalent to
    $((\ch[\Delta,\scr{G}]+ \I_\Delta)/\I_\Delta)_\a \neq 0$.
\end{proof}

For an arrangement of full subcomplexes a good relative complex to compare
$\RC = (\Delta,\Gamma)$ to is $(\K_n,\mathfrak{N}[\Delta,\scr{G}])$.
Proposition~\ref{prp:upb} and Lemma~\ref{lem:inc} then imply immediately:

\begin{thm}\label{thm:upb}
    Let $\Delta$ be a simplicial complex, and let $\scr{G}$ be a full
    arrangement of subcomplexes. Then $\FM=\FM[\Delta,\cfs{G}] \varpropto
    \ch=\ch[\Delta,\scr{G}]$ and for every collection of linear forms
    $\Theta$
    \[
        \dim_\k (\FM/\Theta \FM)_k \ \le \ \dim_\k (\ch/\Theta \ch)_k 
    \]
    for all $k$. If $\Theta$ is a \lsop\ for $M$, and the surjection $\ch \twoheadrightarrow \FM$ 
    is a $\Theta$-surjection, then the following
    are equivalent:
    \begin{enumerate}[\rm (i)]
    \item Equality holds for some $k_0$;
    \item $\FM_k \cong (\ch)_k$ for all $k\le k_0$;
    \item $\I_\Delta \cap \ch$ is generated in degrees $> k_0$;
    \item every non-face of $\Delta$ of dimension $< k_0$ is supported on
        $\V(\Gamma)$ for some $\Gamma\in\scr{G}$.
    \end{enumerate}
\end{thm}

To help decide whether $\ch \twoheadrightarrow \FM$ is a $\Theta$-surjection, it is useful to keep some simple tricks in
mind. For instance, if $\Theta$ is a regular sequence for $\FM$, then $\Theta_{m-1} \FM:\theta_m/\Theta_{m-1} \FM \equiv
0$, so that $\ch \twoheadrightarrow \FM$ is trivially a $\Theta$-surjection. 
This is in particular applicable
if $\FM$ is Cohen--Macaulay.

\begin{prp}
    A \lsop\ $\Theta$ of length $\ell$ induces a $\Theta$-surjection $\ch
    \twoheadrightarrow \FM[\RC]$ if the $(\ell-1)$-skeleton
    $\RC^{(\ell)}$ is Cohen--Macaulay.
\end{prp}

\begin{proof} This follows since every \lsop\ $\Theta$ of length $\le \ell$ is regular if the $(\ell-1)$-skeleton is
    Cohen--Macaulay by a result of Hibi, cf.\ \cite[Corollary~2.6]{Hibi}.\end{proof}

For a more general criterion, we consider complexes whose skeleta are Buchsbaum. 

\begin{thm}\label{thm:reduction_to_topology} 
    Assume that the $(\ell-1)$-skeleta $(\K_{\V(\Delta)},
    \mathfrak{N}[\Delta,\scr{G}])^{(\ell)}$ and $(\Delta,\cfs{G})^{(\ell)}$
    are Buchsbaum, and let $\Theta$ be any \lsop\ of length $\ell$.  Then we
    have a $\Theta$-surjection resp.\ $\Theta$-isomorphism if for every face
    $\sigma$ of $\Delta$, the embedding
\begin{equation}\label{eq:em}
        \Lk(\sigma,(\Delta,\cfs{G})) \longhookrightarrow
        \Lk(\sigma,(\K_{\V(\Delta)},\mathfrak{N}[\Delta,\scr{G}]))
\end{equation}
    induces a surjection (resp.\ isomorphism) of cohomology groups up to
    degree $\ell-\dim \sigma -2$.
\end{thm}

\begin{proof}
The basic idea is that the modules $\Theta_{m-1}\ch:\theta_m/\Theta_{m-1} \ch$ can be 
written as cokernels in short exact sequences of cohomology groups of Koszul complexes: We have
\[
\renewcommand{\labelstyle}{\textstyle}
\xymatrix@C-11pt{
0\ar[r]^{} & H^{m-2}(\Theta_{m-1};M) \ar[r]^{} 
  &H^{m-1}(\Theta_{m};M) \ar[r]^{} 
  & \Theta_{m-1}M:\theta_m/\Theta_{m-1} M \ar[r]^{} 
  & 0}
\]
for $\k[x]$-modules $M$ as in Corollary~\ref{cor:qm}. For Buchsbaum complexes, these homology modules are
determined in terms of local cohomology of $\ch$ and $\FM$, and
by the connection between the $\Z^n$-graded \v{C}ech complex and homology of links (exploited in Hochster's formula), we conclude that
if for all $\sigma\in\Delta$, the embedding~\eqref{eq:em}
induces a surjection of cohomology groups up
to degree $\ell-\dim \sigma -2$, then we also have a surjection on the level of local cohomology modules of $\ch$
and $\FM$ up to dimension $\ell$.

Let now $m\le \ell$. Observe that the key to Lemma~\ref{lem:koszulRanks} is a
quasi-isomorphism of chain complexes
\[
   \tau^\ell(K^\bullet(\Theta_m;\FM)\otimes_{\k[\x]} C^\bullet
   (\FM))\xrightarrow{\ \sim\ }\tau^\ell K^\bullet(\Theta_m;\FM)
\]
in the derived category $D(\k[\x])$, cf.\ \cite{Schenzel81,Schenzel82}.
Here the former chain complex is a chain complex with trivial differentials,
with
\[
    C^i(\FM)=\left\{\begin{array}{ll}H^i(\FM) &\ \text{if\ } i\in
        \{0,\dots,\dim \FM-1\}\text{\ and}\\
          0 &\ \text{otherwise}.
                  \end{array}\right.
\]
so that the complex $C^\bullet (\FM)$ is the exact chain complex of local
cohomology modules of $\FM$, and $\tau^\ell$ denotes the truncation of a
chain complex in degree $\ell$.  Analogously, we have a quasiisomorphism
\[
    \tau^\ell(K^\bullet(\Theta_m;\ch)\otimes_{\k[\x]} C^\bullet
    (\ch))\xrightarrow{\ \sim\ }\tau^\ell K^\bullet(\Theta_m;\ch).
\] 
It follows that a surjection
on the level of local cohomology modules of $\ch$ and $\FM$ induces a
surjection on the level of Koszul
cohomology.

To conclude the desired surjection of modules  $(\Theta_{m-1}\ch:\theta_m)/\Theta_{m-1}\ch\longtwoheadrightarrow(\Theta_{m-1}\FM:\theta_m)/\Theta_{m-1}\FM$, consider
\[
\renewcommand{\labelstyle}{\textstyle}
\xymatrix@C-11pt{
0\ar[r]^{} & H^{m-2}(\Theta_{m-1};\ch) \ar[d]^{} \ar[r]^{} 
  &H^{m-1}(\Theta_{m};\ch) \ar[d]^{} \ar[r]^{} 
  & \Theta_{m-1}\ch:\theta_m/\Theta_{m-1} \ch \ar[d]^{} \ar[r]^{} 
  & 0 \\
0\ar[r]^{} & H^{m-2}(\Theta_{m-1};\FM) \ar[r]^{} 
  &H^{m-1}(\Theta_{m};\FM) \ar[r]^{} 
  & \Theta_{m-1}\FM:\theta_m/\Theta_{m-1} \FM \ar[r]^{} 
  & 0 } 
\]
and the Snake lemma. The claim for the isomorphism follows analogously.
\end{proof}

This motivates us to notice a beautiful relation to Borsuk's Nerve Lemma 
\cite{Borsuk}: \emph{Not all full arrangements are created equal}.

Let us call a full arrangement $\scr{G}$ in $\Delta$ an \Defn{$\ell$-good} cover 
if, for every 
subset $\{\Gamma_1,\dots,\Gamma_t\}$ of $t$ elements of $\scr{G}$, the relative complex 
$(\Delta,\Delta\cap\bigcap_{i=1}^t \Gamma_i)$ is $(\ell-t)$-acyclic, that is, its homology vanishes up to dimension $\ell-t$. We call $\scr{G}$ 
\Defn{$\ell$-magnificent} (w.r.t.\ $(\Delta,\cfs{G})$) if, for every face 
$\sigma$ of $\Delta$, the restriction of $\scr{G}$ to $\Lk(\sigma,\Delta)$ is 
{$(\ell-\dim\sigma-1)$-good}. We have the following application of the Nerve 
Lemma (in its generalization due to Bj\"orner, cf.\ \cite{NFH, BWW}).
\begin{thm}\label{thm:magnificent}
Assume that $\scr{G}$ is an $\ell$-magnificent cover, and that 
$(\K_{\V(\Delta)},\mathfrak{N}[\Delta,\scr{G}])^{(\ell)}$ and $(\Delta,\cfs{G})^{(\ell)}$ 
are Buchsbaum. Then for every face $\sigma$ of $\Delta$, the embedding
\[\Lk(\sigma,(\Delta,\cfs{G})) \longhookrightarrow 
\Lk(\sigma,(\K_{\V(\Delta)},\mathfrak{N}[\Delta,\scr{G}]\, ))\] induces an isomorphism of relative homology up to dimension $\ell-\dim \sigma -2$, and a surjection in degree $\ell-\dim \sigma -1$. In particular, $\ch \twoheadrightarrow \FM$ is 
a $\Theta$-isomorphism for every \lsop\ $\Theta$ of length~$\ell$, and a $\Theta$-surjection for every \lsop\ $\Theta$ of length~$\ell+1$. 
\end{thm}

\begin{ex}\label{ex:cm}
If $\Delta$ is Cohen--Macaulay, and $\scr{G}$ is a collection of disjoint Cohen--Macaulay subcomplexes of $\Delta$ of the same dimension as $\Delta$, then the cover is $d$-magnificent.
\end{ex}

We will see some more interesting examples and an application when investigating Minkowski sums of polytopes, compare also Theorem~\ref{thm:twosummands}.
We finally record a simple trick to compute $\dim_\k (\I/\Theta \I)_k$ for Lemma~\ref{lem:inc} using a dual form of Schenzel's Formula.

\begin{thm}\label{thm:RSI}
    Let $\Gamma \subseteq \K_n$ be a simplicial complex with Stanley--Reisner
    ideal $\I_\Gamma$. For $m \ge \ell$ assume that the relative complex of 
    $(m-1)$- and $(\ell-1)$-skeleta $(\K_n^{(m)},\Gamma^{(\ell)})$ is
    Buchsbaum and let $\FM = \FM[\K_n^{(m)},\Gamma^{(\ell)}]$ be the
    corresponding face module. If $\Theta$ is a \lsop\ for $\FM$ 
    then for all $0 \le j \le \ell$ 
    \[
        \dim_\k (\I_\Gamma/\Theta \I_\Gamma)_j \ = \
        [(1-t)^{m}\Hilb(\I_\Gamma,t)]_j -
        \binom{m}{j} \sum_{i=0}^{j-1}(-1)^{j-i}
        \rBetti_{i-2}\bigl(\Gamma^{(\ell)}\bigr),
    \] 
    where $[(1-t)^{m}\Hilb(\I_\Gamma,t)]_j$ denotes the coefficient of $t^j$
    in $(1-t)^{m}\Hilb(\I_\Gamma,t)$.
\end{thm}

If $(\K_n^{(m)},\Gamma^{(\ell)})$ is Buchsbaum and hence locally
Cohen--Macaulay, we necessarily have $\ell\le m\le \ell+1$ by Corollary
\ref{cor:cmincm}. 
  
\begin{proof}
Notice that $\K_n^{(m)}$ is Cohen--Macaulay and by the long exact sequence in
relative homology
\[
    \rHom_{i-2}(\Gamma^{(\ell)}) \ \cong \
    \rHom_{i-1}(\K_n^{(m)},\Gamma^{(\ell)}) \quad \text{ for all } i-1<
    \ell-1.
\]
Hence, by Theorem~\ref{thm:RS}, we obtain 
\[
        [(1-t)^m \Hilb(\FM,t)]_j  \ = \ 
         [\Hilb(\FM/\Theta \FM,t)]_j  + 
        \binom{m}{j} \sum_{i=0}^{j-1}(-1)^{j-i}
        \rBetti_{i-2}(\Gamma^{(\ell)}).
\]
Passing to the $(\ell-1)$-skeleton changes the ideal $\I_\Gamma$ in degrees $>
\ell$, so that $(\I_\Gamma)_{\le \ell}\cong M_{\le \ell}$. 
The formula follows.
\end{proof}

\begin{cor}\label{cor:RSI}
    Let $\Gamma \subseteq \K_n$ be any simplicial complex.  Assume that for $m 
    \ge \ell$ the relative complex
    $(\K_n^{(m)},\Gamma^{(\ell)})$ is Cohen--Macaulay. If $\Theta$ is a full
    \lsop\ for $\FM[\K_n^{(m)},\Gamma^{(\ell)}]$, then for all $0\le j \le
    \ell$
    \[
        \dim_\k (\I_\Gamma/\Theta \I_\Gamma)_j \ = \ [(1-t)^m
        \Hilb(\I_\Gamma,t)]_j.
    \] 
\end{cor}

To summarize, we reduced the problem of bounding the $h$-numbers, or
equivalently the problem of bounding
$[(1-t)^{d}\Hilb(\FM[\Delta,\Gamma],t)]_j$, to the problem of bounding
$[(1-t)^{\ell}\Hilb(\I,t)]_j$ for some Stanley--Reisner ideal with
$\FM[\Delta,\Gamma]\varpropto \I$. The full power of this approach is seen in
combination with Theorem~\ref{thm:upb}.
Let us close with a simple observation that will close the cycle by computing 
$[\Hilb(\I,t)]_j$ as a straightforward application of the inclusion-exclusion principle on the involved non-face ideals of cliques. 

\enlargethispage{5mm}

\begin{prp}\label{prp:bi}
    Let $\Delta$ be a pure $(d-1)$-dimensional simplicial complex, and let
    $\scr{G}$ be a full arrangement. Then 
    \[
        \dim_\k \ch[\Delta,\scr{G}]_k \ = \  \sum_{p\in\mc{P}(\scr{G})}
        \mu_{\mc{P}(\scr{G})}(\Delta,p) \binom{f_0(p)+k-1}{k}
    \]
    where $\mu_{\mc{P}(\scr{G})}$ is the M\"obius function of the
    intersection poset. \qed
\end{prp}

%
%

\subsection{Arrangements of Cohen--Macaulay complexes}\label{ssec:arrCM}
The estimates via change of presentation enable us to extend the results
from full CM complexes to full arrangements of CM complexes.

\begin{thm}\label{thm:arrCM}
    Let $\Delta$ be a $(d-1)$-dimensional Cohen--Macaulay complex and
    $\scr{G} = \{ \Gamma_1,\dots,\Gamma_m\}$ a full arrangement of $m$
    pairwise disjoint codimension one Cohen--Macaulay subcomplexes of
    $\Delta$. Then for $0 \le k \le d$
    \[
        h_k(\Delta,\cfs{G}) \ \le \ \binom{f_0(\Delta)-d+k-1}{k} -
        \sum_{i=1}^m \binom{f_0(\Gamma_i)-d+k-1}{k}+(m-1)\binom{-d+k-1}{k}
    \]        
    Equality holds for some $k_0$ if and only if every non-face of $\Delta$ of
    dimension $< k_0$ is supported on some $\Gamma_i$.
\end{thm}

Let us write
\[
    \mathbf{1}_{k\ge a} \  = \ 
    \begin{cases} 
1 &  \text{ if } k \ge a \text{ and}\\
0 &  \text{ otherwise.}
    \end{cases}
\]

\begin{proof}
Set $\Gamma := \cfs{G} = \Gamma_1\cup\cdots\cup \Gamma_m$.  
To begin with, we observe that for a vertex $v \in \Gamma_i$, we have
$\Lk(v,(\Delta,\Gamma)) = \Lk(v,(\Delta,\Gamma_i))$. We conclude from
Corollary~\ref{cor:cmincm} that $\Lk(\Delta,\Gamma,v)$ is Cohen-Macaulay for
all vertices $v \in \Delta$. Hence, by Theorem~\ref{thm:topb}, $\RC =
(\Delta,\Gamma)$ is Buchsbaum and $\FM = \FM[\RC]$ a Buchsbaum module. We can
therefore use the relative Schenzel formula (Theorem~\ref{thm:RS}) to bound
$h_k(\RC)$ in terms of the topological contribution $\GT_i(\RC)$ and the
algebraic contribution $\LT_i(\RC)$.

\emph{The topological contribution.} From the Cohen--Macaulayness of the
complexes $\Gamma_i$ it follows that 
\[
    \rBetti_0(\Gamma) \ = \ m-1 \quad\text{and}\quad
    \rBetti_{d-2}(\Gamma) \ = \ \sum_i\rBetti_{d-2}(\Gamma_i)
\]
and $\rBetti_i(\Gamma)=0$ for all other $i$. The long exact sequence in relative
homology
\[ 
 \cdots\; \longrightarrow \, \rHom_{i+1}(\Delta, \Gamma) \,
 \longrightarrow \, \rHom_{i}( \Gamma) \, \longrightarrow \,
 \rHom_{i}(\Delta) \, \longrightarrow \,
 \rHom_{i}(\Delta, \Gamma) \, \longrightarrow \cdots
\]
splits into short sequences and we deduce 
\[
    \rBetti_1(\RC)\   = \ m-1
    \quad\text{and}
    \quad
    \rBetti_{d-1}(\RC) \ = \ \rBetti_{d-1}(\Delta)+\sum_i\rBetti_{d-2}(\Gamma_i)
\]
and $\rBetti_i(\RC)=0$ otherwise. Hence, for $k \le d$, the topological contribution in
Theorem~\ref{thm:RS} is
\[
    \GT_k(\RC) \ = \
    (-1)^{k-2}
    (m-1)\mathbf{1}_{k\ge 3}\binom{d}{k}.
\]

\emph{The algebraic contribution.} The nerve ideal $\ch = \ch[\Delta,\Gamma]$ is the
Stanley--Reisner ideal of the coarse nerve
$
\mathfrak{N} \ = \ \bigcup_{i=1}^m \K_{\V(\Gamma_i)} \subseteq \K_{n}.
$ 
Since each $\Gamma_i$ is of  dimension $d-2$, $\mathfrak{N}$ is the disjoint union
of simplices of dimension $\ge d-2$. Hence, the relative complex
$(\K_n^{(d)},\mathfrak{N}^{(d)})$ is Cohen--Macaulay by Corollary~\ref{cor:cmincm}. The homology of $\mathfrak{N}$ is concentrated in degree 0 with
$\rBetti_0(\mathfrak{N}) = m-1$.  Therefore, we obtain
for a \lsop\ $\Theta$ for $\FM$
\begin{align*}
\LT_k(\RC) &\ = \  \dim_{\k} [\FM/\Theta \FM]_k \\
    & \ \le \  \dim_\k (\ch/\Theta \ch)_k & \text{(by Theorem~\ref{thm:upb})}\\
    & \ = \  [(1-t)^{d}\Hilb(\ch,t)]_k -
        \binom{d}{k} \sum_{i=0}^{k-1}(-1)^{k-i}
        \rBetti_{i-2}(\mathfrak{N}) & \text{(by Theorem~\ref{thm:RSI})}\\
    & \ = \  \sum_{p \in
    \mc{P}(\scr{G})}
    \mu_{\mc{P}}(\Delta,p) \binom{f_0(p)-d+k-1}{k}-(-1)^{k-2}
    (m-1)\mathbf{1}_{k\ge 3}\binom{d}{k}. & \text{(by
    Proposition~\ref{prp:bi})}
\end{align*}
The intersection poset is $\mc{P} = \scr{G} \cup \{ \Delta, \emptyset \}$ and
hence the M\"obius function is given by 
\[
    \mu_{\mc{P}(\scr{G})}(\Delta,p) \ = \ \
    \begin{cases}
                1 &  \text{if }p=\Delta,\\ 
               -1 &  \text{if }p=\Gamma_i  \text{ and}\\
               m-1 &  \text{if }p=\emptyset.
    \end{cases}
\]
Putting the computation of $\GT(\RC)$ and the bound on $\LT(\RC)$ together
yields the bound on $h_k(\RC)$.

\emph{Case of Equality.} Equality can hold for some $k_0$ if and only if it
holds for the algebraic contributions.  The equality is then this of
Theorem~\ref{thm:upb}.
\end{proof}

%


The following result interpolates between the two extreme situations of
Theorem~\ref{thm:arrCM} and the case that $\cfs{G}$ is itself full.
\renewcommand\subset{\mathbf{\color{red}\blacksquare}}

\begin{thm}
    Let $\Delta$ be a $(d-1)$-dimensional Cohen--Macaulay complex and
    $\scr{G} = \{ \Gamma_1,\dots,\Gamma_m\}$ an arrangement of $m$ pairwise
    disjoint, codimension one CM subcomplexes. Assume that for every subset
    $S$ of $[m]$ with $|S|\le \ell,\ \ell>1$, the complex $\cup_{i\in S}
    \Gamma_i$ is full in $\Delta$. Then we have, for all $k\in [d]$,
    \[
        h_k(\Delta,\cup\Gamma_i)\le \sum_{p\in \mc{P}(\scr{G})} \mu_p
        \binom{f_0(p)-d+k-1}{k} + (-1)^{k-2} (m-1)\mathbf{1}_{k\ge
        3}\binom{d}{k}+(-1)^{k+\ell+1} \binom{m-1}{\ell}\mathbf{1}_{k\ge
        \ell+2}\binom{d}{k} 
    \]
    for $\scr{G}=\{\cup_{i\in S} \Gamma_i: S \subseteq [m], |S|\le
    \ell\}\cup\{\emp\}$ 
and
\[\mu_p  \ = \ \mu_{\mc{P}(\scr{G})}(\Delta,p) \ = \ \left\{\begin{array}{cl}
                1 &\  \text{if }p=\Delta,\\ 
               (-1)^{|S|} &\  \text{if }p=\cup_{i\in S} \Gamma_i \text{ and}\\
              - \sum^{\ell}_{j=0} (-1)^j \binom{m}{j}&\  \text{if }p=\emptyset.
               \end{array}
\right.\]
\end{thm}

\begin{proof}The proof is analogous to the proof of Theorem~\ref{thm:arrCM}; the topological part is unchanged, and it remains only to estimate $\rBetti_{i-2}(\mathfrak{N})$. But $\mathfrak{N}$ is homotopy equivalent to the $(\ell-1)$-skeleton of a simplex on $m$ vertices, so that the claim follows.
\end{proof}
We conclude with the proofs of Theorems~\ref{thm:iso_mfds} and 
\ref{thm:comb_iso_full}.

\begin{proof}[\textbf{Proof of Theorem~\ref{thm:comb_iso_full}}]
    By Reisner's Theorem~\ref{thm:SR}, $\Delta$ and $(\Delta,\partial
    \Delta)$ are Cohen--Macaulay. Therefore, Claim (a) is a special case of
    Theorem~\ref{thm:arrCM} with $\scr{G} = \{ \partial \Delta \}$. To see
    Claim (b), notice that by the Dehn--Sommerville relations~\ref{lem:DS},
    we have $h_{k}(\Delta,\partial \Delta)=h_{d-k} (\Delta)\ \text{for
    all}\ 0\le k\le d$, and to $\Delta$ we can apply the standard upper bound theorem. Therefore, the claim follows with characterization
    of equality in Theorem~\ref{thm:arrCM} and Lemma~\ref{lem:UBT-CM}.
\end{proof}

We close with the proof for the combinatorial isoperimetric problem for
manifolds.

\begin{proof}[\textbf{Proof of Theorem~\ref{thm:iso_mfds}}]
By Theorem~\ref{thm:topb}, $\RC = (M,B)$ is a Buchsbaum complex and $\FM =
\FM[\RC]$ a Buchsbaum module. To apply Theorem~\ref{thm:RS}, it remains for
us to bound $h_k(\RC)$ in terms of the topological contribution $\GT_k(\RC)$
and the algebraic contribution $\LT_k(\RC)$.
The topological contribution depends only on the relative Betti numbers $(M,B)$ and hence
is
\[
    \GT_k(\RC) \ = \ \binom{d}{k}\sum_{i=0}^{k-1}(-1)^{k-i}\rBetti_{i-1}(\RC).
\]

As for \emph{the algebraic contribution:} The nerve ideal $\ch = \ch[M,B]$ is
the Stanley--Reisner ideal of the simplicial complex $\K_n$, ${n} \ge d-1$.
Hence, the relative complex $(\K_{m+n}^{(d)},\K_n^{(d)})$ is Buchsbaum by
Corollary~\ref{cor:cmincm} and Theorem~\ref{thm:topb}, and in fact Cohen--Macaulay 
since the homology is concentrated in degree $d-1$.  Therefore,
by Theorems~\ref{thm:upb} and~\ref{thm:RSI} and Proposition~\ref{prp:bi},
\[
    \LT_k(\RC) \ \le \ \binom{m+n-d+k-1}{k}\ -\ \binom{n-d+k-1}{k}. \qedhere
\]
\end{proof}

\subsection{Local-to-global estimates} \label{ssec:primes} 
The purpose of this section is to provide iterative inequalities of the type
given in Proposition~\ref{prop:UBT-g}. We will provide the desired bounds for
$h_k(\RC)$ by combining an integration formula for multivariate formal power series
 with an observation of Stanley and Kalai and a careful use of the fullness property. 
 
\begin{lem}[Formula for local $h$-vectors, cf.\ \cite{mcmullen1970},{\cite[Lemma 2.3]{Swartz}}]\label{lem:hlk}
    For a pure relative simplicial complex $\RC = (\Delta,\Gamma)$ of dimension
    $d-1$ on $[n]$ 
    \[
        \sum_{i=1}^n h_k(\Lk(i,\RC)) \ = \ (k+1)h_{k+1}(\RC) + (d-k)
        h_k(\RC)
    \]
    for all $0 \le k \le d$.
\end{lem}
\begin{proof}
    For $\alpha=(\alpha_1,\dots,\alpha_n)\in \Z^{n}$ and $1 \le i \le
    n$ let us write $\alpha {\setminus} i =
    (\alpha_1,\dots,\alpha_{i-1},0,\alpha_{i+1},\dots,\alpha_n)$.
    Let us abbreviate $\frac{\partial}{\partial \t} := \frac{\partial}{\partial
    t_1} + \cdots + \frac{\partial}{\partial t_n}$. For the fine graded
    Hilbert series of the face module $\FM = \FM[\Delta,\Gamma]$ we compute
    \begin{align*}
        \frac{\partial}{\partial \t}\sm \Hilb(\FM,\t) \ & :=  \ 
        \sum_{\supp(\alpha )\in \RC} \frac{\partial}{\partial \t}\sm  \t^\alpha \ = \ 
\sum_{i=1}^n \sum_{\supp(\alpha )\in \RC} \frac{\partial}{\partial
        t_i} \t^\alpha \\ \ &= \ 
        \sum_{i=1}^n \sum_{\supp(\alpha )\in \St(i,\RC)} \alpha_i
        t_i^{\alpha_i-1} \t^{\alpha {\setminus} i} \ = \
        \sum_{i=1}^n \frac{1}{(1-t_i)^2} \sum_{\supp(\alpha )\in \Lk(i,\RC)}
         \t^{\alpha {\setminus} i}.
    \end{align*}
    If we now specialize $t_1 = \cdots = t_n = t$, we obtain
    \[
    \frac{\mathrm{d}}{\mathrm{d} t} \Hilb(\FM,t) \ = \ 
        \frac{1}{(1-t)^2}
        \sum_{i=1}^n \Hilb(\FM[\Lk(i,\RC)],t) \ = \
        \sum_{i=1}^n \frac{\sum_{k=0}^{d-1}
        h_k(\Lk(i,\RC))t^k}{(1-t)^{d+1}}
    \]
    where $\Hilb(\FM,t)$ is the coarse Hilbert series. On the other hand we
    can directly compute the derivative of $\Hilb(\FM,t)$ as
    \begin{align*}
        \frac{\mathrm{d}}{\mathrm{d}t} \Hilb(\FM,t)  \ &= \
        \frac{\mathrm{d}}{\mathrm{d} t} \frac{\sum_{k=0}^{d}
        h_k(\RC))t^k}{(1-t)^{d}} \ = \ 
        \frac{\sum_{k=0}^{d} kh_k(\RC))t^{k-1}}{(1-t)^{d}} +
        \frac{\sum_{k=0}^d d\,h_k(\RC))t^{k}}{(1-t)^{d+1}} \\
& = \  \frac{\sum_{k=0}^{d-1} ((k+1) h_{k+1}(\RC)+ (d-k) h_{k}(\RC)) t^{k}}{(1-t)^{d+1}}. \qedhere
\end{align*}
\end{proof}

To bound $h_k(\St(v,\RC))$, we need a relative version of a simple
lemma of Stanley~\cite{Stanleymono} and Kalai~\cite{Kalaishifting}. 

\begin{lem}\label{lem:surj}
    Let $\RC=(\Delta,\Gamma)$ be a relative complex of dimension $d-1$ on
    vertex set $[n]$. Let $\Delta' \subseteq \Delta$ be any subcomplex of $\Delta$, 
    and set $\Gamma'=\Delta'\cap \Gamma$ and $\RC'=(\Delta',\Gamma')$.   Then, for every $k$,
    \[
        \dim_\k (\FM[\RC']/\Theta\FM[\RC'])_k
        \ \le \ \dim_\k(\FM[\RC]/\Theta\FM[\RC])_k.
    \]
\end{lem}

\begin{proof}
This follows immediately if we consider $\FM[\RC']$ as an
$\k[\x]$-module: By right-exactness of the tensor product, 
we have a degree preserving surjection
$\FM[\RC]/\Theta\FM[\RC] \longtwoheadrightarrow
    \FM[\RC']/\Theta\FM[\RC'].$\end{proof}


The last ingredient is a property for vertex stars of full subcomplexes.
\begin{lem}\label{lem:fulleq}
    Let $\Gamma \subseteq \Delta$ be a pair of simplicial complexes. Then
    $\Gamma$ is full in $\Delta$ if and only if $\St(v,\Gamma) = \St(v,\Delta)\cap \Gamma$
    for all $v \in \V(\Gamma)$. \qed
\end{lem}

For a relative complex $\RC = (\Delta,\Gamma)$ and a vertex $v\in \Delta$, let us write
$\widetilde{\St}(v,\RC):=(\St(v,\Delta),\St(v,\Delta)\cap \Gamma)$.

\begin{ex}\label{ex:rst}
    Let $\RC = (\Delta,\Gamma)$ be a relative complex such that both $\Delta$
    and $\Gamma$ are Cohen--Macaulay of the same dimension and $\Gamma$ is
    full in $\Delta$. By Lemmas~\ref{lem:hlk} and~\ref{lem:lkst} 
    \[
            \sum_{v\in \Delta} h_k(\St(v,\RC)) \ = \
            (k+1)h_{k+1}(\RC)+(d-k) h_k(\RC).
    \]
    Since $\Gamma$ is full, Lemma~\ref{lem:fulleq} yields
    $\St(v,\RC)=\widetilde{\St}(v,\RC)$ for all vertices $v \in \Gamma$.
    Therefore, for a \lsop\ $\Theta$ and for every vertex $v\in \Gamma$
    \begin{align*}
         h_k(\St(v,\RC)) &\ = \ 
        h_k(\widetilde{\St}(v,\RC))&\ \ \\
        &\ = \ \dim_\k \bigl(\FM[\widetilde{\St}(v,\RC)] /
            \Theta\FM[\widetilde{\St}(v,\RC)]\bigr)_k 
            & \text{(by Cohen--Macaulayness)}\\
        &\ \le \ \dim_\k\bigl(\FM[\RC]/\Theta\FM[\RC]\bigr)_k&\ \text{(by Lemma~\ref{lem:surj})}\\
        &\ = \  h_k(\RC)&\ \text{(by Cohen--Macaulayness)}.
    \end{align*}
    If  $v\notin \Gamma$, the reasoning becomes a little more difficult as 
    $\St(v,\RC)$ not necessarily coincides with $\widetilde{\St}(v,\RC)$ any
    more. However, we can simply estimate
    \begin{align*}
        h_k(\St(v,\RC)) &\ = \  h_k({\St}(v,\Delta))&\ \ \\
        & \ = \ \dim_\k \bigl(\FM[{\St}(v,\Delta)] /
                        \Theta\FM[{\St}(v,\Delta)]\bigr)_k\ 
                        & \text{(by Cohen--Macaulayness)}\\
        & \ \le \ \dim_\k \bigl(\FM[\Delta]/\Theta\FM[\Delta]\bigr)_k 
                        & \text{(by Lemma~\ref{lem:surj})}\\
        & \ = \ h_k(\Delta) & \text{(by Cohen--Macaulayness)}\\
        & \ =\  h_k(\RC)+h_k(\Gamma)&\ \text{(by linearity of the h-vector)}.
    \end{align*}
    These inequalities, for the special case of simplicial polytopes
    were the key to McMullen's proof of the UBT for polytopes.
    Integrating these inequalities over all vertices and using Lemma~\ref{lem:hlk},
    we obtain
    \begin{equation}\label{eqn:mcmullen-gen}
        (k+1)\, h_{k+1}(\RC) \ \le \ (f_0(\Delta)-d+k) \, h_{k}(\RC) +
        f_0(\RC)\, h_k(\Gamma).
    \end{equation}
\end{ex}

To handle situations with
$\dim \Gamma < \dim \Delta$, let us define for $\ell \ge 0$
\[
    h_i^{\langle\ell\rangle}(\RC) \ := \ [(1-t)^\ell \Hilb(\FM[\RC],t)]_i \ =
    \ \sum^i_{k=0} (-1)^{i-k} \binom{\ell-k}{\ell-i}f_{k-1}(\RC).
\]
Comparing this with the definition of $h$-vectors in
Section~\ref{sec:RelativeSR}, we see that $h_i^{\langle d \rangle}(\RC) =
h_i(\RC)$ for $d = \dim \RC + 1$.
Also, for an arrangement $\scr{G}$ and a vertex $v \in \Delta$ we set
$\scr{G}(v):=\{\Gamma\in \scr{G}: v\in \Gamma\}$ and consequently
$
    \cfs{G}(v) \ := \ \bigcup_{\Gamma\in \scr{G}(v) } \Gamma.$
We call the pair $(\Delta,\scr{G})$ \Defn{universally Buchsbaum} of dimension
$d-1$ if for every vertex $v$ of the $(d-1)$-complex $\Delta$, the relative
complex $(\Delta,\cfs{G}(v))$ is Buchsbaum of dimension $(d-1)$. 
\begin{lem}\label{lem:h_uBuchs}
    Let $(\Delta,\scr{G})$ be universally Buchsbaum of dimension
    $d-1$, where $\scr{G}$ is some full arrangement of subcomplexes of
    $\Delta$. Let $v$ be any vertex of $\Delta$. Then
    \begin{equation}\label{eqn:h_uBuchs}
        h_k(\Lk(v,\Delta),\Lk(v,\cfs{G})) \ \le \
        h_k(\Delta,\cfs{G}(v))-\GT_k(\Delta,\cfs{G}(v)).
    \end{equation}
\end{lem}
\begin{proof}
    Let $v \in \Delta$ be a vertex and let us write $\RC(v) =
    (\Delta,\cfs{G}(v))$. With a \lsop\ $\Theta$ we deduce
    \begin{align*}
        h_k(\Lk(v,\Delta),\Lk(v,\cfs{G})) & \ = \ h_k(\Lk(v,\RC(v))&\ \ \\
        & \ = \ h_k(\St(v,\RC(v))&\ \ \text{(Lemma~\ref{lem:lkst})}\\
        & \ = \ h_k(\widetilde{\St}(v,\RC(v))&\ \ \text{(using fullness)}\\
        & \ = \ \dim_\k \bigl(\FM[\widetilde{\St}(v,\RC(v))] /
                \Theta\FM[\widetilde{\St}(v,\RC(v))]\big)_k
                & \text{(since $\widetilde{\St}(v,\RC(v))$ is CM)}\\
        & \ \le \  \dim_\k \big(\FM[\RC(v)]/\Theta\FM[\RC(v)]\bigr)_k 
                & \text{(by Lemma~\ref{lem:surj})}
    \end{align*} 
By Theorem~\ref{thm:RS}, the last expression equals $h_k(\Delta,\cfs{G}(v))-\GT_k(\Delta,\cfs{G}(v))$.
\end{proof}

Summing equation~\eqref{eqn:h_uBuchs} over all vertices of $\Delta$ and using
Lemma~\ref{lem:hlk} as in Example~\ref{ex:rst}, we obtain the following result.

\begin{thm}\label{thm:loc}
    Let $\RC = (\Delta,\cfs{G})$ be a relative complex of dimension $d-1$
    where $\scr{G}$ is a full arrangement. If $(\Delta,\scr{G})$ is
    universally Buchsbaum, then
    \[
        (k+1)\, h_{k+1}(\RC) \ \le \  (f_0(\Delta)-d+k) \, h_{k}(\RC) +
        \sum_{v\in \Delta} \Bigl( h_k^{\langle
        d\rangle}(\cfs{G},\cfs{G}(v))-\GT_k(\Delta,\cfs{G}(v)) \Bigr).
    \]
\end{thm}

The results are somewhat simpler if the pair $(\Delta,\scr{G})$ is
\Defn{universally Cohen--Macaulay} of dimension $d-1$, i.e., 
$(\Delta,\cfs{G}(v))$ is Cohen--Macaulay of dimension $d-1$ for every vertex
$v \in \Delta$. This also means that the topological terms in Theorem~\ref{thm:loc} vanish.

\begin{cor}\label{cor:loc}
    Let $(\Delta,\scr{G})$ be universally Cohen--Macaulay of dimension $d-1$
    where $\scr{G}$ is a full arrangement of $\Delta$. 
    \begin{compactenum}[\rm (1)]
    \item  For every vertex $v \in \Delta$ 
        \[
            h_k(\Lk(v,\Delta),\Lk(v,\cfs{G})) \ \le \
            h_k(\Delta,\cfs{G}(v)) \ = \ h_k(\Delta,\cfs{G})+h_k^{\langle
            d\rangle}(\cfs{G},\cfs{G}(v)).
        \]
        Equality holds up to some $k_0$ if and only if for every $\sigma$ of
        $(\Delta,\scr{G}(v))$ of dimension $<k_0$, the simplex $\sigma \ast v$
        is a face of $\Delta$.
    \item Moreover, we have
        \[
            (k+1)\, h_{k+1}(\RC) \ \le \  (f_0(\Delta)-d+k) \, h_{k}(\RC)
            + \sum_{v\in \Delta} h_k^{\langle
            d\rangle}(\cfs{G},\cfs{G}(v)).
        \]
        Equality holds if and only if it holds for all $v$ in {\rm (1)}.
\end{compactenum}
\end{cor}

\begin{proof}
    {\rm (1)} and {\rm (2)} are Lemma~\ref{lem:h_uBuchs} and
    Theorem~\ref{thm:loc} for universally CM pairs.  It remains to
    characterize the equality cases. Recall that the inclusion
    $\St(v,\Delta) \cap \cfs{G}(v) \subseteq \cfs{G}(v)$ induces a
    degree-preserving surjection
    \[
        \varphi_v: \FM[\RC(v)] \ = \ \I_{\cfs{G}(v)} / \I_\Delta 
        \ \longtwoheadrightarrow \  
        \I_{\cfs{G}(v)} / \I_{\St(v,\Delta)} \ = \
            \FM[\widetilde{\St}(v,\RC(v))].
    \]
    Since $\FM[\widetilde{\St}(v,\RC(v))]$ is CM, for a \lsop\ $\Theta$ for
    $\FM[\RC(v)]$ we get a short exact sequence 
    \[
        0\longrightarrow \ker(\varphi_v)/\Theta\ker(\varphi_v) \longrightarrow 
        \FM[\RC(v)]/\Theta \FM[\RC(v)]\longrightarrow 
        \FM[\widetilde{\St}(v,\RC(v))]/\Theta 
        \FM[\widetilde{\St}(v,\RC(v))]\longrightarrow 0.
    \]
    Therefore, equality holds if and only if the surjection $\varphi_v$ is an
    isomorphism if and only if \[(\ker(\varphi_v)/\Theta\ker(\varphi_v))_{\le k_0}\ =\ (\FM[\Delta,\cfs{G}(v)\cup \St(v,\Delta)]/\Theta\FM[\Delta,\cfs{G}(v)\cup \St(v,\Delta)])_{\le k_0}\ =\ 0\]
    which is only the case if the face module $\FM[\Delta,\cfs{G}(v)\cup \St(v,\Delta)]$ is generated in degree $>k_0$.
\end{proof}

\begin{rem}\label{rem:tight}
The equality cases in Theorem~\ref{thm:loc} are a bit harder to characterize;
    one can use Proposition~\ref{prp:exact} and
    Theorem~\ref{thm:reduction_to_topology}.
\end{rem}

\subsection{A reverse isoperimetric inequality} \label{ssec:rel_diff}

We can use the philosophy of Lemma~\ref{lem:surj} in yet another way to
provide upper bounds on algebraic $h$-numbers by replacing
Lemma~\ref{lem:surj} with a stronger inequality. The results, even though they
require more work, yield inequalities stronger than the ones provided in
Section~\ref{ssec:primes}. For simplicity, we focus on the Cohen--Macaulay
case and leave the general case to the interested reader.

Let $\RC = (\Delta,\Gamma)$ be a relative complex. A \Defn{relative
subcomplex} of $\RC$ is a relative complex $\RC' = (\Delta',
\Gamma')$ with $\Delta' \subseteq \Delta$ and $\Gamma' \subseteq \Gamma$. The
pair $(\RC,\RC')$ is again a relative complex with face module
\[
    \FM[\RC,\RC'] \ := \ \ker( \FM[\RC] \twoheadrightarrow
    \FM[\RC']) \ \cong \ \FM[\Delta, \Gamma \cup \Delta'].
\]
We say that $\RC'$ is a full relative subcomplex if $\Delta' \subseteq
\Delta$ is full.

\begin{thm}\label{thm:rii}
    Let $\RC$ be a $(d-1)$-dimensional Cohen--Macaulay relative complex and
    $\RC'$ a codimension one Cohen--Macaulay full relative subcomplex.  Then
    \[
        h_k(\RC,\RC') \ \ge \ h_{k-1}(\RC')
    \]
for all $k$.
\end{thm}
\begin{proof}
    Let $\theta_1,\dots,\theta_d$ be a \lsop\ for $\FM = \FM[\RC]$ such that
    $\Theta = (\theta_1,\dots, \theta_{d-1})$ is a \lsop\ for $\FM' =
    \FM[\RC']$ and $\theta_d$ is a linear form $\theta_d \in
    \k\text{-span} \{ x_v : v \not\in \V(\Delta')\}$; this can be done as
    $\Delta'$ is full in $\Delta$ and of codimension one, so that every facet
    of $\Delta$ contains at least one vertex not in $\Delta'$, compare also
    \cite[Section III.9]{Stanley96}.
    
    Consider the injective map $\widetilde{\varphi}:\FM/\Theta \FM \rightarrow \FM /
    \Theta \FM$ given by multiplication by $\theta_d$. Now 
    $\theta_d \FM \subseteq \FM[\RC,\RC']$ by choice of $\theta_d$ and hence, we get a
    homogeneous map 
    $
        \varphi : \FM / \Theta \FM \, \rightarrow \, \FM[\RC,\RC'] /
        (\Theta,\theta_d) \FM[\RC,\RC']
    $ of degree one induced by the multiplication with $\theta_d$.
    Again by the regularity of $\theta_d$ and using the fullness property, the kernel 
    of $\varphi$ is given by $\ker \varphi = \FM[\RC,\RC']/ \Theta \FM[\RC,\RC']$. 
    Factoring out the kernel we get an injection 
    \[
        \overline{\varphi} : 
        \FM' / \Theta \FM' \  \longhookrightarrow \  \FM[\RC,\RC'] /
        (\Theta,\theta_d) \FM[\RC,\RC'].
    \]
    Since $\overline{\varphi}$ is homogeneous of degree one, we obtain
    \[
    h_{k-1}(\RC') \ = \ \dim_\k (\FM' / \Theta \FM')_{k-1}
    \ \le \ 
    \dim_\k (\FM[\RC,\RC'] / (\Theta,\theta_d) \FM[\RC,\RC'])_k \ = \
    h_k(\RC,\RC').\qedhere
    \]
\end{proof}

A simple application of this inequality yields a reverse isoperimetric
inequality.

\begin{cor}\label{cor:rev_iso}
    Let $\Delta$ be a simplicial ball and assume that $\partial \Delta$ is full in
    $\Delta$, then 
    \[
        h_k(\Delta,\partial \Delta) \ \ge\ h_{k-1}(\partial \Delta)
    \]
for all $k$.
\end{cor}

In the same situation, Lemma~\ref{lem:surj} only yields $h_k(\Delta,\partial
\Delta)\ \ge\ 0$ which also follows easily since $(\Delta,\partial \Delta)$ is
Cohen--Macaulay. This \emph{almost} is a Lefschetz-type result that
characterizes primitive Betti numbers, compare also~\cite{Stanley96}. We refer
to \cite{AdipII} for related applications towards a quantitative lower bound
theorem, and also Remark~\ref{rem:bb} for a small application.



\subsection{Relations to relative shellability.}\label{ssec:rel_shell}

The estimates of Section~\ref{ssec:primes} are reminiscent of McMullen's approach
to $h$-vectors and the Upper Bound Theorem via shellings. In this section we
want to put our techniques into perspective via the notion of relative
shellability. The results presented here are not essential for the following
sections, but provide a combinatorial viewpoint.

Let $\RC = (\Delta,\Gamma)$ be a pure relative complex of dimension $d-1$ and let
$F \in \Delta {\setminus} \Gamma$ be a
facet. The deletion $\RC' = \RC - F := (\Delta - F, \Gamma)$ is a step in a \Defn{relative shelling}
if $\RC' \cap F$ is pure of codimension one. A relative complex is shellable
if there is sequence of shelling steps to the relative complex
$(\Gamma,\Gamma)$. If $\Gamma = \emptyset$, then this is the classical notion
of shelling of simplicial complexes. Relative shellings
where introduced by Stanley~\cite{Stanley87} and further developed
in~\cite{AB-SSZ}. Shellability has proven to be an invaluable tool
in topological combinatorics. The basis for our situation is the 
following result due to Kind--Kleinschmidt \cite{KK79} and
Stanley~\cite{Stanley96}. 

\begin{prp}
    A shellable relative complex is Cohen--Macaulay over any ground
    field.
\end{prp}

In particular, the $h$-vector of a relative complex can be read off a shelling.

\begin{prp}\label{prp:shell_to_h}
    Let $\RC' = \RC - F$ be a shelling step and let $\sigma$ be the
    unique minimal face in $2^F\setminus\RC'$. Then
    $h_k(\RC)  = h_k(\RC') + 1$ for $k = |\sigma|$ and
    $h_k(\RC)=h_k(\RC')$ otherwise.
\end{prp}

Let us revisit the situation of Example~\ref{ex:rst} from the perspective of
relative shellings: We call $\RC = (\Delta,\Gamma)$ \Defn{universally
shellable} if for every vertex $v$ of $\Delta$ there is a shelling of $\RC$
that removes $\St(v,\RC)$ first. For a universally shellable complex $\RC$
such that $\Gamma$ is full, the arguments of~\cite{mcmullen1970} (see
also~\cite[Section~8.4]{Z}) yield once again
\[
    (k+1)\, h_{k+1}(\RC) \ \le \  (f_0(\Delta)-d+k) \, h_{k}(\RC) +
    f_0(\RC)\, h_k(\Gamma).
\]
This is sufficient to provide a solution to the upper bound problem for
universally shellable relative complexes in the sense of
Lemma~\ref{lem:UBT-CM}; see also Theorem~\ref{thm:comb_iso_full}.
The challenge, of course, is to show that a given relative complex is shellable, that is,
to exhibit an actual shelling. For this, one can use a variety of methods
from poset theory~\cite{Bj1}, geometry~\cite{BM}, and tools such as Alexander duality and gluing theorems
for relative shellings, cf.~\cite{AB-SSZ}.

\section{The Upper Bound Theorem for Minkowski Sums}
\label{sec:UBTMS}

We now come to our main application of relative Stanley--Reisner theory: A
tight upper bound theorem for Minkowski sums of polytopes. In analogy to the
classical UBT, the class of polytopes that maximize the number of $k$-faces is
rather special and we devote the first section to their definition and the
statement of results. The proofs are rather intricate and we illustrate the
main ideas in the case of two summands $P_1+P_2$ in Section~\ref{ssec:two}
which recovers the results of~\cite{Karavelas11} with a simple argument.

The transition from Minkowski sums to relative simplicial complexes
is via the \emph{Cayley polytope} and the \emph{(relative) Cayley complex},
whose definition and properties are presented in Subsection~\ref{ssec:cayley}.
In particular, the Cayley complex allows us to introduce the notion of an
$h$-vector for special families of simplicial polytopes and reduce the upper
bound problem to one on $h$-vectors. The general scheme for the proof is then
similar to that of the UBT for polytopes: 
We will prove sharp upper bounds for
the `first half' of the $h$-vector (Section~\ref{ssec:small_h}).  For the
`second half' of the $h$-vector we prove Dehn--Sommerville-type relations in
Section~\ref{ssec:DSM}.  Unfortunately, this formula does not express $h_k$ of the 
second half as positive linear combinations of such from the first half,
so that we need a further strengthening of the bounds provided in
Section~\ref{ssec:UBTM}. We finally conclude the Upper Bound Theorem for
Minkowski sums (Theorem~\ref{mthm:minkm}). 
\enlargethispage{3mm}
While some statements in this section are general, we
focus in this section on Minkowski sums of \emph{pure} collections, i.e.,
Minkowski sums of polytopes in $\R^d$ with at least $d+1$ vertices each.
We discuss the nonpure case in Section~\ref{sec:nonpure}.

\newcommand\n{\mathbf{n}}
\subsection{Minkowski-neighborly polytopes and main results} \label{ssec:M-neighborly}
Let us recall the setup for the Minkowski upper bound problem. For given $m,d
\ge 1$ and $\n = (n_1,\dots,n_m) \in \Z_{\ge d+1}^m$, we seek to find tight upper bounds on 
\[
    f_k(P_1+P_2+\cdots+P_m)
\]
for polytopes $P_1,\dots,P_m$ such that $f_0(P_i) = n_i$ for all
$i=1,\dots,m$. We shall focus here on \Defn{pure} families, that is, families where each of the summands 
has at least $d+1$ vertices. To ease the notational burden, let us write $P_{[m]}:=
(P_1,\dots,P_m)$ and $f_k(P_{[m]}) = (f_k(P_1),\dots,f_k(P_m))$. We also
abbreviate $|P_{[m]}| := P_1 + P_2 + \cdots + P_m$ for the Minkowski sum of a
family. We extend these notions to subfamilies $P_S = (P_i : i \in S)$ for $S
\subseteq [m]$.

As for the UBT, we can make certain genericity assumptions. Recall that every
face $F$ of $|P_{[m]}|$ can be written as $F = F_1+\cdots+F_m= |F_{[m]}|$
where $F_i \subseteq P_i$ are unique nonempty faces. It follows that
\begin{equation}\label{eqn:Mink_dim}
    \dim F \ \le \ \dim F_1 + \cdots + \dim F_m.
\end{equation}

We call the polytopes $P_{[m]}$ in \Defn{relatively general position} if
equality holds in~\eqref{eqn:Mink_dim} for all proper faces $F \subsetneq
|P_{[m]}|$.
Similar to the situation of the UBT for polytopes and spheres, it
is possible to reduce the UBPM to simplicial polytopes in relatively general
position
 by a simple perturbation; compare
\cite[Theorem~1]{FukudaWeibel10}.  
We need a notion similar to neighborliness of
polytopes that will describe the polytopes attaining the upper bound.  
\begin{dfn}\label{dfn:M-neighborly} 
    Let $P_{[m]}=(P_1,\dots,P_m)$ be a collection of polytopes in $\R^d$.
    Then $P_{[m]}$ is \Defn{Minkowski $(k,\ell)$-neighborly} for $k \ge
    0$ if for every subset $J \subseteq [m]$ of cardinality $\ell$, and for any choice of vertices $\emptyset \neq S_j \subseteq
    \V(P_j)$ with $j \in J$ such that
    \[
        \sum_{j\in J} |S_j| \ \le \ k + |J| - 1, 
    \]  
    the polytope $\sum_{j\in J}\conv(S_j)$ is a simplex of $|P_J|$.
\end{dfn}

For $\ell = 1$, this recovers the definition of $k$-neighborly polytopes. For $\ell=m$, the number of $k$-faces in a Minkowski
$(k,m)$-neighborly family, if it exists, satisfies
\[
    f_k(|P_{[m]}|) \ = \ \sum_{\substack{\alpha \in \Z^m_{\ge 1}\\ 
    |\alpha|=m+k}} \prod_{i=1}^m \binom{f_0(P_i)}{\alpha_i}
\]
which is the trivial upper bound for face numbers of Minkowski sums. The following theorem characterizes 
Minkowski neighborly polytopes and generalizes the standard properties of 
neighborly polytopes.

\begin{thm}\label{thm:M-neighborly}
Let $m,d \ge 1$ be fixed.
\begin{compactenum}[\rm (i)] 
    \item There is no pure Minkowski $(k,\ell)$-neighborly family $P_{[m]}$ in $\R^d$ for $k +\ell -1>
            \frac{d+\ell-1}{2} $.
    \item For all $\n \in \Z_{\ge d+1}^m$ there is a family $P_{[m]}$ in $\R^d$ with
        $f_0(P_{[m]})=\n$  that is Minkowski $(k,\ell)$-neighborly for all
        $\ell\le m$ and $\ell-1 \le k+\ell-1 \le
        \lfloor\frac{d+\ell-1}{2}\rfloor$.
\end{compactenum}
\end{thm}

The first claim is a straightforward consequence of Radon's Theorem
once we phrase the UBPM in the language of
Cayley polytopes; cf.~Proposition~\ref{prp:radon_cayley}.  
It suffices to prove the assertion for $\ell=m$, the general case of
the assertion is a straightforward corollary.

As for Theorem~\ref{thm:M-neighborly}(ii), the construction is provided by Theorem~2.6 in \cite{MPP}. The constructions are based on cyclic
polytopes and generalize those of~\cite{Karavelas11,
Karavelas12}. 

Theorem~\ref{thm:M-neighborly} suggests the following notion: A family
$P_{[m]}$ of polytopes is called \Defn{Minkowski neighborly}  if $P_{[m]}$ is
Minkowski $(k,\ell)$-neighborly for all $\ell\le m$ and $\ell-1 \le k+\ell-1 <
\lfloor\frac{d+\ell-1}{2}\rfloor$.  As in the case of the UBT for spheres, the
face numbers of Minkowski neighborly polytopes only depend on $m$, $d$ and
$f_0(P_{[m]})$.

\begin{prp}\label{prop:M-neigh-fvector}
If $P_{[m]},P_{[m]}^\prime$ are two
Minkowski neighborly families of $m$ simplicial $d$-polytopes with
$f_0(P_{[m]}) = f_0(P^\prime_{[m]})$, then
$
    f_k(|P_{[m]}|) \ = \ f_k(|P^\prime_{[m]}|)
$
for all $0 \le k \le d$.
\end{prp}
    
This result will be a simple consequence of the Dehn--Sommerville relations
for Cayley complexes developed in Section~\ref{ssec:DSM}; cf.\ Corollary
\ref{cor:lh_hh}. Unless $\sum_i f_0(P_i)\le d+m$, it is not true that Minkowski neighborly families of simplicial polytopes in relative general position have combinatorially equivalent Minkowski sums, so that the combinatorial types of such Minkowski sums remain to be understood; instead, Proposition~\ref{prop:M-neigh-fvector} allows us to study the face numbers.

With the help of Proposition~\ref{prop:M-neigh-fvector}, we
define $\mr{nb}_k(d,m,\n)  := f_k(|\mr{Nb}_{[m]}|)$ for $m,d \ge 1$ and $\n \in \Z^m_{\ge d+1}$,
where $\mr{Nb}_{[m]}$ is any Minkowski neighborly family of $m$ simplicial
$d$-polytopes in $\R^d$ with $\n=f_0(\mr{Nb}_{[m]})$. 

\begin{thm}[Upper Bound Theorem for Minkowski sums]\label{thm:UBTMS}
    Let $m,d \ge 1$ and $\n \in \Z^m_{\ge d+1}$. If $P_{[m]} = (P_1,\dots,P_m)$
    is a pure family of $m$ polytopes in $\R^d$ with $f_0(P_{[m]})=\n$, then
    \[
        f_k(|P_{[m]}|) \ \le \ \mr{nb}_k(d,m,\n).
    \]
    for all $k = 0, \dots, d-1$. Moreover, the family $P_{[m]}$ is Minkowski
    neighborly if and only if equality holds for some $k_0,\ k_0+1 \ge\frac{d+2m-2}{2}$.
\end{thm}

Unfortunately, closed formulas for $\mr{nb}_k(d,m,\n)$ are rather involved,
even for small $k$. As in the case of the UBT
for polytopes/spheres, upper bounds are best expressed in terms of
$h$-numbers. We introduce $h$-numbers for simplicial families in relatively
general position in the next section and give a rigorous treatment in
Section~\ref{ssec:cayley}.

\subsection{Minkowski sums of two polytopes} \label{ssec:two}
In this section we illustrate the
general proof strategy along the case of two summands. Let $P_{[2]} = (P_1,P_2)$
be two simplicial $d$-dimensional polytopes in $\R^d$ in relatively general
position with $f_0(P_{[2]}) = (n_1,n_2)$.  We seek to find the maximum possible
$f_k(|P_{[2]}|) = f_k(P_1+P_2)$ for any fixed choice of $k$. Let us define the
\Defn{Cayley polytope} of $P_{[2]}$ as the $(d+1)$-dimensional polytope
\[
    C \ = \ \Cay(P_1,P_2) \ := \ \conv( P_1 \times \{0\} \cup P_2 \times \{1\} )
    \ \subseteq \ \R^d \times \R
\]
as sketched in Figure~\ref{fig:Cayley}.
\begin{figure}[htb]
    \includegraphics[width=6cm]{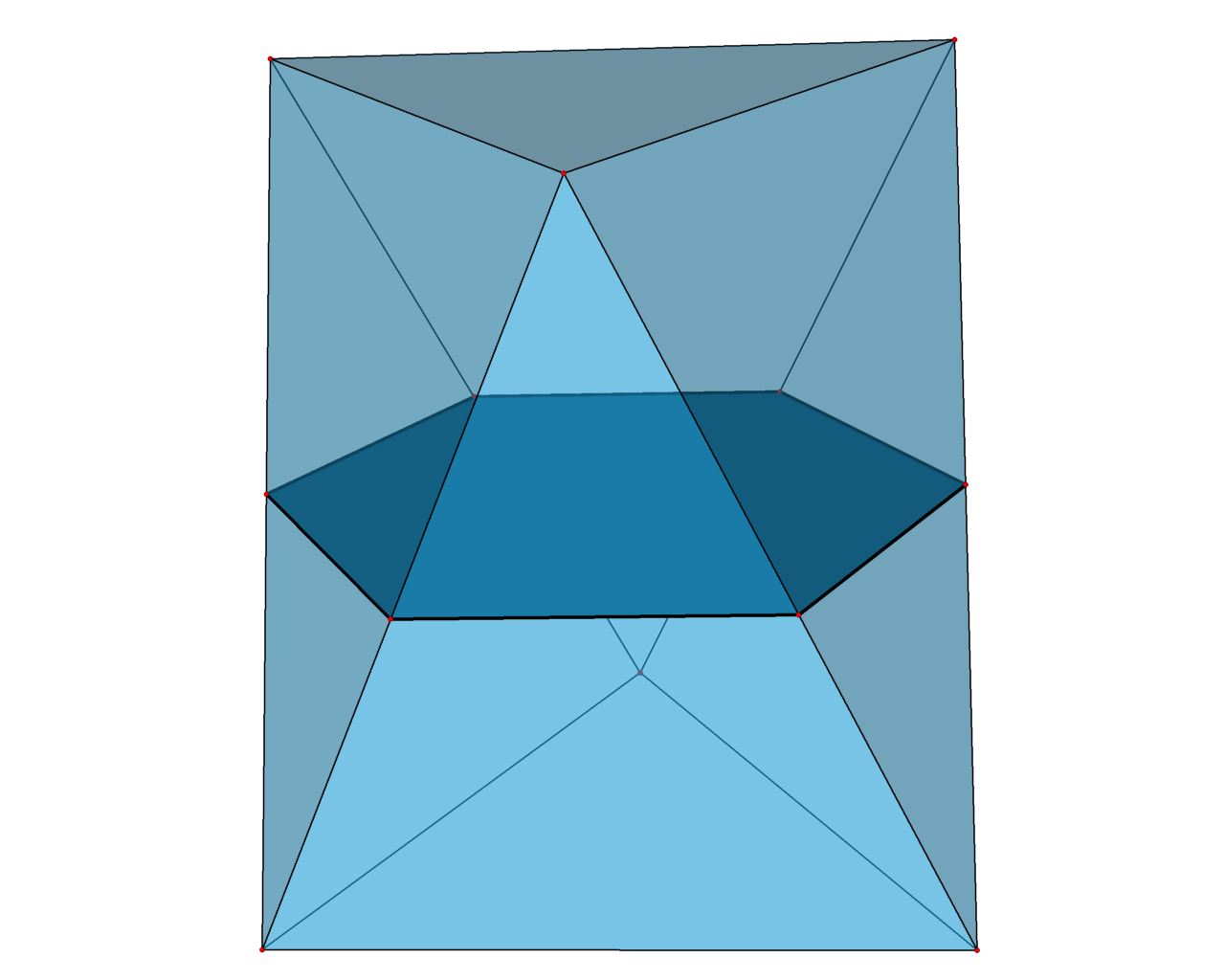}
\caption{Cayley polytope of two triangles and the middle section.}
\label{fig:Cayley}
\end{figure}

The Cayley polytope has the favorable property that for $L = \R^d \times
\{\frac{1}{2}\}$ 
\[
    C \cap L \ \cong \ P_1 + P_2
\]
where the isomorphism is affine.  As the intersection of $L$ with faces of $C$
is transverse, we infer
\[
    f_k(P_1+P_2) \ = \ f_k(C \cap L) \ = \ f_{k+1}(C) - f_{k+1}(P_1) -
    f_{k+1}(P_2)
\]
for $k = 0,\dots,d-1$.  By assumption on $P_{[2]}$, the only proper faces of
$C$ which are possibly not simplices are $P_1$ and $P_2$ and we define $\Delta
:= \partial C {\setminus} \{P_1,P_2\}$ as the simplicial complex spanned by all
proper faces different from $P_1$ and $P_2$. Observe that the boundary complexes
$\partial P_1, \partial P_2$ are disjoint subcomplexes of $\Delta$ and we define
$\Gamma := \partial P_1 \cup \partial P_2$.  For the relative simplicial complex
$\RC = (\Delta,\Gamma)$, we record
\[
    f_k(P_1 + P_2) \ = \ f_{k+1}(\RC) \ = \ f_{k+1}(\Delta) - f_{k+1}(\Gamma)
\]
for all $k=0,\dots,d-1$. For later perspective, $\Delta$ is called the \emph{Cayley complex}, and $(\Delta,\Gamma)$ is the relative Cayley complex.

We can now appeal to Observation~\ref{obs:f_via_h}
to reduce the task to bounding $h_k(\RC)$ instead. Hence, we define
    $h_k(P_{[2]})  :=  h_k(\Delta,\Gamma)$
for $i=0,\dots,d$. This setup now fits into the scheme of a relative upper
bound problem. Using the developed techniques of relative Stanley--Reisner
theory we can resolve this upper bound problem which recovers the the main
theorem of Karavelas and Tzanaki~\cite[Theorem~18]{Karavelas11}.

\begin{thm}[UBT for two summands]\label{thm:twosummands}
    Let $P_{[2]} = (P_1,P_2)$ be two simplicial $d$-polytopes in relatively
    general position with $n_1$ and $n_2$ vertices, respectively. Then
    \begin{align*}
        h_{k+1}(P_{[2]}) \ &\le \ 
            \binom{n_1+n_2 - d + k - 1}{k+1} -
            \binom{n_1 - d + k - 1}{k+1} - 
            \binom{n_2 - d + k - 1}{k+1}+ (-1)^{k+1}\binom{d+1}{k+1}
\\
            \intertext{for $k+1 \le \lfloor\frac{d+1}{2}\rfloor$ and }
    h_{k+1}(P_{[2]}) \ &\le \ \binom{n_1+n_2 -k - 2}{d-k}+ (-1)^{k+1}\binom{d+1}{k+1}             \end{align*}
            for $k+1 > \lfloor\frac{d+1}{2}\rfloor$. Equality holds for all $k$ simultaneously if and only if $P_{[2]}$ is Minkowski neighborly.
\end{thm}

\begin{proof}
    The complex $\Delta' := \Delta \cup P_1 \cup P_2$ is a
    $d$-sphere and hence Cohen--Macaulay. In particular $\cfs{G} = \{
    \partial P_1, \partial P_2\}$ is, up to excision, a full arrangement of disjoint codimension one CM
    subcomplexes of $\Delta'$. For the first inequality is provided by
    Theorem~\ref{thm:arrCM}.
    
    For the second inequality, we use the Dehn--Sommerville relations for
    relative complexes (Lemma~\ref{lem:DS}) together with the fact that
    $(\Delta,\cfs{G})$ is weakly Eulerian.  Finally, we observe that the full
    arrangement $\cfs{G}$ is $d$-magnificent in the sense of
    Theorem~\ref{thm:magnificent} (see example Example~\ref{ex:cm}).  It now
    follows with Theorems~\ref{thm:upb} and~\ref{thm:magnificent} that
    tightness in the inequalities implies the desired neighborliness.
\end{proof}

\subsection{Cayley polytopes and Cayley complexes} \label{ssec:cayley}
The geometric construction of the previous section is easily generalized to
higher dimensions.  For a family $P_{[m]} = (P_1,\dots,P_m)$ of $m$
polytopes in $\R^d$, we define the \Defn{Cayley polytope} as 
\[
    \Cay(P_{[m]}) \ := \ \conv\Biggl(\,\bigcup_{i=1}^m P_i \times
    e_{i}\,\Biggr) \ \subseteq \ \R^d \times \R^m.
\]
The coordinate projection $\R^d \times \R^m \rightarrow \R^m$ restricts to a
linear projection 
\begin{equation}\label{eq:projection}
    \pi : \Cay(P_{[m]}) \ \longrightarrow \ \K_{m-1} \ = \ \conv\{e_1,\dots,e_m\}
\end{equation}
of the Cayley polytope to the (geometric) standard $(m-1)$-simplex. It is easy
to see that for $\lambda = (\lambda_1,\dots,\lambda_m) \in \Delta_{m-1}$, we
have
\begin{equation}\label{eq:fib}
 \pi^{-1}(\lambda) \ \cong \ \lambda_1P_1 + \cdots + \lambda_m P_m.
\end{equation}   
In particular, for any $\lambda \in \relint \Delta_{m-1}$, $\pi^{-1}(\lambda)$
is combinatorially equivalent (and even \emph{normally} equivalent, cf.\ 
\cite[Section~7]{Z}) to $P_1+\cdots+P_m$.  Let us denote by $\Delta_J
= \conv\{e_i : i \in J\}$ the faces of $\Delta_{m-1}$ for the various subsets
$J \subseteq [m]$.  Cayley polytopes are an indispensable tool in the study of
Minkowski sums, cf.\ \cite{dLRS}. For nonempty faces $F_i \subseteq P_i$ for
$i=1,\dots,m$
\[
    F_1+ \cdots + F_m \subseteq |P_{[m]}| \text{ is a face} \quad
    \Longleftrightarrow \quad \Cay(F_1,\dots,F_m) \subseteq \Cay(P_1,\dots,P_m)
    \text{ is a face};
\]
see~\cite[Observation 9.2.1]{dLRS}.
Together with the next result, this
correspondence yields a simple proof of Theorem~\ref{thm:M-neighborly}(i).

\begin{prp}\label{prp:radon_cayley}
    Let $P_{[m]}$ be a family of $m$ polytopes in $\R^d$.  If $\sum_{i=1}^m
    f_0(P_i)>d+m$, then there exist a choice of vertices  $\emptyset \neq S_i
    \subseteq \V(P_i)$ for $i=1,\dots,m$ with
    \[
        \sum_{i=1}^{m} |S_i|  \ \le \ \left\lfloor \frac{d+m+1}{2}
        \right\rfloor
    \]
    such that $\Cay(\conv(S_1),\dots,\conv(S_m))$ is not a face of $\Cay(P_{[m]})$.
\end{prp}
\begin{proof}
    Let $M$ be any choice of $d+m+1$ vertices of the $(d+m-1)$-dimensional
    Cayley polytope in $\R^{d+m-1}$.  By Radon's theorem, $M$ may be
    partitioned into two sets $M_1,\ M_2$, whose convex hulls intersect. Without loss of generality, we may assume that $M_1$ is the smaller of the two, so that $|M_1| \le \left\lfloor
    \frac{d+m+1}{2}\right\rfloor$. 
    Hence $\conv(M_1)$ is the desired non-face of $\Cay(P_{[m]})$.    
\end{proof}

The following simple proposition summarizes the most important properties of
the Cayley polytope. For proofs and more information see~\cite{dLRS}. An
illustration of the Cayley polytope for three summands is given in
Figure~\ref{fig:cayley_sketch}.
\begin{figure}[htb]
    \includegraphics[width=10cm]{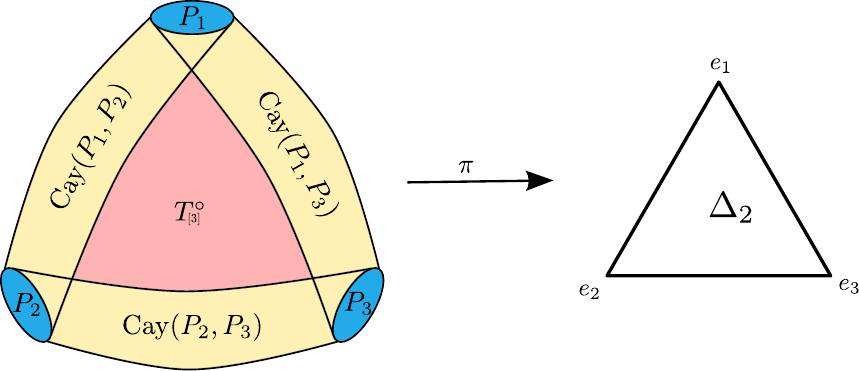}
    \caption{An illustration of the Cayley polytope for
    $P_{[3]}=(P_1,P_2,P_3)$ including the projection to the simplex.}
    \label{fig:cayley_sketch}
\end{figure}

\begin{prp}\label{prp:CP}
    Let $\Cay(P_{[m]})$ be the Cayley polytope associated to
    $P_{[m]}=(P_1,\dots,P_m)$, and let $\pi: \Cay(P_{[m]}) \rightarrow \K_{m-1}$ denote the projection of Cayley polytopes~\eqref{eq:projection}.
    \begin{enumerate}[\rm (i)]
        \item For $\lambda \in \relint \Delta_J$, $\pi^{-1}(\lambda)$ is
            combinatorially equivalent to $ \sum_{i \in J} P_i $.
        \item For any $J \subseteq [m]$ 
            \[
                \pi^{-1}(\Delta_J) \ \cong \ \Cay(P_J).
            \]
        \item If all polytopes $P_i$ are of the same dimension $d$, i.e., if $P_{[m]}$ is pure, then
            \[
                \dim \Cay(P_J) \ = \ d + |J| - 1
            \]
            for all $J \subseteq [m]$.
        \item If $P_{[m]}$ is a family of simplicial polytopes in relatively
            general position, then the only non-simplex faces of $\Cay(P_{[m]})$ are
            $\Cay(P_J)$ for all $\emptyset \neq J \subseteq [m]$.
    \end{enumerate}
\end{prp}

The proposition suggests that the boundary of the Cayley polytope
$\Cay(P_{[m]})$ is stratified along the facial structure of the
$(m-1)$-simplex. We define the \Defn{Cayley complex} $\mr{T}_{[m]} =
\mr{T}(P_{[m]})$ as the closure of $\pi^{-1}(\relint \Delta_{m-1}) \cap \partial
\Cay(P_{[m]})$. 

Then a family $P_{[m]}$ of simplicial
polytopes is in relatively general position if and only if $\mr{T}_{[m]}$ is a simplicial complex.  For a subset $S \subseteq [m]$,
let us write $\mr{T}_S := \mr{T}(P_S)$, and
$\mr{T}_{\emptyset}=\emp$. It is easy to see that the boundary of $\mr{T}_{[m]}$ is
covered by the Cayley complexes $\mr{T}_{J}$ for $J \subsetneq [m]$ and we define the
\Defn{Cayley arrangement} as $\scr{T} := \{ \mr{T}_{[m] {\setminus} j}: j \in [m]
\}$. 

\begin{ex} Consider  $P_{[3]}=(P_1,P_2,P_3)$ a family of three distinct pentagons. Then $\cfs{T}$ is a two-dimensional torus, cf.\ Figure~\ref{fig:Cayley_boundary},
which is glued from the Cayley complexes for $\Cay(P_i,P_j)$, $1 \le i < j \le
3$, and $\mr{T}_{[3]}$ is the complementary three-dimensional torus in the Cayley polytope $\Cay(\mr{T}_{[3]})$.
\begin{figure}[htb]
    \includegraphics[width=7.5cm]{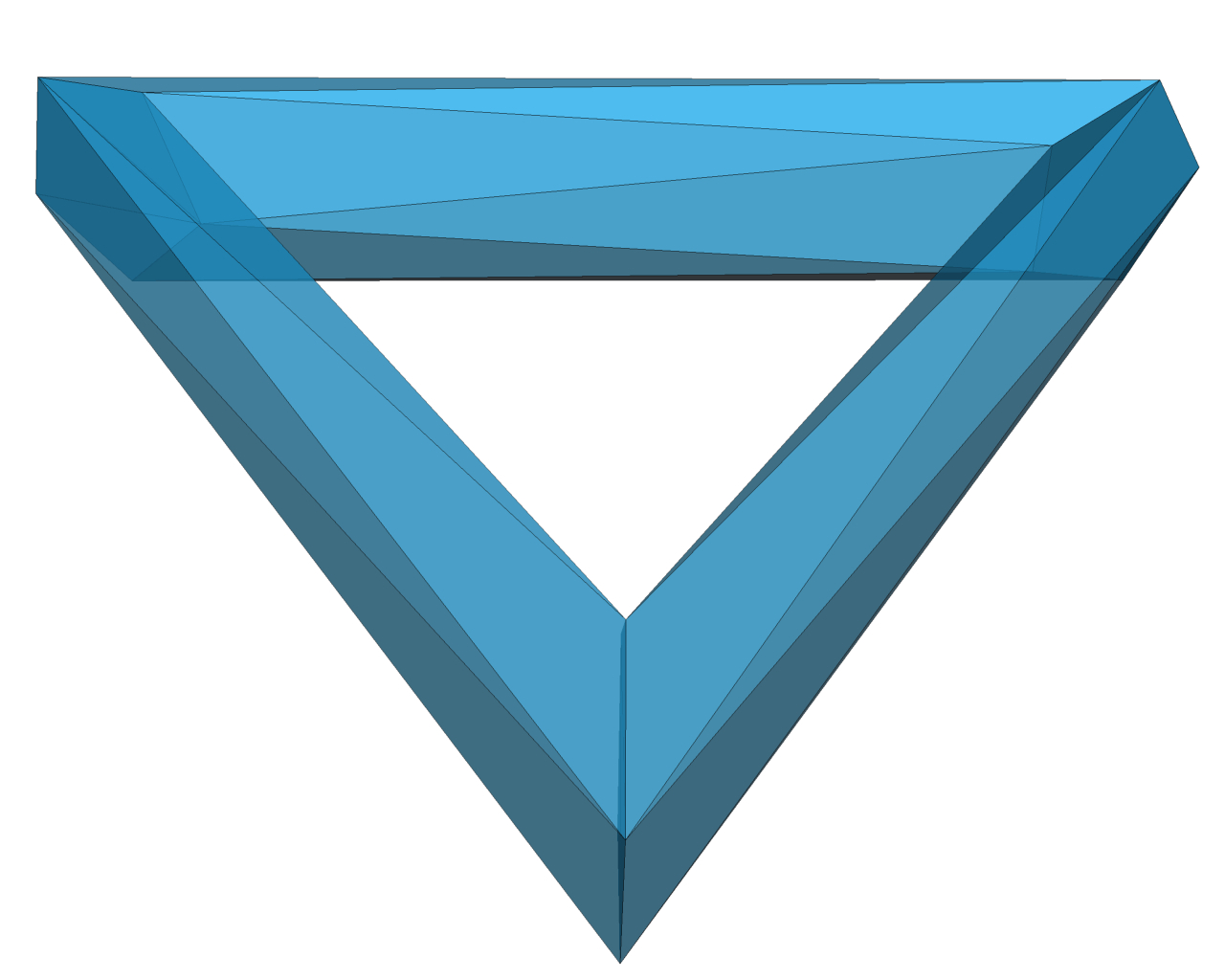}
    \caption{Boundary of the Cayley complex for $P_{[3]}=(P_1,P_2,P_3)$ being three
        general position pentagons in $\R^2$. The Cayley complex $\mr{T}_{[3]}$ is a solid torus in the three-sphere $\partial \Cay(\mr{T}_{[3]})$}
\label{fig:Cayley_boundary}
\end{figure}
\end{ex}

Finally, we define the \Defn{relative Cayley complex} as
\[
    \mr{T}^\circ_{[m]} \ := \ \Bigl(\mr{T}_{[m]}, \cfs{T} = \bigcup_i
    \mr{T}_{[m]{\setminus} i} \Bigr)
\]
and consequently $\mr{T}^\circ_{\emptyset}=\varnothing$.
For $S \subseteq [m]$, we define the restrictions $\mr{T}_S, \scr{T}_S,$ and
$\mr{T}^\circ_S$ analogously.

To apply our techniques, it remains to see that the topological properties of
the Cayley complex are well-behaved.

\begin{prp}\label{prp:CP2} 
    Let $P_{[m]} = (P_1,\dots,P_m)$ be a family of simplicial $d$-polytopes in
    relatively general position. Let $\mr{T}_{[m]} = \mr{T}(P_{[m]})$ be the
    corresponding Cayley
    complex and $\scr{T}$ the Cayley arrangement.
    \begin{enumerate}[\rm (i)]
        \item For $0 \le k \le d-1$
            \[
            f_{k+m-1}(\mr{T}^\circ_{[m]}) \ = \ f_{k}(|P_{[m]}|).
            \]
        \item $\scr{T}$ is an arrangement of full subcomplexes of $\mr{T}_{[m]}$.
        \item $\mr{T}^\circ_{[m]} = (\mr{T}_{[m]},\cfs{T})$ is relative Buchsbaum. In fact, $\mr{T}_{[m]}$ is a manifold, and $\cfs{T}$ is its boundary.
        \item $(\mr{T}_{[m]}, \scr{T})$ is universally Cohen--Macaulay.
        \item We have $\beta_i(\mr{T}^\circ_{[m]})=\mathbf{1}_{i=m-1}+\mathbf{1}_{i=d+m-2}$.
    \end{enumerate}
\end{prp}
\begin{proof}
    The relative complex $\mr{T}^\circ_{[m]}$ is exactly the set of faces of $\Cay(P_{[m]})$ for
    which the intersection with $\pi^{-1}(\lambda)$ is nonempty for any
    $\lambda \in \relint\Delta_{m-1}$.  For any such $\lambda$ the
    intersection of $\pi^{-1}(\lambda)$ is normally equivalent to $|P_{[m]}|$ (i.e., their normal fans coincide). Hence,
    the $f$-vector of $\mr{T}^\circ_{[m]}$ is the shifted $f$-vector of
    $|P_{[m]}|$ which proves (i).
    
    For (ii) note that $\mr{T}_S$ is the subcomplex of $\mr{T}_{[m]}$ induced by the vertices $\V(P_i), i\in S$.

    Let $W = \pi^{-1}(\partial \Delta_{m-1}) \subsetneq \partial \Cay(P_{[m]})$
    be the shadow boundary, which can be seen as a subset of $\R^d\times \partial \Delta_{m-1}$.
    The fibers $\pi^{-1}(x),\ x\in \partial \Delta_{m-1}$ are convex and of dimension $d$; hence $W$ is a full-dimensional submanifold of $\R^d\times \partial \Delta_{m-1}$ that collapses to $\partial \Delta_{m-1}\cong S^{m-1}$. It follows in particular that $W$ is in fact homeomorphic to $B_d \times S^{m-1}$ where $B_d$ is a $d$-ball. 
    
    For (v), we can use excision to compute
    $\rHom_{\bullet}(\mr{T}_{[m]},\cfs{T}) \cong \rHom_{\bullet}(\partial \Cay(P_{[m]}),
    W)$. The same argument applied to (relative) links then shows (iii) and~(iv).
\end{proof}

So, for a proof of Theorem~\ref{thm:UBTMS} it is sufficient to find tight
upper bounds on the $h$-vector of the Cayley complex. To emphasize the
relation to $P_{[m]}$, we define the \Defn{$h$-vector} of a simplicial family
in relatively general position as
\[
    h_\ast(P_{[m]}) \ := \ h_\ast(\mr{T}_{[m]}^\circ).
\]
In particular, $\mr{T}_{\{i\}} = \partial P_i$. For two summands, the relative Cayley
complex is a cylinder over a sphere relative to its boundary; cf.\
Section~\ref{ssec:two}. 

\subsection{Initial terms of the $h$-vector}\label{ssec:small_h} 

In the proof of the UBT, it is only necessary to find tight upper bounds on
$h_k$ for $k \le \lfloor\frac{d}{2}\rfloor$ and let the Dehn--Sommerville
equations take care of the rest. In this section we find bounds for $h_k$ 
for $k \le \frac{d-m+1}{2}$. For higher $k$, we will
also employ suitably generalized versions of the Dehn--Sommerville equations
which we treat in the next section. In contrast to the case of spheres, we
will need bounds on $g$-vectors (and more).

\begin{thm}\label{thm:minkit}
    Let $P_{[m]} = (P_1,\dots,P_m) $ be a family of simplicial $d$-polytopes
    in $\R^d$ with Cayley complex $\mr{T}_{[m]} = \mr{T}(P_{[m]})$. Then we have
\[  (k+m) g_{k+m} (\mr{T}_{[m]}^\circ) \ \le \  
    (f_0(\mr{T}_{[m]})-d-m) h_{k+m-1} (\mr{T}_{[m]}^\circ)+\sum_{i=1}^m 
    f_0(\mr{T}_{\{i\}}) g_{k+m-1}(\mr{T}_{[m]{\setminus} \{i\}}^\circ). 
\]
We have  
\[(i+|S|)  g_{i+|S|} (\mr{T}_{S}^\circ) \ = \ (f_0(\mr{T}_{S})-d-|S|)  h_{i+|S|-1} (\mr{T}_{S}^\circ)+\sum_{i\in S} f_0(\mr{T}_{\{i\}}) g_{i+|S|-1}(\mr{T}_{S{\setminus} \{i\}}^\circ)\] for all $i\le k_0$ and $S\subseteq[m]$ if and only if all non-faces of $\mr{T}_{S}$ of dimension $< k_0+|S|$ are supported in some $\V(\mr{T}_R)$, $R\subsetneq S$. 
\end{thm}

\begin{proof}
By Proposition~\ref{prp:CP2}(iv), the pair $(\mr{T}_{[m]},\scr{T})$ is universally
Cohen--Macaulay. By Corollary~\ref{cor:loc}(2), we conclude at once that
\[(k+m)  h_{k+m} (\mr{T}_{[m]}^\circ) \ \le \  
    (f_0(\mr{T}_{[m]})-d+k)  h_{k+m-1} (\mr{T}_{[m]}^\circ)+\sum_{i=1}^m
        f_0(\mr{T}_{\{i\}}) g_{k+m-1}(\mr{T}_{[m]{\setminus} \{i\}}^\circ).
\]
The desired inequality follows by subtracting $(k+m)  h_{k+m-1} (\mr{T}_{[m]}^\circ)$ on both sides.  The characterization for the case of equality follows with iterative application of the characterization in Corollary~\ref{cor:loc}(2).
\end{proof}

Theorem~\ref{thm:minkit} is the key to the UBTM. An alternative, geometric
proof can be given rather elegantly using relative shellability: It is a consequence of the
work of Bruggesser--Mani \cite{BM}, Proposition~\ref{prp:shell_to_h} and
Alexander duality of shellings provided in \cite{AB-SSZ}. This program has been implemented to some extent in~\cite{MR3392772}.

The theorem directly enables us to give (tight) upper bounds on small
$h$-entries. The following corollary is a direct consequence of
Theorem~\ref{thm:minkit}. We present 
an alternative, direct proof by change of presentation (see
Section~\ref{ssec:change_presentation}).

\begin{cor}\label{cor:minkh}
    Let $P_{[m]} = (P_1,\dots,P_m)$ be a family of $m$ simplicial
    $d$-polytopes in relatively general position and let $\mr{T}_{[m]}^\circ$ be
    the corresponding relative Cayley complex.  Then for all $-m+1\le k\le d$
    \[
        h_{k+m-1} (\mr{T}_{[m]}^\circ) \ \le \  \sum_{S\subseteq [m]}
        (-1)^{m-|S|} \binom{f_0(\mr{T}_S)-d+k-1}{k+m-1},
    \]
Equality holds for some $k_0+m-1$ if and only if all non-faces of $\mr{T}_{[m]}$ of
dimension $< k_0+m-1$ are supported in some $\V(\mr{T}_S)$, $S\subsetneq [m]$. 
\end{cor}

\begin{proof}
By Proposition~\ref{prp:CP2}(iii), the relative complex $\mr{T}^\circ_{[m]} =
(\mr{T}_{[m]},\cfs{T})$ is Buchsbaum and thus we can apply Theorem~\ref{thm:RS}.

For the topological contribution, we use Proposition~\ref{prp:CP2}(v) to infer
that all Betti numbers are zero except for 
$\rBetti_{m-1}(\mr{T}^\circ_{[m]})= 1$ and $\rBetti_{d+m-2}(\mr{T}^\circ_{[m]})= 1$. Hence,
for $k \ge 2$ 
\[
    \GT_{k+m-1}(\mr{T}_{[m]}^\circ) \ = \ 
    \binom{d+m-1}{d-k} \sum_{i=0}^{k+m-2}(-1)^{k+m-1-i}
    \rBetti_{i-1}(\mr{T}_{[m]}^\circ) \ = \ \mathbf{1}_{k\ge 2} 
    (-1)^{k-1}
    \binom{d+m-1}{d-k}.
\]
For the algebraic component $\LT_{k+m-1}(\mr{T}_{[m]}^\circ)$, recall from
Proposition~\ref{prp:CP2}(ii) that $\scr{T}$ is an arrangement of full
subcomplexes of $\mr{T}_{[m]}$.  Hence, for $\FM = \FM[\mr{T}_{[m]}^\circ,\cfs{T}]$,
nerve ideal $\ch = \ch[\mr{T}_{[m]},\cfs{T}]$, and a \lsop\ $\Theta$ of
length $\ell = d+m-1$ we obtain for $-m+1 \le k \le d$ 
\[
    \LT_{k+m-1}(\mr{T}_{[m]},\cfs{T}) \ = \ \dim_\k (\FM/\Theta\FM)_{k+m-1} \ \le \
    \dim_\k (\ch/\Theta \ch)_{k+m-1}
\]
by Theorem~\ref{thm:upb}.  Now, upon closer examination, we notice that $\ch$
is generated by squarefree monomials corresponding to subsets $\tau \subseteq
\bigcup_i \V(P_i \times e_i)$ such that $\tau \cap \V(P_i \times e_i) \neq
\emptyset$ for all $i$. Hence, $\ch$ is generated in degree $m$ by $\prod_i
f_0(P_i)$ minimal generators.  That is, the coarse nerve $\mathfrak{N}$ with
Stanley--Reisner ideal $\ch$ is
\[
    \mathfrak{N} \ = \ \bigcup_{S\subsetneq [m]} \bigast_{\ \ i\in S} \Delta_{\V(P_i
    \times e_i)}.
\]
 While $\mathfrak{N}$ is in general not Buchsbaum, its
$(d+m-2)$-skeleton is. In particular $\mathfrak{N}^{(d+m-1)}$ has homology only in
dimensions $d+m-2$ and $m-2$ and $\rBetti_{m-2}(\mathfrak{N}^{(d+m-2)}) = 1$.  By
Corollary~\ref{cor:cmincm} and Theorem~\ref{thm:topb}, the relative complex
$(\K_{[n]}^{(d+m-1)},\mathfrak{N}^{(d+m-1)})$ is Buchsbaum.  Moreover, the intersection
poset $\mc{P} = \mc{P}(\scr{T})$ of the arrangement coincides with the dual to
the face poset of $\K_{[m]}$ and hence 
\[
    \mu_{\mc{P}}(\mr{T}_{[m]},\mr{T}_S)=(-1)^{m-|S|}
\]
for all $S \subseteq [m]$.  We can now use Theorem~\ref{thm:RSI} and
Proposition~\ref{prp:bi} to evaluate
\[\dim_\k (\ch/\Theta \ch)_{k+m-1}\ = \ \sum_{S\subseteq [m]} (-1)^{m-|S|}
\binom{f_0(\mr{T}_S)-d+k-1}{k+m-1}-(-1)^{k-1} \mathbf{1}_{k\ge
2}\binom{d+m-1}{d-k}.\]

 Consider the cover
$\widehat{\scr{T}}:=\{{\mr{T}}_{\overline i}:=\mr{T}_{[m]\setminus \{i\}} : i\in  [m]\}$ of $\partial \mr{T}_{[m]}=\cfs{T}$. For tightness, notice that for every strict subset $R\subsetneq [m]$ of $\widehat{\scr{T}}$, $(\mr{T}_{[m]},\mr{T}_{[m]}\cap \bigcap_{i\in R} {\mr{T}}_{\overline i})$ is acyclic, that is, its Betti numbers are trivial. On the other hand, 
\[\bigg(\mr{T}_{[m]},\mr{T}_{[m]}\cap \bigcap_{i\in [m]} {\mr{T}}_{\overline i}\bigg)\ =\ \left(\mr{T}_{[m]},\{\emptyset\}\right)\]
has only one nontrivial homology group, that in dimension $d-1$.
Hence, $\widehat{\scr{T}}$ is $(d+m-2)$-magnificent w.r.t.\ $(\mr{T}_{[m]},\cfs{T})$ in the sense of Theorem~\ref{thm:magnificent}. Tightness follows with Theorem~\ref{thm:upb} and Theorem~\ref{thm:reduction_to_topology}.
\end{proof}

\subsection{The Dehn--Sommerville formula and other linear relations.}\label{ssec:DSM} 

Let us give some linear relations among the $h$-vectors of our particular
simplicial complexes.  To give Dehn--Sommerville-type relations among the
entries of the $h$-vector of the relative  Cayley complex, it will prove
useful to renormalize the $h$-vector to
\[
    \widetilde{h}_{k+m-1} (\mr{T}_{[m]}^\circ) \ := \ h_{k+m-1}
    (\mr{T}_{[m]}^\circ)+(-1)^{k}\binom{d+m-1}{k+m-1}
\] 
and 
\[
    \widetilde{g}_{k+m-1}(\mr{T}_{[m]}^\circ)\  :=\
    \widetilde{h}_{k+m-1}(\mr{T}_{[m]}^\circ) \ -\
    \widetilde{h}_{k+m-2}(\mr{T}_{[m]}^\circ).
\] 
On a purely enumerative level, this corresponds to setting the number of
empty faces of $\mr{T}_{[m]}^\circ$ to 
\[
    \widetilde{f}_{-1} (\mr{T}_{[m]}^\circ)\ = \ \widetilde{h}_{0}
    (\mr{T}_{[m]}^\circ)\ = \ (-1)^{m-1}\ = \ -\mu_{\mc{P}(\scr{T}_{[m]})}(\mr{T}_{[m]},
    \emptyset).  
\] 
With this, we can compute $\widetilde{h}$ from $(\widetilde{f}_{-1},f_0,\cdots)$ as usual and rewrite Corollary~\ref{cor:minkh} as
\[
    \widetilde{h}_{k+m-1} (\mr{T}_{[m]}^\circ) \ \le \  \sum_{\emptyset \neq
    S\subseteq [m]} (-1)^{m-|S|} \binom{f_0(\mr{T}_S)-d+k-1}{k+m-1}.
\]  An important ingredient to our approach is, once
again, Dehn--Sommerville duality.
\begin{lem}\label{lem:DSM} 
    Let $P_{[m]}$ be a pure collection of $m$ polytopes in $\R^d$, $d\ge 1$,
    such that the Cayley complex $\mr{T}(P_1,\dots,P_m)$ is simplicial. Then,
    for all $-m+1\le k\le d$, we have 
    \[
        h_{d-k}(\mr{T}_{[m]})\ = \ \widetilde{h}_{k+m-1}
        (\mr{T}_{[m]}^\circ) \ = \ h_{k+m-1}
        (\mr{T}_{[m]}^\circ)+(-1)^{k}\binom{d+m-1}{d-k}.
    \] 
\end{lem} 
\begin{proof} 
    By Proposition~\ref{prp:CP2}, the complexes $\mr{T}_{[m]}^\circ$ and
    $\mr{T}_{[m]}$ are homology manifolds. Moreover,
    $\rBetti_{m-1}(\mr{T}^\circ_{[m]})= 1$ and
    $\rBetti_{d+m-2}(\mr{T}^\circ_{[m]})= 1$, but
    $\rBetti_{i}(\mr{T}^\circ_{[m]})\equiv 0$ otherwise. The assertion now
    follows by Lemma~\ref{lem:DS}.  \end{proof} 
    
    This allows us to translate
    from bounds on ${h}_{\ast}(\mr{T}_{[m]})$ to bounds on
    ${h}_{\ast}(\mr{T}_{[m]}^\circ)$, and vice versa. Let us define 
\[    \widetilde{g}_{k+m-1}^{\langle\ell\rangle} (\cdot) \ := \ \sum_{i=0}^{\ell}
    (-1)^i \binom{\ell}{i} \widetilde{h}_{k+m-1-i}(\cdot) \text{ and }\\ 
    {g}_{k+m-1}^{\langle\ell\rangle} (\cdot) \ := \ \sum_{i=0}^{\ell} (-1)^i
    \binom{\ell}{i} {h}_{k+m-1-i}(\cdot).\]
With this, we have the following elementary relations. 
\begin{prp}\label{prp:recursive} 
    For $\mr{T}_{[m]}$ as above, any $\ell\ge 0$, and any $1\le s\le m$, we have 
    \[ 
        h_{k+m-1}(\mr{T}_{[m]}) \ =\ \sum_{S\subseteq [m]}
        \widetilde{g}_{k+m-1}^{\langle m-|S|\rangle} (\mr{T}_{S}^\circ). 
    \] 
\end{prp}
\begin{proof}
    From the stratification of the Cayley complex $\mr{T}_{[m]}$ into the open
    Cayley complexes $\mr{T}_{S}^\circ$, $S\subseteq [m]$, it follows by
    linearity of $h$-vector that
\begin{equation}\label{eq:splitt}
         h_{k+m-1}(\mr{T}_{[m]})\ =\ \sum_{S\subseteq [m]} {g}_{k+m-1}^{\langle
        m-|S|\rangle} (\mr{T}_{S}^\circ).
 \end{equation}
    Now, observe that the binomial correction terms, when passing from $g_\ast$ to $\widetilde{g}_\ast$, cancel out in the sum~\eqref{eq:splitt} since $\sum_{S\subseteq [m]} -1^{S}=0$, so that
    \[ 
        \sum_{S\subseteq [m]} {g}_{k+m-1}^{\langle m-|S|\rangle}
        (\mr{T}_{S}^\circ)\ =\ \sum_{S\subseteq [m]}
        \widetilde{g}_{k+m-1}^{\langle m-|S|\rangle} (\mr{T}_{S}^\circ).
        \qedhere 
    \] 
\end{proof} \enlargethispage{5mm}
To visualize Proposition~\ref{prp:recursive}, note that every $g$-vector entry $\widetilde{g}_{k+m-1}^{\langle m-|S|\rangle} (\mr{T}_{S}^\circ)$ can be written 
as a sum of $h$-numbers $\widetilde{h}_{i}^{\langle m-|S|\rangle} (\mr{T}_{S}^\circ)$. Hence, we can encode the formula of Proposition~\ref{prp:recursive} by recording the coefficients of these $h$-numbers in the following table:
\begin{table}[h!tb]
\small
 \centering
$\arraycolsep=10pt\def\arraystretch{1.1}
\begin{array}{c|cccccc}
  & |S| = m & m-1 & m-2 & m-3 & m-4 &  \cdots \\ \hline 
 k+m-1 & \ 1 & {\ 1} & {\ 1} & {\ 1} & {\ 1}& \cdots \\ 
 k+m-2 & \ 0 & \mbox{-}1 & \mbox{-}2 & \mbox{-}{3} & \mbox{-}{4} &\cdots \\ 
 k+m-3 & \ 0 & {\ 0} & \ 1 & \ 3 & {\ 6}& \cdots \\ 
 k+m-4 & \ 0 & {\ 0} & \ 0 & \mbox{-}1 &  \mbox{-}4 &\cdots \\ 
 \vdots & \vdots & \vdots & \vdots  & \vdots & \vdots &\ddots
\end{array}$
\vskip 2mm
\caption{Visualization of Proposition~\ref{prp:recursive}, recording the coefficients for $h$-numbers.}\label{tab:DS1}
\end{table}

For $k, m, d\in \Z$, let us 
now define ${\mr{c}}:\Z^3\rightarrow \Z$ by 
\[
    \mr{c}'(k, m, d) \ := \ 2k+2m-1-d \quad \Longleftrightarrow \quad
    k+m-1 \ = \ \frac{d+\mr{c}'(k, m,
    d)-1}{2},
\] and 
\[
    \mr{c}(k, m, d) \ := \ \min \{m, \max \{\mr{c}'(k, m, d),1 \}\}.
\]
Let us furthermore denote the \Defn{covering relation} by $\prec$, i.e., for
$R, S\subseteq [m]$ with $R \subseteq S$ we write $R \prec S$ to denote the fact that there is no set $Q$ with $R\subsetneq Q\subsetneq S$.

\begin{lem}\label{lem:recursive} 
    For $\mr{T}_{[m]}$ as above, any $1\le s\le m$, and
  $k+m-1\le \frac{d+m-1}{2}$, and with $\mr{c}=\mr{c}(k,m,d)$ we have 
    \begin{align*} 
 &  \ {h}_{k+m-1}(\mr{T}_{[m]})\\
     =&\ \sum_{j=0}^{\lfloor\nicefrac{m}{2}\rfloor}
     \sum_{s=\mr{c}-2j+1}^{m-2j} \sum_{\substack{S\subseteq [m]\\ |S|=s}}
     \binom{m-s}{2j}\Bigg(\widetilde{h}_{k+m-1-2j}(\mr{T}_S^\circ)-\frac{1}{2j+1}
     \sum_{\substack {R\prec S}}\widetilde{h}_{k+m-2-2j}(\mr{T}_R^\circ)\Bigg)
     \tag{A}\label{eqn:A}\\ 
      +&\ 
\sum_{j=0}^{\lfloor\nicefrac{m}{2}\rfloor} \sum_{\substack{S\subseteq [m]\\
|S|=\mr{c}-2j}} \tag{B}\label{eqn:B} 
\binom{m-|S|-1}{2j}\Bigg(\widetilde{h}_{k+m-1-2j}(\mr{T}_S^\circ)-\frac{m-S}{(m-|S|+1)(2j+1)} 
\sum_{\substack {R\prec S}}\widetilde{h}_{k+m-2-2j}(\mr{T}_R^\circ)\Bigg). 
\end{align*} 
\end{lem}

\begin{proof} 
    We use induction on $m$, the case $m=1$ being trivial as the only term in
    the sums~\eqref{eqn:A} and~\eqref{eqn:B} is 
    $
        \widetilde{h}_{k}(\mr{T}_{[1]}^\circ) \  =  \ h_{k}
        (\mr{T}_{[1]}^\circ)+(-1)^{k}\binom{d}{k} \ = \ {h}_{k}(\mr{T}_{[1]}).
    $
    For $m>1$, note that by  Proposition~\ref{prp:recursive} and the
    definition of~$\widetilde{g}$, we
    have 
\[
    h_{k+m-1}(\mr{T}_{[m]}) \  =\ 
    \sum_{S\subseteq [m]} \widetilde{g}_{k+m-1}^{\langle m-|S|\rangle}
        (\mr{T}_{S}^\circ)
    \ = \  \sum_{\substack{S\subseteq [m]}} \sum_{i=0}^{m-|S|} 
(-1)^{i} \binom{m-|S|}{i} \widetilde{h}_{k+m-1-i}(\mr{T}_{S}^{\circ}). 
\]
The coefficients of $\widetilde{h}_{\ast}(\mr{T}_{S}^{\circ})$ are summarized in Table~\ref{tab:DS1}.
We split this last sum into two subsums, cf.\ Table~\ref{tab:DS2}, and obtain
\begin{align*} \ &\ 
\sum_{\substack{S\subseteq [m]}}\ \
\sum_{i=\mr{c}-|S|+1}^{m-|S|} (-1)^{i} \binom{m-|S|}{i} 
\widetilde{h}_{k+m-1-i}(\mr{T}_{S}^{\circ}) \tag{$\upalpha$}\label{eqn:ddag}. 
\\ +\ &\  
\sum_{\substack{S\subseteq [m]}} \ \ 
\sum_{i=0}^{\mr{c}-|S|} (-1)^{i} \binom{m-|S|}{i}
\widetilde{h}_{k+m-1-i}(\mr{T}_{S}^{\circ}) \tag{$\upbeta$}\label{eqn:dag}
\end{align*}

\begin{table}[h!tb]
\small
 \centering
$\arraycolsep=10pt\def\arraystretch{1.1}
\begin{array}{c|cccccc}
  & |S| = m & m-1 & m-2 & m-3 & m-4 &  \cdots \\ \hline 
 k+m-1 & \ 1 & {\ 1} & \hlightb{\ 1} & \hlightb{\ 1} & \hlightb{\ 1}& \cdots \\ 
 k+m-2 & \ 0 & \mbox{-}1 & \mbox{-}2 & \hlightb{-3} & \hlightb{-4} &\cdots \\ 
 k+m-3 & \ 0 & {\ 0} & \ 1 & \ 3 & \hlightb{\ 6}& \cdots \\ 
 k+m-4 & \ 0 & {\ 0} & \ 0 & \mbox{-}1 &  \mbox{-}4 &\cdots \\ 
 \vdots & \vdots & \vdots & \vdots  & \vdots & \vdots &\ddots
\end{array}$
\vskip 2mm
\caption{Example of the splitting into two subsums for
$k+m-1=\frac{d+m-3}{2}$. Summands of~\eqref{eqn:dag} highlighted.}\label{tab:DS2}
\end{table}

Notice now that in the sum~\eqref{eqn:dag}, whenever $\widetilde{h}_{k'}(\mr{T}_{S}^{\circ})$ is 
evaluated, then $k'\ge \frac{d+|S|-1}{2}$. Therefore, we may use the 
Dehn--Sommerville equations to substitute~\eqref{eqn:dag} by 
\[\sum_{\substack{S\subseteq [m]}} 
\sum_{i=0}^{\mr{c}-|S|} (-1)^{i} \binom{m-|S|}{i} 
{h}_{d+|S|-k-m}(\mr{T}_{S})\] with $d+|S|-k-m\le \frac{d+|S|-1}{2}$. We may now evaluate ${h}_{|S|-m+i}(\mr{T}_{S})$ using the induction assumption, as illustrated in Table~\ref{tab:DS3}. 
We obtain, after also rewriting $\upalpha$ to (A'), a decomposition
\begin{align*} 
   &\ h_{k+m-1}(\mr{T}_{[m]}) \\
       \ =& \ \sum_{\substack{S\subseteq [m]}}\ \
\sum_{i=\mr{c}-|S|+1}^{m-|S|} (-1)^{i} \binom{m-|S|}{i} 
\widetilde{h}_{k+m-1-i}(\mr{T}_{S}^{\circ}) \\ 
    \ +&\ \sum_{j=0}^{{m}} \sum_{\substack{S\subseteq [m]\\ |S|=\mr{c}-j}}
    (-1)^j\binom{m-|S|-1}{j} 
    \widetilde{h}_{k+m-1-j}(\mr{T}_S^\circ). \\
    \ =& \ \sum_{j=0}^{{m}} \sum_{s=\mr{c}-j+1}^{m-j}
    \sum_{\substack{S\subseteq
        [m]\\ |S|=s}} (-1)^j\binom{m-s}{j}
        \widetilde{h}_{k+m-1-j}(\mr{T}_S^\circ) \tag{A'}\label{eqn:A'} \\ 
    \ +&\ \sum_{j=0}^{{m}} \sum_{\substack{S\subseteq [m]\\ |S|=\mr{c}-j}}
    (-1)^j\binom{m-|S|-1}{j} 
    \widetilde{h}_{k+m-1-j}(\mr{T}_S^\circ)\tag{B'}\label{eqn:B'}. 
\end{align*}
\begin{table}[h!tb]
\small
 \centering
$\arraycolsep=10pt\def\arraystretch{1.1}
\begin{array}{c|cccccc}
  & |S| = m & m-1 & m-2 & m-3 & m-4 &  \cdots \\ \hline 
 k+m-1 & \ 1 & {\ 1} & \hlightb{\ 1} & \hlightb{\ 0} & \hlightb{\ 0}& \cdots \\ 
 k+m-2 & \ 0 & \mbox{-}1 & \mbox{-}2 & \hlightb{-2} & \hlightb{\ 0} &\cdots \\ 
 k+m-3 & \ 0 & {\ 0} & \ 1 & \ 3 & \hlightb{\ 3}& \cdots \\ 
 k+m-4 & \ 0 & {\ 0} & \ 0 & \mbox{-}1 &  \mbox{-}4 &\cdots \\ 
 \vdots & \vdots & \vdots & \vdots  & \vdots & \vdots &\ddots
\end{array}$
\vskip 2mm
\caption{Applying the Dehn--Sommerville relations and induction simplifies the array.}\label{tab:DS3}
\end{table}

We now pair summands with positive coefficient (specifically  summands of $\widetilde{h}_{k+m-1-j}(\mr{T}_S^\circ)$, $j$ even) with summands $\widetilde{h}_{k+m-2-j}(\mr{T}_R^\circ)$, where $R\prec S$ (see also Table~\ref{tab:DS4}).

\begin{table}[h!tb]
\small
 \centering
$\arraycolsep=10pt\def\arraystretch{1.1}
\begin{array}{c|cccccc}
  & |S| = m & m-1 & m-2 & m-3 & m-4 &  \cdots \\ \hline 
 k+m-1 & \ 1 & \ {1} & {\ 1} & {\ 0} & {\ 0}& \cdots \\ 
 k+m-2 & \ 0 & \mbox{-}1 & {\mbox{-}2} & {-2} & {\ 0} &\cdots \\ 
 k+m-3 & \ 0 & {\ 0} & \ \hlightr{1} & \ {3} & {\ 3}& \cdots \\ 
 k+m-4 & \ 0 & {\ 0} & \ 0 & \hlightr{\mbox{-}1} &  {\mbox{-}4} &\cdots \\ 
 \vdots & \vdots & \vdots & \vdots  & \vdots & \vdots &\ddots
\end{array}$
\caption{Pairing positive and negative terms of the sums~\eqref{eqn:A'}
and~\eqref{eqn:B'}, corresponding to sets $S \subseteq [m]$ and $R\prec S$.}\label{tab:DS4}
\end{table}

We see 
that the sum~\eqref{eqn:A'} is equal to \[\sum_{j=0}^{\lfloor\nicefrac{m}{2}\rfloor} 
\sum_{s=\mr{c}-2j+1}^{m-2j} \sum_{\substack{S\subseteq
[m]\\ |S|=s}} 
\binom{m-s}{2j}\Bigg(\widetilde{h}_{k+m-1-2j}(\mr{T}_S^\circ)-\frac{1}{2j+1} 
\sum_{\substack {R\prec S}}\widetilde{h}_{k+m-2-2j}(\mr{T}_R^\circ)\Bigg)\]
and the sum~\eqref{eqn:B'} equals \[ 
\sum_{j=0}^{\lfloor\nicefrac{m}{2}\rfloor} \sum_{\substack{S\subseteq [m]\\ |S|=\mr{c}-2j}} 
\binom{m-|S|-1}{2j}\Bigg(\widetilde{h}_{k+m-1-2j}(\mr{T}_S^\circ)-\frac{m-|S|}{(m-|S|+1)(2j+1)} 
\sum_{\substack {R\prec S}}\widetilde{h}_{k+m-2-2j}(\mr{T}_R^\circ)\Bigg). \qedhere\] 

\end{proof}

We conclude in particular that it is sufficient to know the various initial
$h$-vector entries.

\begin{cor}\label{cor:lh_hh}
    The $h$-vector entries $h_{k+m-1}(\mr{T}_{[m]}^\circ)$ for  $k+m-1>
    \lfloor\frac{d+m-1}{2}\rfloor$ are determined by
    $\wt{h}_{k'+m-1}(\mr{T}_{S}^\circ)$ for $k'+m-1\le
    \lfloor\frac{d+m-1}{2}\rfloor$ and $S \subseteq [m]$.
\end{cor}

Passing from $h$-vectors to $f$-vectors and noting that Minkowski neighborly
families maximize the `small' $h$-entries in the sense of
Corollary~\ref{cor:minkh}, we immediately conclude Proposition
\ref{prop:M-neigh-fvector}.

\subsection{The Upper Bound Theorem for Minkowski sums}\label{ssec:UBTM}

We can finally give sharp and explicit bounds for $h_k(\mr{T}_{[m]}^\circ)$.
Let us define the functions 
\[\widetilde{\MUBT}^{\langle\ell\rangle}:\Z^{[m]}\times\Z\times\Z\longrightarrow \Z,\ m\ge 0,\ \ell\ge 0\] 
(where we abbreviate $\widetilde{\MUBT}:=\widetilde{\MUBT}^{\langle0\rangle}$), and
\[{\MUBT}:\Z^{[m]}\times\Z\times\Z\longrightarrow \Z,\ m\ge 0\] 
by the following conditions:
\begin{compactenum}[\rm (a)]
\item \textbf{Basic relation:} For all $k,\,d\ge 0$ 
    \[
        \widetilde{\MUBT}^{\langle\ell\rangle}
        (\cdot,\cdot,k) \ = \ \widetilde{\MUBT}^{\langle \ell-1 \rangle}
        (\cdot,\cdot,k)-\widetilde{\MUBT}^{\langle \ell-1
        \rangle}(\cdot,\cdot,k-1);
    \]
\item \textbf{Linearity:} For all $-m+1\le k\le d$
    \[
        {\MUBT}(\alpha,d,k+m-1)\ = \ \sum_{S\subseteq [m]}
        \widetilde{\MUBT}^{\langle m-|S| \rangle} (\alpha_{S},d,k+m-1),
    \]
    where $\alpha_{S}\in \Z^{S}$ is the restriction of $\alpha\in \Z^{[m]}$
    to the index set $S\subseteq [m]$;
\item \textbf{Dehn--Sommerville relation:} For all $-m+1\le k\le d$
    \[
        {\MUBT}
        (\alpha,d,d-k) \ = \
        \widetilde{\MUBT}(\alpha,d,k+m-1);
    \]
\item \textbf{Initial terms:} For $k+m-1\le\frac{d+m-1}{2}$,
    \[
        \widetilde{\MUBT}(\alpha,d,k+m-1)\ = \
        \sum_{\emptyset\neq S \subseteq [m]} (-1)^{m-|S|}
        \binom{|\alpha_{S}|-d+k-1}{k+m-1}.
    \] 
\end{compactenum}
Using the fact that $\MUBT$ and $\widetilde{\MUBT}$ encode the $h$-vector for Minkowski neighborly families, whose $h$-vectors satisfy relations (a)-(d) and which exist by Theorem~\ref{thm:M-neighborly}, we conclude consistency of these relations. 
By linearity (a) and the definition of the initial terms (d), we see that:
\begin{lem}\label{lem:omega}
    For all $k+m\le\frac{d+m-1}{2}$
    \begin{align*}
&\  (k+m)\widetilde{\MUBT} (\alpha,d,k+m) \\
 =&\  (f_0(\mr{T}_{[m]})-d+k)\widetilde{\MUBT}
(\alpha,d,k+m) (\mr{T}_{[m]}^\circ)\\
 +& \ \sum_{i\in[m]} f_0(\mr{T}_{\{i\}})
\widetilde{\MUBT}^{\langle1\rangle} (\alpha,d,k+m-1).
\end{align*}

\vspace{-3.0em} \qed
\end{lem}

The next lemma holds all the missing ingredients necessary for resolving the
Upper Bound Conjecture for Minkowski sums.

\begin{lem}\label{lem:cenm}
    Let $P_{[m]} = (P_1,\dots,P_m)$ be a family of simplicial $d$-polytopes in
    relatively general position in $\R^d$ with Cayley complex $\mr{T}_{[m]} =
    \mr{T}(P_1,\dots,P_m)$. Let $-m+1 \le k\le d$ such that  $k+m-1\le \lfloor
    \frac{d+m-1}{2}\rfloor$ and let $0 \le \delta\le\frac{d+1}{d-1}$
    be any real parameter. Then
\begin{align}
\begin{split}
&\ \widetilde{g}_{k+m-1}(\mr{T}_{[m]}^\circ)-\delta  \sum_{S\prec [m]} \widetilde{g}_{k+m-2}(\mr{T}_S^\circ)\\
\le\ &\  \widetilde{\MUBT}^{\langle 1 \rangle}(f_0(P_{[m]}),d,k+m-1)-\delta  \sum_{S\prec [m]} \widetilde{\MUBT}^{\langle1\rangle}(f_0(P_{S}),d,k+m-2)
\end{split}
\end{align}
If $\delta<\frac{d+1}{d-1}$, then equality holds if and only if it holds for each summand separately.
\end{lem}

\begin{rem}\label{rem:bb}
Alternatively, and as an application of Theorem~\ref{thm:rii}, one can prove a result that can be used just as well to prove the Upper Bound Theorem on Minkowski sums: For a family $P_{[m]}$ in $\R^d$ of polytopes in relative general position, and a monotone increasing family of nonnegative real parameters $(\delta_s),\ 0\le s\le m$, we have
\[\sum_{S \subseteq [m]} (-1)^{m-\# S} \delta_{|S|} g_{k+|S|-1}(\mr{T}_{S}^\circ)\ \le\ \sum_{S \subseteq [m]} (-1)^{m-\# S} \delta_{|S|} \widetilde{\MUBT}^{\langle1\rangle}(f_0(P_{S}),d,k+|S|-1),\]
with equality if and only if it holds for each summand. For the proof, notice that following Theorem~\ref{thm:rii}, we have for $v\in P_i$ an inequality
\[\sum_{i\in S \subseteq [m]} (-1)^{m-\# S} h_{k+|S|-1}(\Lk(v,\mr{T}_{S}^\circ))\ \le \ \sum_{i\in S \subseteq [m]} (-1)^{m-\# S} \Big(h_{k+|S|-1}(\mr{T}_{S}^\circ)-g_{k+|S|-1}(\mr{T}_{S\setminus \{i\}}^\circ) \Big),\]
sum over all $v$ in $\mr{T}_{[m]}$ and argue as in the proof of Lemma~\ref{lem:cenm} below.
\end{rem}

\begin{proof}[\textbf{Proof of Lemma~\ref{lem:cenm}}]
Using Theorem~\ref{thm:minkit}, we obtain
\begin{align*}
 &\ \  \widetilde{g}_{k+m-1}(\mr{T}_{[m]}^\circ)-\delta  \sum_{S\prec [m]}\widetilde{g}_{k+m-2}(\mr{T}_S^\circ)\\
\le &\ \left( \frac{f_0(\mr{T}_{[m]})-d-m}{k+m-1}\right)\widetilde{g}_{k+m-2}(\mr{T}_{[m]}^\circ)+ \sum_{S\prec [m]}\left(\frac{f_0(\mr{T}_{[m]{\setminus} S}) }{k+m-1}-\delta \right)\widetilde{g}_{k+m-2}(\mr{T}_S^\circ)
\end{align*}
As the equality is only nontrivial if $k\ge 1$ and therefore \[m\le 
    \frac{d+m-1}{2}\Longleftrightarrow m\le d-1,\]
    we may assume that $k+m-1\le d-1$. Hence, as $f_0(P_i)\ge d+1$ and $k\le d$, we have $\frac{f_0(\mr{T}_{[m]{\setminus} S}) }{k+m-1}-\delta\ge 0$, and the latter sum is bounded from above by $\MUBT$:
\begin{align*}
&\  \left( \frac{f_0(\mr{T}_{[m]})-d-m}{k+m-1}\right) \widetilde{\MUBT}^{\langle 1 \rangle}(f_0(P_{[m]}),d,k+m-2)\\ 
+ &\ \sum_{S\prec [m]}\left(\frac{f_0(\mr{T}_{[m]{\setminus} S}) }{k+m-1}-\delta \right) \widetilde{\MUBT}^{\langle1\rangle}(f_0(P_{S}),d,k+m-2)\\
=  &\ \widetilde{\MUBT}^{\langle 1 \rangle}(f_0(P_{[m]}),d,k+m-1)-\delta  \sum_{S\prec [m]} \widetilde{\MUBT}^{\langle1\rangle}(f_0(P_{S}),d,k+m-2)
\end{align*}
where the last equality follows from Lemma~\ref{lem:omega}. The equality case follows directly from Theorem~\ref{thm:minkit}.
\end{proof}
We summarize the upper bounds in the following theorem. We let $\mr{c}$ as in Lemma~\ref{lem:recursive}.
\enlargethispage{5mm}

\begin{thm}\label{mthm:minkm}
    Let $P_{[m]} = (P_1,\dots,P_m)$ be a family of simplicial $d$-polytopes in
    relatively general position in $\R^d$ with $\n = f_0(P_{[m]})$. For the
    corresponding Cayley complex
    $\mr{T}_{[m]}=\mr{T}(P_{[m]})$ the following holds
\begin{compactenum}[\rm (1)]
\item 
$\MUBT$ \textbf{is an upper bound:} 
\begin{compactenum}[\rm (a)]
\item for any $-m+1\le k\le d$ with $k+m-1\le\frac{d+m-1}{2}$
    \[
        \widetilde{h}_{k+m-1}(\mr{T}_{[m]}^\circ) \ \le \
        \widetilde{\MUBT} (\n,d,k+m-1),
    \]
\item for any $-m+1\le k\le d$ with $k+m-1 \le
    \frac{d+m-1}{2}$, we have
    \[
        h_{k+m-1}(\mr{T}_{[m]}) \ \le \   {\MUBT} (\n,d,k+m-1),
    \]
\end{compactenum}
\item \textbf{Equality cases:}
\begin{compactenum}[\rm (a)]
\item equality holds up to some $k_0+m-1$ in $\mr{(1a)}$ if and only if, for all $S\subseteq [m]$,
all non-faces of $\mr{T}_{[m]}$ of cardinality $ k_0+m-1$
are supported in some $\V(\mr{T}_R)$, $R\subsetneq S$.
\item equality holds up to some $k_0+m-1$ in $\mr{(1b)}$ if and only if, for a nonface of cardinality $k_0+m-1-q$ in $\mr{T}_S$, $\# S \ge \mr{c}(k_0,m,d)-q$ is a nonface
are supported in some $\V(\mr{T}_R)$, $R\subsetneq S$.
\end{compactenum}
\item \textbf{Tightness: } 
    there is a pure collection of $m$ polytopes ${Q_i}$ in $\R^d$ with
    $f_0(Q_{[m]})=f_0(P_{[m]})$ such that for all $-m+1\le k\le d$ with
    $k+m-1\le\frac{d+m-1}{2}$
\begin{align*}
         \widetilde{h}_{k+m-1}(\mr{T}_{[m]}^\circ(Q_i))\ =& \ \ \widetilde{\MUBT}
        (f_0(Q_{[m]}),d,k+m-1), \text{ and} \\
        h_{k+m-1}(\mr{T}_{[m]}(Q_i))\ =& \ \ {\MUBT} (f_0(Q_{[m]}),d,k+m-1).
\end{align*}
\end{compactenum}
\end{thm}

\begin{rem}
Note that to bound $\widetilde{h}_{k+m-1}(\mr{T}_{[m]}^\circ)$ for $k+m-1\ge\frac{d+m-1}{2}$ we use the Dehn-Sommerville inequalities, and can resolve to bounding 
\[{h}_{d-k}(\mr{T}_{[m]}).\]
For instance, if we want to bound the number of vertices in a Minkowski sum, then we are really asked to bound
 $\widetilde{h}_{m}(\mr{T}_{[m]}^\circ)$ or equivalently  ${h}_{d-1}(\mr{T}_{[m]})$. The tightness condition of (2b) says that this is attained if and only if any subface of cardinality $d-1-q$ in $\mr{T}_S$, $\# S \ge d-1-q$ is contained in a subsummand. 
\end{rem}

\begin{proof}
Notice first that claims  $\mr{(1a)}$ and  $\mr{(2a)}$ are verbatim special
cases of Theorem~\ref{cor:minkh}. Therefore, the proof of the stated claims
splits into two parts: We first prove $\mr{(1b)}$ and  $\mr{(2b)}$, and then
we address the question of tightness.

 \noindent \textbf{Claims $\mr{(1b)}$ and  $\mr{(2b)}$:}
By Lemma~\ref{lem:recursive}, we have
\begin{align*} 
 &\ {h}_{k+m-1}(\mr{T}_{[m]})\\
 =\ &\sum_{j=0}^{\lfloor\frac{m}{2}\rfloor} \sum_{s=\mr{c}-2j+1}^{m-2j}
 \sum_{\substack{S\subseteq [m]\\ |S|=s}} 
\binom{m-s}{2j}\Bigg(\widetilde{h}_{k+m-1-2j}(\mr{T}_S^\circ)-\frac{1}{2j+1}\sum_{\substack {R\prec S}}\widetilde{h}_{k+m-2-2j}(\mr{T}_R^\circ)\Bigg) \tag{A}\\ 
 +\ &\sum_{j=0}^{\lfloor\frac{m}{2}\rfloor} \sum_{\substack{S\subseteq [m]\\ |S|=\mr{c}-2j}} \tag{B}
\binom{m-|S|-1}{2j}\Bigg(\widetilde{h}_{k+m-1-2j}(\mr{T}_S^\circ)-\frac{m-|S|}{(m-|S|+1)(2j+1)}
\sum_{\substack {R\prec S}}\widetilde{h}_{k+m-2-2j}(\mr{T}_R^\circ)\Bigg).
\end{align*}

Now, clearly, $\frac{1}{2j+1}\le 1< \frac{d+1}{d-1}$, so that we
can estimate the sums~\eqref{eqn:A} and~\eqref{eqn:B} using Lemma~\ref{lem:cenm}, obtaining 
\begin{align*} 
 &\ \widetilde{h}_{k+m-1}(\mr{T}_{[m]})\\
 \le\ &\sum_{j=0}^{\lfloor\frac{m}{2}\rfloor} \sum_{s=\mr{c}-2j+1}^{m-2j}
 \sum_{\substack{S\subseteq [m]\\ |S|=s}} 
\binom{m-s}{2j}\Bigg(\widetilde{\MUBT}(f_0(P_{S}),d,k'_j)-\frac{1}{2j+1} \sum_{\substack {R\prec S}}\widetilde{\MUBT}(f_0(P_{[m]})_R,d,k'_j-1)\Bigg)\\ 
 +\ &\sum_{j=0}^{\lfloor\frac{m}{2}\rfloor} \sum_{\substack{S\subseteq [m]\\ |S|=\mr{c}-2j}}
\binom{m-|S|-1}{2j}\Bigg(\widetilde{\MUBT}(f_0(P_{S}),d,k'_j)-\frac{m-|S|}{(m-|S|+1)(2j+1)} \sum_{\substack {R\prec S}}\widetilde{\MUBT}(f_0(P_{[m]})_R,d,k'_j-1)\Bigg)
\end{align*}
where we abbreviate $k'_j:=k+m-1-2j$.
Since $\widetilde{\MUBT}$ satisfies linearity and the Dehn--Sommerville symmetries as well, we can reverse the logic of Lemma~\ref{lem:recursive}, the latter sums equals ${\MUBT}(f_0(P_{[m]}),d,k+m-1)$ of $\mr{(1b)}$. Equality only holds if it holds in the application of Lemma~\ref{lem:cenm}, therefore also concluding the proof of claim $\mr{(2b)}$.

\noindent \textbf{Claim $\mr{(3)}$:} By 
Theorem~\ref{mthm:minkm}$\mr{(2a)}$ and  $\mr{(2b)}$, it suffices to show that
there is a Minkowski neighborly family $Q_{[m]}$ of simplicial $d$-polytopes in
$\R^d$ with $f_0(Q_{[m]})=\n$. Such a family is provided by
Theorem~\ref{thm:M-neighborly}(ii).
\end{proof}

\section{Minkowski sums of nonpure collections} \label{sec:nonpure}

\newcommand{\dm}[1]{\xi\sm (P_{#1})}
\newcommand{\dmv}[1]{\xi({#1})}

In Section~\ref{sec:UBTMS}, a basic assumption on the collection
$P_{[m]}=(P_1,\dots,P_m)$ of polytopes in $\R^d$ was that $f_0(P_i) \ge d+1$
for all $i = 1,\dots,m$. Minkowski sums of \emph{nonpure} collections, i.e.,
collections $P_{[m]}$ such that $f_0(P_i) < d+1$ for some $i$, are however of
importance. The simplest case is when all summands have exactly two vertices.
In this case the resulting Minkowski sum is a zonotope and the corresponding
Upper Bound Theorem is well-known~\cite{buck} (and in essence goes back
to Steiner \cite{steiner}). In this section we will give an extension of the
Upper Bound Theorem for Minkowski sums to nonpure collections. This is a
nontrivial step and the reader will observe the increase in complexity of the
arguments and especially notation. For this reason we devote a separate section for the 
nonpure situation.  Nevertheless, the basic line of reasoning remains the same 
and we only sketch the main amendments.

Let us notice that if $|P_{[m]}|$ maximize the number of $k$-faces, then the
polytopes $P_i$ are simplicial and in relatively general position.
In particular, if $P_i$ has fewer than $d+1$ vertices then genericity implies
that $P_i$ is a $(f_0(P_i)-1)$-simplex.  For the nonpure UBPM, we need to
introduce an additional parameter: For a family $P_{[m]}$ of $m$ polytopes in
$\R^d$ with $\n = f_0(P_{[m]})$, let us abbreviate the dimension of the
Minkowski sum of the subfamily $P_S$ by
$\dm{S}$. Note that this quantity is determined purely in terms of the vector $\n$ of vertex numbers:
\[
    \dm{S} \ = \ \dmv{\n_S} \ := \ \min\bigl(d, |\n_S|-|S|\bigr).
\] 
We start with an analogue of Theorem~\ref{thm:M-neighborly} that applies to
nonpure collections.
\begin{thm}\label{thm:M-neighborly-np}
Let $m,d \ge 1$ be fixed.
\begin{compactenum}[\rm (i)] 
    \item There is no Minkowski $(k,\ell)$-neighborly family
        $P_{[m]}$ in $\R^d$ for $k +\ell -1> \lfloor\frac{\dm{[m]}+\ell-1}{2}\rfloor
        $.
        \item For all $\n \in \Z_{\ge 0}^m$ there is a family $P_{[m]}$ in $\R^d$ with
            $f_0(N_{[m]})=\n$  that is Minkowski $(k,\ell)$-neighborly for all
            $\ell\le m$ and $\ell-1 \le k+\ell-1 \le
            \lfloor\frac{\dm{[m]}+\ell-1}{2}\rfloor$.
\end{compactenum}
\end{thm}

The first claim follows from an analogues statement to
Proposition~\ref{prp:radon_cayley}.  The second statement follows again from
the work of Matschke--Pfeifle--Pilaud \cite{MPP}. We continue to call the collections of
Theorem~\ref{thm:M-neighborly-np}(ii) \Defn{Minkowski neighborly}.
Similar to the case of the UBT for spheres, we can use Minkowski neighborly
polytopes to abbreviate the UBTM, because their $f$-vectors depend on $m$, $d$
and $f_0(P_{[m]})$ only.

\begin{prp}\label{prop:M-neigh-fvector-np}
    If $P_{[m]},P_{[m]}^\prime$ are two Minkowski neighborly families of $m$
    simplicial polytopes with $f_0(P_{[m]}) = f_0(P^\prime_{[m]})$, then
    $f_k(|P_{[m]}|) \ = \ f_k(|P^\prime_{[m]}|)$ for all $0 \le k \le d$.
\end{prp}

At this point, let us remark two curious properties that make our life
simpler.
\begin{obs}\label{rem:simp}
Let $P_{[m]}$ be nonpure collection of relative general position 
polytopes in $\R^d$.
\begin{compactenum}[\rm (1)]
    \item If the Minkowski sum of
        polytopes in relatively general position is not full-dimensional, 
        then $|P_{[m]}| \cong P_1 \times P_2 \times \cdots \times P_m$. In this
        case we say that $P_{[m]}$ is \Defn{deficient}.
    \item If $\dim P_i=0$ for some $i\in [m]$, 
    then $|P_{[m]}|$ is a translate of $|P_{[m]{\setminus} i}|$.
\end{compactenum}
\end{obs}

We recover Buck's Theorem on zonotopes~\cite{buck}.

\begin{cor}\label{prop:M-neigh-fvector-zt}
    Any family $\Pm$ of $m$ segments in relatively general position in $\R^d$ is Minkowski neighborly.
    In particular, the $f$-vector of the zonotope $|P_{[m]}|$ only depend on $m$ and $d$.
\end{cor}

\begin{proof}
Use Proposition~\ref{prop:M-neigh-fvector-np} and the fact that all families of at most $d$ edges in $\R^d$ is deficient.
\end{proof}

For the UBT for Minkowski sums, we define for $m,d \ge 1$ and $\n \in \Z^m_{\ge 0}$
\[
    \mr{nb}_i(d,m,\n) \ := \ f_i(|P_{[m]}|)
\]
where $P_{[m]}$ is any Minkowski neighborly family of $m$ simplicial
polytopes in $\R^d$ with $\n=f_0(P_{[m]})$. 

\begin{thm}[Upper Bound Theorem for Minkowski sums of general families]\label{thm:UBTMS-np}
    Let $m,d \ge 1$ and $\n \in \Z^m_{\ge 0}$. If $P_{[m]} = (P_1,\dots,P_m)$ is a
    family of $m$ polytopes in $\R^d$ with $P_{[m]}=\n$, then
    \[
        f_k(|P_{[m]}|) \ \le \ \mr{nb}_k(d,m,\n).
    \]
    for all $k = 0, \dots, d-1$. Moreover, the family $P_{[m]}$ is Minkowski neighborly
    if and only if equality holds for the number of facets.
\end{thm}

The remainder of this section will provide the proof of Theorem~\ref{thm:UBTMS-np}. We only sketch the line of reasoning for the main points.

\subsection{Initial terms of the $h$-vector.}
We start with a replacement for Theorem~\ref{thm:minkit} that applies to
nonpure collections.

\begin{thm}\label{thm:minkitnonpure}
    Let $P_{[m]} = (P_1,\dots,P_m)$ be a family of simplicial polytopes in
    relatively general position in $\R^d$. Let $\mr{T}_{[m]}$ be the
    corresponding Cayley complex and $e = \dm{[m]}$.  Then, for every
    $k$, 
\begin{compactenum}[\rm (a)]
\item and for every $i\in [m]$ and $v$ vertex of $P_i$
    \[
        h_k(\Lk(v,\mr{T}_{[m]}^\circ)) \ \le \  h_k(\mr{T}_{[m]}^\circ) +
        h^{\langle e+m-1 \rangle}_k (\mr{T}_{[m]{\setminus} \{i\}}^\circ),\ \text{and}
    \]
\item 
    \[
(k+m) g_{k+m} (\mr{T}_{[m]}^\circ)\le (f_0(\mr{T}_{[m]})-e-m) 
h_{k+m-1} (\mr{T}_{[m]}^\circ)+\sum_{i\in[m]} f_0(\mr{T}_{\{i\}}) 
h^{\langle e+m-1\rangle}_{k+m-1}(\mr{T}_{[m]{\setminus} \{i\}}^\circ).
\]
\item We have  
\[(i+|S|)  h_{i+|S|} (\mr{T}_{S}^\circ) \ = \ (f_0(\mr{T}_{S})-d+i)  h_{i+|S|-1} (\mr{T}_{S}^\circ)+\sum_{i\in S} f_0(\mr{T}_{\{i\}}) h^{\langle \dm{S}\rangle}_{i+|S|-1}(\mr{T}_{S{\setminus} \{i\}}^\circ)\] for all $i\le k_0$ and $S\subseteq[m]$ if and only if all non-faces of $\mr{T}_{S}$ of dimension $< k_0+|S|$ are supported in some $\V(\mr{T}_R)$, $R\subsetneq S$. 
\end{compactenum}
\end{thm}

\begin{proof}
It suffices to prove (a) and characterize equality in this case; the other
inequalities are obtained by simply summing over all vertices. To prove the
inequality, notice that by Lemma~\ref{lem:lkst}, $h_{k+m-1}(\Lk(v,\mr{T}_{[m]}^\circ)) =
h_{k+m-1}(\St(v,\mr{T}_{[m]}^\circ))$ and 
\[ 
    h_{k+m-1}(\mr{T}_{[m]}^\circ) + h^{\langle e+m-1
    \rangle}_{k+m-1}(\mr{T}_{[m]{\setminus} \{i\}}^\circ)\ = \
    h_{k+m-1}(\mr{T}_{[m]}^\circ(i))
\]
where we define
\[
    \mr{T}_{[m]}^\circ(i) \ := \ \Big(\mr{T}_{[m]},\bigcup_{i\in S\subsetneq [m]}
    \mr{T}_{S}\Big).
\]
Hence it suffices to prove that
$
    h_{k+m-1}(\St(v,\mr{T}_{[m]}^\circ))  \le 
    h_{k+m-1}(\mr{T}_{[m]}^\circ(i)).
$
For this, let $\mr{C}=\Cay(P_1,\dots,P_m)$, together with the faces 
\[
\mr{C}_S \ := \ \conv\Bigl( \bigcup_{i\in S} P_i+e_i\Bigr) \ \text{ for}\
S\subsetneq [m].
\]
The complex $C(i)=\bigcup_{i\in S\subsetneq [m]}
    \mr{C}_{S}$ is a PL ball and
$\St(v,\partial\mr{C})$ is a PL ball of the same dimension contained in it, so
that $(\partial\mr{C},\St(v,\partial\mr{C}))$ is Cohen--Macaulay by
Theorem~\ref{thm:SR}. Hence
\[
 (C(i),\St(v,\partial\mr{C})) \ \cong\ \Big(C(i),\St(v,\mr{T}_{[m]})\cup
 \bigcup_{i\in S\subsetneq [m]} 
\mr{T}_{S} \Big)
\  \cong\ \big(\mr{T}_{[m]}^\circ(i),\St(v,\mr{T}_{[m]}^\circ)\big)
\]
where the last complex is the complement of $\St(v,\mr{T}_{[m]}^\circ)$ in 
$\mr{T}_{[m]}^\circ(i)$. Hence
\[h_{k+m-1}(\St(v,\mr{T}_{[m]}^\circ))+h_{k+m-1}(\mr{T}_{[m]}^\circ(i),\St(v,\mr{T}_{[m]}
^\circ)=h_{k+m-1}(\mr{T}_{[m]}^\circ(i))\]
by linearity of the $h$-vector and 
$h_{k+m-1}(\mr{T}_{[m]}^\circ(i),\St(v,\mr{T}_{[m]}^\circ)\ge 0$ by 
Cohen--Macaulayness. Equality, i.e.\ $h_{k+m-1}(\mr{T}_{[m]}^\circ(i),\St(v,\mr{T}_{[m]}^\circ)=0$, holds up to some $k_0+m-1$ if $\FM[\mr{T}_{[m]}^\circ(i),\St(v,\mr{T}_{[m]}^\circ)]$ is generated in
degree $> k_0+m-1$. We conclude by iteratively applying the same argument to all subsets $S\subseteq [m]$.
\end{proof}

We conclude as in Section~\ref{ssec:small_h}:

\begin{cor}\label{cor:minkh_P}
    Let $P_{[m]} = (P_1,\dots,P_m)$ be a family of $m$ simplicial
    in relatively general position and let $\mr{T}_{[m]}^\circ$ be
    the corresponding relative Cayley complex, and let $e = \dm{[m]}$.  Then for all $-m+1\le k\le d$
    \[
        h_{k+m-1} (\mr{T}_{[m]}^\circ) \ \le \  \sum_{S\subseteq [m]}
        (-1)^{m-|S|} \binom{f_0(\mr{T}_S)-e+k-1}{k+m-1}.
    \]
    Equality holds for some $k_0+m-1$ if and only if all non-faces of
    $\mr{T}_{[m]}$ of dimension $< k_0+m-1$ are supported in some $\V(\mr{T}_S)$, $S\subsetneq [m]$. 
\end{cor}

\subsection{The Dehn--Sommerville formula and other linear relations.} 
The most challenging part of the Upper Bound Problem for Minkowski sums are
the Dehn--Sommerville relations. Recall that we can assume that $\dim
|P_{[m]}| = d$. The following is a simple corollary of Proposition~\ref{prp:Hilbert_reciproc}.
\begin{lem}\label{lem:DSM-np} 
    Let $\mr{T} = \mr{T}(P_{[m]})$ be the Cayley complex for $P_{[m]}$. 
     Then, for all $-m+1\le k\le d$ 
    \[
        h_{d-k}(\mr{T}_{[m]})\ = \ \widetilde{h}_{k+m-1}
        (\mr{T}_{[m]}^\circ) + \sum_{\substack{S \subseteq [m]\\ \dm{S}<d}}  (-1)^{k}\binom{d+m-1}{m+k+\dm{S}}.\] 
\end{lem} 
The important point to note here is that the correction term
$
    \sum_{\substack{S \subseteq [m]\\ \dm{S}<d}}  (-1)^{k}\binom{d+m-1}{m+k+\dm{S}}
$
only depends on $k, d, m$ and $f_0(P_{[m]})$ but \emph{not} on the
combinatorial type of the Cayley polytope. Now, we note that by linearity of
the $g$-vector:

\begin{prp}\label{prp:recursive-np} 
    For $\mr{T}_{[m]}$ as above, any $\ell\ge 0$, and any $1\le s\le m$, we have 
    \[ 
        h_{k+m-1}(\mr{T}_{[m]}) \ =\ \sum_{S\subseteq [m]}
        \widetilde{g}_{k+m-1}^{\langle m-|S|+\dm{[m]}-\dm{S}\rangle} (\mr{T}_{S}^\circ). 
    \] 
\end{prp}

Recall that by Observation~\ref{rem:simp}(i), the combinatorial type of 
$\mr{T}_{S}^\circ$ is determined by $f_0(P_S)$ (and $d$) if $\dm{S}<d$. 
Note furthermore that if $\dm{S}<d$, then $\dm{R}<d$ for all $R\subseteq S$. 
We obtain, with arguments analogous to Lemma~\ref{lem:recursive}.

\begin{lem}\label{lem:recursive-np} 
    For $\mr{T}_{[m]}$, $\dm{[m]}=d$ as above, any $1\le s\le m$, any $k+m-1\le \frac{d+m-1}{2}$, and with $\mr{c}=\mr{c}(k,m,d)$ we have
    \begin{align*} 
      &\ \ {h}_{k+m-1}(\mr{T}_{[m]})\\
      =& \ \sum_{j=0}^{\lfloor\nicefrac{m}{2}\rfloor}
     \sum_{s=\mr{c}-2j+1}^{m-2j} \sum_{\substack{S\subseteq [m]\\ |S|=s \\ \dm{S}=d}}
     \binom{m-s}{2j}\Bigg(\widetilde{h}_{k+m-1-2j}(\mr{T}_S^\circ)-\frac{1}{2j+1}
     \sum_{\substack {R\prec S\\ \dm{R}=d}}\widetilde{h}_{k+m-2-2j}(\mr{T}_R^\circ)\Bigg)
    \\ 
      +&\ \sum_{j=0}^{\lfloor\nicefrac{m}{2}\rfloor} \sum_{\substack{S\subseteq [m]\\ |S|=\mr{c}-2j\\ \dm{S}=d}} 
\binom{m-|S|-1}{2j}\Bigg(\widetilde{h}_{k+m-1-2j}(\mr{T}_S^\circ)-\frac{m-|S|}{(m-|S|+1)(2j+1)} 
\sum_{\substack {R\prec S\\ \dm{R}=d}}\widetilde{h}_{k+m-2-2j}(\mr{T}_R^\circ)\Bigg)\\
 +&\ \gamma(f_0(P_{[m]}),m,d,k). 
\end{align*}
where the correction term $\gamma$ only depends on $f_0(P_{[m]})$, $m$, $d$ and $k$, but not on the combinatorial type of
$\Cay(P_1,\cdots,P_m)$.
\end{lem} 
\subsection{The Upper Bound Theorem for nonpure Minkowski sums} 
\newcommand{\NMUBT}{{\nu}}
We finally conclude the Upper Bound Theorem for pure Minkowski sums:
We define 
\[
    \widetilde{\NMUBT}^{\langle\ell\rangle}:\Z^{[m]}\times\Z\times\Z
    \longrightarrow \Z,\ m\ge 0,\ \ell\ge 0
\] 
and
\[{\NMUBT}:\Z^{[m]}\times\Z\times\Z\longrightarrow \Z,\ m\ge 0\] 
by the following conditions:
\begin{compactenum}[\rm (a)]
\item \textbf{Basic relation:} For all $k,\,\ell\ge 0$ 
    \[
        \widetilde{\NMUBT}^{\langle\ell\rangle}
        (\cdot,\cdot,k) \ = \ \widetilde{\NMUBT}^{\langle \ell-1 \rangle}
        (\cdot,\cdot,k)-\widetilde{\NMUBT}^{\langle \ell-1
        \rangle}(\cdot,\cdot,k-1);
    \]
\item \textbf{Linearity:} For all $-m+1\le k\le d$
    \[
        {\NMUBT}(\alpha,d,k+m-1)\ = \ \sum_{S\subseteq [m]}
        \widetilde{\NMUBT}^{\langle m-|S|+\dmv{\alpha}-\dmv{\alpha_{S}} \rangle} (\alpha_{S},d,k+m-1).
    \]
\item \textbf{Dehn--Sommerville relation:} For all $-m+1\le k\le d$
    \[
        {\NMUBT}
        (\alpha,d,d-k) \ = \
        \widetilde{\NMUBT}(\alpha,d,k+m-1)+\sum_{\substack{S \subseteq [m]\\ \dmv{\alpha_S}<d}}  (-1)^{k}\binom{d+m-1}{m+k+\dmv{\alpha_S}};
    \]
\item \textbf{Initial terms:} For $k+m-1\le\frac{\dmv{\alpha_{[m]}}+m-1}{2}$, we have
    \[
        \widetilde{\NMUBT}(\alpha,d,k+m-1)\ = \
        \sum_{\emptyset\neq S\subseteq [m]} (-1)^{m-|S|}
        \binom{|\alpha_{S}|-\dmv{\alpha_S}+k-1}{k+m-1}.
    \] 
\end{compactenum}

With this we obtain the desired UBT for Minkowski sums of nonpure collections.
\begin{thm}\label{mthm:minkm-np}
    Let $P_{[m]} = (P_1,\dots,P_m)$ be a family of simplicial polytopes in
    relatively general position in $\R^d$ with $\n = f_0(P_{[m]})$ and $\dim
    |P_{[m]}|=d$. For the
    corresponding Cayley complex
    $\mr{T}_{[m]}=\mr{T}(P_{[m]})$ the following holds
\begin{compactenum}[\rm (1)]
\item 
$\NMUBT$ \textbf{is an upper bound:} 
\begin{compactenum}[\rm (a)]
\item for any $-m+1\le k\le d$ with $k+m-1\le\frac{d+m-1}{2}$
    \[
        \widetilde{h}_{k+m-1}(\mr{T}_{[m]}^\circ) \ \le \
        \widetilde{\NMUBT} (\n,d,k+m-1),
    \]
\item for any $-m+1\le k\le d$ with $k+m-1 \le
    \frac{d+m-1}{2}$
    \[
        h_{k+m-1}(\mr{T}_{[m]}) \ \le \   {\NMUBT} (\n,d,k+m-1),
    \]
\end{compactenum}
\item \textbf{Equality cases:}
\begin{compactenum}[\rm (a)]
\item equality holds up to some $k_0+m-1$ in $\mr{(1a)}$ if and only if, for all $S\subseteq [m]$, all non-faces of $\mr{T}_S$ of cardinality $\le k_0+|S|-1$
are supported in some $\V(\mr{T}_R)$, $R\subsetneq S$.
\item equality holds up to some $k_0+m-1$ in $\mr{(1b)}$ if and only if, for a nonface of cardinality $k_0+m-1-q$ in $\mr{T}_S$, $\# S \ge \mr{c}(k_0,m,d)-q$ is a nonface
are supported in some $\V(\mr{T}_R)$, $R\subsetneq S$.
\end{compactenum}
\item \textbf{Tightness: }\\ there is a collection of $m$ polytopes ${Q_i}$ in $\R^d$ with $f_0(Q_{[m]})=f_0(P_{[m]})$ for which, 
 for any $-m+1\le k\le d$ with $k+m-1\le\frac{d+m-1}{2}$, we have 
\begin{align*}
\widetilde{h}_{k+m-1}(\mr{T}_{[m]}^\circ(Q_i))\ =& \ \  \widetilde{\NMUBT} (f_0(Q_{[m]}),d,k+m-1),
\text{ and} \\
h_{k+m-1}(\mr{T}_{[m]}(Q_i))\ =& \ \  {\NMUBT} (f_0(Q_{[m]}),d,k+m-1).
 \end{align*}
\end{compactenum}
\end{thm}

\begin{proof}
The proof is analogous to Theorem~\ref{mthm:minkm}: The crucial cases to verify are (1b) and (2b). For this, one can
disregard deficient subfamilies of $P_{[m]}$ (those with $\dm{S}<d$), as their contribution is purely combinatorial. For the remaining subfamilies, one can use
Lemma~\ref{lem:recursive-np} as in
Theorem~\ref{mthm:minkm}. 
\end{proof}
\enlargethispage{7mm}
\section{Mixed faces of Minkowski Sums}\label{sec:t_mxf}
\newcommand{\htm}{h^{\mr{mix}}}
\newcommand{\ftm}{f^{\mr{mix}}}
\newcommand\mix{\mathrm{mix}}

Let $\Pm = (P_1,\dots,P_m)$ be a pure collection of $m$ polytopes in $\R^d$ in
relatively general position. Every proper face $F \subsetneq |\Pm|$ has a unique
decomposition $F = F_1 + \cdots + F_m$ where $F_i \subseteq P_i$ is a face. A
face $F$ is called \Defn{mixed} if $\dim F_i > 0$ for all $i=1,\dots,m$. In
this section, we will study the \Defn{mixed $f$-vector} $f^\mix(\Pm)$ giving
the number of mixed faces of $|\Pm|$.   Mixed faces and in particular mixed
facets are related to the better known mixed cells in mixed subdivisions via
liftings; see~\cite{dLRS}. In this section, we prove an upper bound
theorem for the number of mixed faces. 

Notice, that by definition $f^\mix_{-1}(\Pm) = f^\mix_0(\Pm) = 0$. Moreover,
the `relatively general position' assumption forces $f^\mix_k(\Pm) = 0$ for
all $k<m$, which also limits the number of summands to $m<d$. One can drop the assumption
on general position but this is less natural.  Let us start with a simple
observation. A face $F \subsetneq |\Pm|$ is mixed if and only if it is not a
face of a subsum in the following sense: For a linear function $\ell$ let us
denote by $P_S^\ell$ the face of $P_S$ maximizing $\ell$. Then $F$ is mixed if for all
$\ell$ such that $\Pm^\ell = F$, $P_S^\ell \neq F$ for all $S \subsetneq [m]$.
We may now evaluate the mixed faces as the difference of all faces of the relative Cayley complex, minus the non-mixed faces. Following this basic equation gives an inclusion-exclusion, and we obtain that 
\[
    f^\mix_k(\Pm) \ \le \ \sum_{S \subseteq [m]} (-1)^{m-|S|}f_k(|P_S|),
\]
with equality for the number of mixed facets.
Let us define the \Defn{mixed $h$-vector} of $\Pm$ by
\begin{equation}\label{eq:hmix}
h^\mix_{i+m-1}(\Pm) \ := \ \sum_{S \subseteq [m]} (-1)^{m - |S|}
g^{\langle m-|S|\rangle}_{i+|S|-1}(\mr{T}_S^\circ) 
\end{equation}
for all $-m+1 \le i \le d$. Using~\eqref{eqn:f-h-polynomial} and
Proposition~\ref{prp:CP2} proves the following.
\begin{lem}\label{lem:mixed_h}
    Let $\Pm$ be a collection of polytopes in relatively general position.
    Then
    \[
    f_k^\mix(\Pm) \ \le \ \sum_{i=-m+1}^d \binom{d-i}{k-i} h^\mix_{i+m-1}(\Pm)
    \]
    for all $k \ge 0$.
\end{lem}

Thus, in analogy to the UBT for Minkowski sums, it suffices to prove upper
bounds on the mixed $h$-vector of~$\Pm$.

\begin{thm}\label{thm:mixed_faces}
    Let $\Pm$ be a pure collection of $m$ simplicial polytopes in relatively
    general position in $\R^d$ with $\n = f_0(\Pm)$.  Let $\mr{T}_{[m]} =
    \mr{T}(\Pm)$ be the corresponding Cayley complex. Then for $-m+1 \le
    k\le{d-m+1}{}$
\begin{align*}
     &\ \ \htm_{k+m-1}(P_{[m]}) \\
    =& \ \sum_{S\subseteq [m]} (-1)^{m-|S|}\, g^{\langle m-|S|\rangle}_{i+|S|-1}(\mr{T}_S^\circ)\\
    \le& \ \sum_{S\subseteq [m]} (-1)^{m-|S|}\, \widetilde{\MUBT}^{\langle m-|S|\rangle}(\n_{S},d,k+|S|-1),
        \end{align*}
    with equality for some $k_0 + m-1$ if and only if it holds for all summands. 
\end{thm}

For mixed facets, this results in the following tight upper bound.

\begin{thm}\label{thm:mixed_facets}
    Let $0 < m < d$ and $\Pm$ a collection of $m$ simplicial $d$-polytopes in
    relatively general position in $\R^d$. Then for any Minkowski neighborly
    family $\mr{Nb}_{[m]}$ of $d$-polytopes with $f_0(\Pm) = f_0(\mr{Nb}_{[m]})$ we have
    \[
        f^\mix_{d-1}(\Pm) \ \le\ f^\mix_{d-1}(\mr{Nb}_{[m]})
    \]
    with equality if and only if $\Pm$ is Minkowski neighborly.
\end{thm}

This is an immediate consequence of Theorem~\ref{thm:mixed_faces}.

\begin{proof}[\textbf{Proof of Theorem~\ref{thm:mixed_faces}}]
We first show the bounds on the initial terms of the mixed $h$-vector, i.e.\ bounds on $\htm_{k+m-1}(P_{[m]})$ for $k+m-1\le\frac{d+m-1}{2}$.

We have
\[\htm_{k+m-1}(P_{[m]}) 
 \ =\  \sum_{S\subseteq [m]} (-1)^{m-|S|} g^{\langle m-|S|\rangle}_{k+|S| -1}(\mr{T}_{S}^\circ).
\]
The second sum may be written as
\begin{equation}\label{eq:rec}
\sum_{S\subseteq [m]} \sum_{\substack{j\in [0,m-|S|]\\ j+m-|S| \text{ even}}} \binom{m-|S|}{j}\bigg( h_{k+|S| -1-j}(\mr{T}_{S}^\circ) - \frac{1}{m-|S|-j+1}\sum_{R\prec S}  h_{k+|R| -1-j}(\mr{T}_{R}^\circ) \bigg) \end{equation}
where we recall that $\prec$ denotes the covering relation.
Using Lemma~\ref{lem:cenm} we can therefore estimate 
\[\sum_{S\subseteq [m]} (-1)^{m-|S|} g^{\langle m-|S|\rangle}_{k+|S| -1}(\mr{T}_{S}^\circ)
\ \le\  \sum_{S\subseteq [m]} (-1)^{m-|S|} \widetilde{\MUBT}^{\langle m-|S|\rangle}
(f_0(P_{S}),d,k+|S|-1).\]
The second bound can be derived in a similar manner as the first: Combining the Dehn--Sommerville relations and Lemma~\ref{lem:recursive}, we rewrite 
$\htm_{k+m-1}(P_{[m]}) =  \sum_{S\subseteq [m]} (-1)^{m-|S|}\, g^{\langle m-|S|\rangle}_{i+|S|-1}(\mr{T}_S^\circ)$ as a sum of $h_{j}(\mr{T}_S^\circ)$, $j\le \frac{d+|S|-1}{2}$. We can now pair $h$- and $g$-numbers of $\mr{T}_S^\circ$ and $\mr{T}_R^\circ,\ R\prec S$ and use Lemma~\ref{lem:cenm} to bound each term by the corresponding term of $\wt{\MUBT}$ and $\wt{\MUBT}^{\langle 1 \rangle}$. In details: 

Recall that as in Lemma~\ref{lem:recursive}, we may think of the coefficients of $h_{i+m+S}(\mr{T}_S^\circ)$ Equation~\eqref{eq:hmix} as elements in an array with sides recording $i$ and $|S|$. 

The interplay with the Dehn--Sommerville relations for Cayley complexes now becomes relevant if, in a summand $h_{k+|S| -1-j}(\mr{T}_{S}^\circ) - \frac{1}{m-|S|-j+1}\sum_{R\prec S}  h_{k+|R| -1-j}(\mr{T}_{R}^\circ)$ in the Sum~\eqref{eq:rec}, $k+|S| -1-j > \frac{d+|S|-1}{2}$. To understand this, we rewrite Sum~\eqref{eq:rec} as
\begin{equation*}\label{eq:rec2}
\sum_{j=0}^m \sum_{\substack{S\subseteq[m]\\ |S|=m-j}} \sum_{R\subseteq S} (-1)^{|S\setminus R|} \frac{1}{|S\setminus R|!} h_{k+|R|-1-j}(\mr{T}_{R}^\circ)
\end{equation*}
For any $S\subsetneq m$, we can now apply the Dehn--Sommerville relations and Lemma~\ref{lem:recursive} to rewrite, for $S\subsetneq [m]$ and $j=m-|S|$,
\begin{align*}
&\ \ \sum_{R\subseteq S} (-1)^{|S\setminus R|} \frac{1}{|S\setminus R|!} h_{k+|R|-1-j}(\mr{T}_{R}^\circ)\\
=&\  \sum_{\substack{R\subseteq S \\ |S\setminus R|\ \text{even}\\ |R|\ge d-2k+2m-2|S|+1}} \frac{1}{|S\setminus R|!} \sum_{T\subseteq R} (-1)^{|T|-|R|} h_{d-k+m-|S|
+|T|-|R|}(\mr{T}_{T}^\circ)
\end{align*}
and secondly rewrite, for $R\subseteq S$, and $|S\setminus R|$ {even}, the summands of the  identity as
\begin{align*}
&\ \  \sum_{T\subseteq R} (-1)^{|T|-|R|} h_{d-k+m-|S|
+|T|-|R|}(\mr{T}_{T}^\circ)\\
=&\  \sum_{\substack{T\subseteq R\\ |R\setminus T|\ \text{even}}} \Bigg(h_{d-k+m-|S|
+|T|-|R|}(\mr{T}_{T}^\circ) -\frac{1}{|S|-|T|+1}\sum_{U\prec T} h_{d-k+m-|S|
+|U|-|R|}(\mr{T}_{U}^\circ)\Bigg).
\end{align*} 
The claim now follows by application of Lemma~\ref{lem:cenm}.
\end{proof} 

%
%

\newcommand{\etalchar}[1]{$^{#1}$}
\def\cprime{$'$}
\providecommand{\bysame}{\leavevmode\hbox to3em{\hrulefill}\thinspace}
\providecommand{\MR}{\relax\ifhmode\unskip\space\fi MR }
\providecommand{\MRhref}[2]{%
  \href{http://www.ams.org/mathscinet-getitem?mr=#1}{#2}
}
\providecommand{\href}[2]{#2}


\begin{thebibliography}{MRTT53}

\bibitem[Adi15]{AdipII}
K.~A. Adiprasito, \emph{{Toric chordality}}, preprint,
  \href{http://arxiv.org/abs/1503.06640}{arXiv:1503.06640}.

\bibitem[AB12]{AB-SSZ}
K.~A. Adiprasito and B.~Benedetti, \emph{Subdivisions, shellability, and
  collapsibility of products}, February 2012, to appear in Combinatorica,
  available at \href{http://arxiv.org/abs/1202.6606}{arXiv:1202.6606}.

\bibitem[ABG83]{ABG}
K.~A. Adiprasito, A.~Bj\"orner, and A.~Goodarzi, \emph{{Face numbers of
  Sequentially Cohen-Macaulay complexes and Betti numbers of componentwise
  linear ideals}}, preprint,
  \href{http://arxiv.org/abs/1502.01183}{arXiv:1502.01183}.

\bibitem[ABPS15]{ABPS}
K.~A. Adiprasito, P.~Brinkmann, A.~Padrol, and R.~Sanyal, \emph{Mixed faces and
  colorful depth}, in preparation, 2015.

\bibitem[BKL86]{BKL86}
D.~Barnette, P.~Kleinschmidt, and C.~W. Lee, \emph{An upper bound theorem for
  polytope pairs}, Math.\ Oper.\ Res. \textbf{11} (1986), no.~3, 451--464.

\bibitem[BL80]{BL}
L.~J. {Billera} and C.~W. {Lee}, \emph{{Sufficiency of McMullen's conditions
  for f-vectors of simplicial polytopes.}}, {Bull. Am. Math. Soc., New Ser.}
  \textbf{2} (1980), 181--185.

\bibitem[BL81]{bl81}
\bysame, \emph{The numbers of faces of polytope pairs and unbounded polyhedra},
  European J. of Combinatorics \textbf{2} (1981), 307--322.

\bibitem[Bj{\"o}80]{Bj1}
A.~Bj{\"o}rner, \emph{Shellable and {C}ohen-{M}acaulay partially ordered sets},
  Trans. Amer. Math. Soc. \textbf{260} (1980), 159--183.

\bibitem[Bj{\"o}03]{NFH}
\bysame, \emph{Nerves, fibers and homotopy groups}, Journal of Combinatorial
  Theory, Series A \textbf{102} (2003), no.~1, 88 -- 93.

\bibitem[Bj{\"o}07]{Bj07}
\bysame, \emph{A comparison theorem for $f$-vectors of simplicial polytopes},
  Pure Appl. Math. Q. \textbf{3} (2007), no.~1, 347--356.

\bibitem[BWW05]{BWW}
A.~Bj{\"o}rner, M.~L. Wachs, and V.~Welker, \emph{Poset fiber theorems}, Trans.
  Amer. Math. Soc. \textbf{357} (2005), 1877--1899.

\bibitem[{Bor}48]{Borsuk}
K.~{Borsuk}, \emph{{On the imbedding of systems of compacta in simplicial
  complexes}}, {Fundam. Math.} \textbf{35} (1948), 217--234.

\bibitem[BM71]{BM}
H.~Bruggesser and P.~Mani, \emph{Shellable decompositions of cells and
  spheres}, Math. Scand. \textbf{29} (1971), 197--205 (1972).

\bibitem[BH93]{Bruns-Herzog}
W.~Bruns and J.~Herzog, \emph{Cohen-{M}acaulay rings}, Cambridge Studies in
  Advanced Mathematics, vol.~39, Cambridge University Press, Cambridge, 1993.

\bibitem[Buc43]{buck}
R.~C. Buck, \emph{Partition of space}, Amer. Math. Monthly \textbf{50} (1943),
  541--544.

\bibitem[CD95]{CD}
R.~M. Charney and M.~W. Davis, \emph{Strict hyperbolization}, Topology
  \textbf{34} (1995), no.~2, 329--350.

\bibitem[CLS11]{CoxLittleSchenck}
D.~A. Cox, J.~B. Little, and H.~K. Schenck, \emph{Toric varieties}, Graduate
  Studies in Mathematics, vol. 124, American Mathematical Society, Providence,
  RI, 2011.

\bibitem[Dav08]{Davis}
M.~W. Davis, \emph{The geometry and topology of {C}oxeter groups}, London
  Mathematical Society Monographs Series, vol.~32, Princeton University Press,
  Princeton, NJ, 2008.

\bibitem[dLRS10]{dLRS}
J.~A. de~Loera, J.~Rambau, and F.~Santos, \emph{Triangulations}, Algorithms and
  Computation in Mathematics, vol.~25, Springer-Verlag, Berlin, 2010,
  Structures for algorithms and applications.

\bibitem[{Duv}96]{Duval}
A.~M. {Duval}, \emph{Algebraic shifting and sequentially cohen-macaulay
  simplicial complexes}, {Electron. J. Comb.} \textbf{3} (1996), no.~1,
  research paper r21, 14.

\bibitem[Eis95]{eis}
D.~Eisenbud, \emph{Commutative algebra with a view toward algebraic geometry},
  Graduate Texts in Mathematics, vol. 150, Springer-Verlag, New York, 1995.

\bibitem[EC95]{EC}
I.~Z. Emiris and J.~F. Canny, \emph{Efficient incremental algorithms for the
  sparse resultant and the mixed volume}, J. Symbolic Comput. \textbf{20}
  (1995), no.~2, 117--149.

\bibitem[FW07]{FukudaWeibel07}
K.~Fukuda and C.~Weibel, \emph{{$f$}-vectors of {M}inkowski additions of convex
  polytopes}, Discrete Comput. Geom. \textbf{37} (2007), no.~4, 503--516.

\bibitem[FW10]{FukudaWeibel10}
\bysame, \emph{A linear equation for {M}inkowski sums of polytopes relatively
  in general position}, European J. Combin. \textbf{31} (2010), no.~2,
  565--573.

\bibitem[FS97]{FS}
W.~Fulton and B.~Sturmfels, \emph{Intersection theory on toric varieties},
  Topology \textbf{36} (1997), no.~2, 335--353.

\bibitem[Geo08]{GR}
R.~Geoghegan, \emph{Topological methods in group theory}, Graduate Texts in
  Mathematics, vol. 243, Springer, New York, 2008.

\bibitem[Gr{\"a}87]{HGG}
H.-G. Gr{\"a}be, \emph{Generalized {D}ehn-{S}ommerville equations and an upper
  bound theorem}, Beitr\"age Algebra Geom. (1987), no.~25, 47--60.

\bibitem[GS93]{GS}
P.~Gritzmann and B.~Sturmfels, \emph{Minkowski addition of polytopes:
  computational complexity and applications to {G}r\"obner bases}, SIAM J.
  Discrete Math. \textbf{6} (1993), no.~2, 246--269.

\bibitem[Hib91]{Hibi}
T.~Hibi, \emph{Quotient algebras of {S}tanley-{R}eisner rings and local
  cohomology}, Journal of Algebra \textbf{140} (1991), no.~2, 336 -- 343.

\bibitem[Hoc77]{Hochster77}
M.~Hochster, \emph{Cohen-{M}acaulay rings, combinatorics, and simplicial
  complexes}, Ring theory, {II} ({P}roc. {S}econd {C}onf., {U}niv. {O}klahoma,
  {N}orman, {O}kla., 1975), Dekker, New York, 1977, pp.~171--223. Lecture Notes
  in Pure and Appl. Math., Vol. 26.

\bibitem[HR74]{H-R-Purity}
M.~Hochster and J.~L. Roberts, \emph{Rings of invariants of reductive groups
  acting on regular rings are {C}ohen-{M}acaulay}, Advances in Math.
  \textbf{13} (1974), 115--175.

\bibitem[Hov78]{KV}
A.~G. Hovanski{\u\i}, \emph{Newton polyhedra, and the genus of complete
  intersections}, Funktsional. Anal. i Prilozhen. \textbf{12} (1978), no.~1,
  51--61.

\bibitem[ILL{\etalchar{+}}07]{24loCo}
S.~B. Iyengar, G.~J. Leuschke, A.~Leykin, C.~Miller, E.~Miller, A.~K. Singh,
  and U.~Walther, \emph{Twenty-four hours of local cohomology}, Graduate
  Studies in Mathematics, vol.~87, American Mathematical Society, Providence,
  RI, 2007.

\bibitem[JMR83]{JMR}
W.~Julian, R.~Mines, and F.~Richman, \emph{Alexander duality}, Pacific J. Math.
  \textbf{106} (1983), no.~1, 115--127.

\bibitem[Kal91]{Kalaishifting}
G.~Kalai, \emph{The diameter of graphs of convex polytopes and {$f$}-vector
  theory}, Applied geometry and discrete mathematics, DIMACS Ser. Discrete
  Math. Theoret. Comput. Sci., vol.~4, Amer. Math. Soc., Providence, RI, 1991,
  pp.~387--411.

\bibitem[KKT15]{Karavelas12}
M.~I. Karavelas, C.~Konaxis, and E.~Tzanaki, \emph{{The maximum number of faces
  of the Minkowski sum of three convex polytopes}}, J. Comput. Geom. \textbf{6}
  (2015), no.~1, 21--74.

\bibitem[KT11]{Karavelas11}
M.~I. Karavelas and E.~Tzanaki, \emph{{The maximum number of faces of the
  Minkowski sum of two convex polytopes}}, October 2011, preprint, to appear in
  Discrete Comput. Geom., available at
  \url{dx.doi.org/10.1007/s00454-015-9726-6}.

\bibitem[KT15]{MR3392772}
\bysame, \emph{A geometric approach for the upper bound theorem for {M}inkowski
  sums of convex polytopes}, 31st {I}nternational {S}ymposium on
  {C}omputational {G}eometry, LIPIcs. Leibniz Int. Proc. Inform., vol.~34,
  Schloss Dagstuhl. Leibniz-Zent. Inform., Wadern, 2015, pp.~81--95.

\bibitem[Kat12]{katz}
E.~Katz, \emph{Tropical intersection theory from toric varieties}, Collect.
  Math. \textbf{63} (2012), no.~1, 29--44.

\bibitem[KK79]{KK79}
B.~Kind and P.~Kleinschmidt, \emph{Sch\"albare {C}ohen-{M}acauley-{K}omplexe
  und ihre {P}arametrisierung}, Math. Z. \textbf{167} (1979), no.~2, 173--179.

\bibitem[Kle64]{Klee64}
V.~Klee, \emph{On the number of vertices of a convex polytope}, Canad. J. Math.
  \textbf{16} (1964), 701--720.

\bibitem[Lat91]{latombe}
J.-C. Latombe, \emph{Robot motion planning}, vol.~25, Kluwer Academic
  Publishers, Boston, MA, 1991.

\bibitem[MPP11]{MPP}
B.~Matschke, J.~Pfeifle, and V.~Pilaud, \emph{Prodsimplicial-neighborly
  polytopes}, Discrete Comput. Geom. \textbf{46} (2011), no.~1, 100--131.

\bibitem[McM70]{mcmullen1970}
P.~McMullen, \emph{The maximum numbers of faces of a convex polytope},
  Mathematika \textbf{17} (1970), no.~02, 179--184.

\bibitem[McM04]{mcmullen04}
\bysame, \emph{Triangulations of simplicial polytopes}, Beitr\"age Algebra
  Geom. \textbf{45} (2004), no.~1, 37--46.

\bibitem[MW71]{McMullenWalkup}
P.~McMullen and D.~W. Walkup, \emph{A generalized lower-bound conjecture for
  simplicial polytopes}, Mathematika \textbf{18} (1971), 264--273.

\bibitem[MNS11]{MillerNovikSwartz11}
E.~Miller, I.~Novik, and E.~Swartz, \emph{Face rings of simplicial complexes
  with singularities}, Math. Ann. \textbf{351} (2011), no.~4, 857--875.

\bibitem[MS05]{MillerSturmfels05}
E.~Miller and B.~Sturmfels, \emph{Combinatorial commutative algebra}, Graduate
  Texts in Mathematics, vol. 227, Springer-Verlag, New York, 2005.

\bibitem[Min11]{minkowski}
H.~Minkowski, \emph{Theorie der konvexen k{\"o}rper, insbesondere
  {B}egr{\"u}ndung ihres {O}berfl{\"a}chenbegriffs}, Gesammelte Abhandlungen
  von Hermann Minkowski \textbf{2} (1911), 131--229.

\bibitem[Miy89]{miyazaki}
M.~Miyazaki, \emph{Characterizations of {B}uchsbaum complexes}, manuscripta
  mathematica \textbf{63} (1989), no.~2, 245--254.

\bibitem[Mot57]{Motzkin57}
T.~S. Motzkin, \emph{{Comonotone curves and polyhedra}}, Bull. Amer. Math. Soc.
  \textbf{63} (1957), 35.

\bibitem[MRTT53]{games}
T.~S. Motzkin, H.~Raiffa, G.~L. Thompson, and R.~M. Thrall, \emph{The double
  description method}, Contributions to the theory of games, vol. 2, Annals of
  Mathematics Studies, no. 28, Princeton University Press, Princeton, N. J.,
  1953, pp.~51--73.

\bibitem[Nov03]{Novik03}
I.~Novik, \emph{Remarks on the upper bound theorem}, J. Combin. Theory Ser. A
  \textbf{104} (2003), 201--206.

\bibitem[Nov05]{Novik05}
\bysame, \emph{On face numbers of manifolds with symmetry}, Adv. Math.
  \textbf{192} (2005), no.~1, 183--208.

\bibitem[NS09]{NovikSwartz09}
I.~{Novik} and E.~{Swartz}, \emph{Applications of {K}lee's
  {D}ehn--{S}ommerville relations}, {Discrete Comput. Geom.} \textbf{42}
  (2009), no.~2, 261--276.

\bibitem[NS12]{NovikSwartz12}
\bysame, \emph{Face numbers of pseudomanifolds with isolated singularities},
  Math. Scand. \textbf{110} (2012), no.~2, 198--222.

\bibitem[PS93]{PS}
P.~Pedersen and B.~Sturmfels, \emph{Product formulas for resultants and {C}how
  forms}, Math. Z. \textbf{214} (1993), no.~3, 377--396.

\bibitem[Rei76]{Reisner}
G.~A. Reisner, \emph{Cohen-{M}acaulay quotients of polynomial rings}, Advances
  in Math. \textbf{21} (1976), no.~1, 30--49.

\bibitem[RS72]{RourkeSanders}
C.~P. Rourke and B.~J. Sanderson, \emph{Introduction to {P}iecewise-{L}inear
  {T}opology}, Springer, New York, 1972, Ergebnisse Series vol.\ 69.

\bibitem[San09]{Sanyal09}
R.~Sanyal, \emph{Topological obstructions for vertex numbers of {M}inkowski
  sums}, J. Combin. Theory Ser. A \textbf{116} (2009), no.~1, 168--179.

\bibitem[Sch81]{Schenzel81}
P.~Schenzel, \emph{On the number of faces of simplicial complexes and the
  purity of {F}robenius}, Math. Z. \textbf{178} (1981), no.~1, 125--142.

\bibitem[Sch82]{Schenzel82}
\bysame, \emph{Dualisierende {K}omplexe in der lokalen {A}lgebra und
  {B}uchsbaum-{R}inge}, Lecture Notes in Mathematics, vol. 907,
  Springer-Verlag, Berlin, 1982, With an English summary.

\bibitem[Sch93]{Schneider93}
R.~Schneider, \emph{Convex bodies: the {B}runn-{M}inkowski theory},
  Encyclopedia of Mathematics and its Applications, vol.~44, Cambridge
  University Press, Cambridge, 1993.

\bibitem[Sta75]{Stanley75}
R.~P. Stanley, \emph{The upper bound conjecture and {C}ohen-{M}acaulay rings},
  Studies in Appl. Math. \textbf{54} (1975), no.~2, 135--142.

\bibitem[Sta87]{Stanley87}
\bysame, \emph{Generalized {$H$}-vectors, intersection cohomology of toric
  varieties, and related results}, Commutative algebra and combinatorics
  ({K}yoto, 1985), Adv. Stud. Pure Math., vol.~11, North-Holland, Amsterdam,
  1987, pp.~187--213.

\bibitem[Sta93]{Stanleymono}
\bysame, \emph{A monotonicity property of {$h$}-vectors and {$h^*$}-vectors},
  European J. Combin. \textbf{14} (1993), no.~3, 251--258.

\bibitem[Sta96]{Stanley96}
\bysame, \emph{Combinatorics and commutative algebra}, second ed., Progress in
  Mathematics, vol.~41, Birkh\"auser Boston Inc., Boston, MA, 1996.

\bibitem[ST10]{TS}
R.~Steffens and T.~Theobald, \emph{Combinatorics and genus of tropical
  intersections and {E}hrhart theory}, SIAM J. Discrete Math. \textbf{24}
  (2010), no.~1, 17--32.

\bibitem[Ste26]{steiner}
J.~Steiner, \emph{Einige {G}esetze {\"u}ber die {T}heilung der {E}bene und des
  {R}aumes}, Journal f{\"u}r die reine und angewandte Mathematik \textbf{1}
  (1826), 349--364.

\bibitem[Stu02]{Sturmfels02}
B.~Sturmfels, \emph{Solving systems of polynomial equations}, CBMS Regional
  Conference Series in Mathematics, vol.~97, Published for the Conference Board
  of the Mathematical Sciences, Washington, DC, 2002.

\bibitem[{Swa}05]{Swartz}
E.~{Swartz}, \emph{{Lower bounds for $h$-vectors of $k$-CM, independence, and
  broken circuit complexes.}}, {SIAM J. Discrete Math.} \textbf{18} (2005),
  no.~3, 647--661.

\bibitem[Wei12]{Weibel12}
C.~Weibel, \emph{Maximal f-vectors of {M}inkowski sums of large numbers of
  polytopes}, Discrete Comput. Geom. \textbf{47} (2012), no.~3, 519--537.

\bibitem[Zee66]{ZeemanBK}
E.~C. Zeeman, \emph{Seminar on combinatorial topology}, Institut des Hautes
  Etudes Scientifiques, Paris, 1966.

\bibitem[Zie95]{Z}
G.~M. Ziegler, \emph{Lectures on {P}olytopes}, Graduate Texts in Mathematics,
  vol. 152, Springer, New York, 1995, Revised edition, 1998; seventh updated
  printing 2007.

\end{thebibliography}
\end{document}